\documentclass[12pt]{amsart}
\usepackage{amssymb,amsmath,amsfonts,texdraw}
\usepackage[all]{xy}

\newtheorem{theorem}{Theorem}[section] 
\newtheorem{definition}[theorem]{Definition}
\newtheorem{proposition}[theorem]{Proposition}
\newtheorem{lemma}[theorem]{Lemma}
\newtheorem{corollary}[theorem]{Corollary}
\newtheorem{conjecture}[theorem]{Conjecture}

\theoremstyle{definition}

\newtheorem{example}[theorem]{Example}

\def\A{\mathbb{A}}
\def\C{\mathbb{C}}
\def\cf{\mathcal{F}}

\def\ci{\mathcal{I}}
\def\cA{\mathcal{A}}

\def\k{\ensuremath{\mathbf{k}}}
\def\G{\mathbb{G}}

\def\P{\mathbb{P}}
\def\Q{\mathbb{Q}}
\def\R{\mathbb{R}}
\def\U{\mathbb{U}}
\def\Z{\mathbb{Z}}
\def\O{\mathcal{O}}
\def\cB{\mathcal{B}}
\def\cI{\mathcal{I}}
\def\F{\mathbb{F}}
\def\cT{\mathcal{T}}
\def\K{\mathbb{K}}
\def\cR{\mathcal{R}}
\def\N{\mathbb{N}}
\def\B{\mathbb{B}}

\def\<{\ensuremath{\langle}}
\def\>{\ensuremath{\rangle}}

\def\excise#1{}

\DeclareMathOperator{\Hilb}{Hilb} 
\DeclareMathOperator{\Hom}{Hom}

\DeclareMathOperator{\PGL}{PGL}
\DeclareMathOperator{\Spec}{Spec}

\DeclareMathOperator{\Gr}{Gr}
\DeclareMathOperator{\Trunc}{Trunc}
\DeclareMathOperator{\codim}{codim}

\DeclareMathOperator{\Span}{Span}
\DeclareMathOperator{\Var}{Var}
\DeclareMathOperator{\Conv}{Conv}
\DeclareMathOperator{\Mat}{Mat}
\DeclareMathOperator{\charact}{char}
\DeclareMathOperator{\St}{St}
\DeclareMathOperator{\Gl}{Gl}
\DeclareMathOperator{\Cone}{Cone}
\DeclareMathOperator{\Fl}{Fl}
\DeclareMathOperator{\cl}{cl}

\begin{document}

\title{Matroid Theory for Algebraic Geometers}

\author[Katz]{Eric Katz}
\address{Eric Katz: Department of Combinatorics \& Optimization \\ University of Waterloo \\ Waterloo \\ ON \\ Canada N2L 3G1}
\email{eekatz@math.uwaterloo.ca}

\begin{abstract}
This article is a survey of matroid theory aimed at algebraic geometers.  Matroids are combinatorial abstractions of linear subspaces and hyperplane arrangements.  Not all matroids come from linear subspaces; those that do are said to be \emph{representable}.  Still, one may apply linear algebraic constructions to non-representable matroids.  There are a number of different definitions of matroids, a phenomenon known as \emph{cryptomorphism}.  In this survey, we begin by reviewing the classical definitions of matroids, develop operations in matroid theory,  summarize some results in representability, and construct polynomial invariants of matroids.  Afterwards, we focus on matroid polytopes, introduced by Gelfand-Goresky-MacPherson-Serganova, which give a cryptomorphic definition of matroids.  We explain certain locally closed subsets of the Grassmannian, thin Schubert cells, which are labeled by matroids, and which have applications to representability, moduli problems, and invariants of matroids following Fink-Speyer.  We explain how matroids can be thought of as cohomology classes in a particular toric variety, the \emph{permutohedral variety}, by means of Bergman fans, and apply this description to give an exposition of the proof of log-concavity of the characteristic polynomial of representable matroids due to the author with Huh.
\end{abstract}

\maketitle


\section{Introduction}

This survey is an introduction to matroids for algebraic geometry-minded readers.  Matroids are a combinatorial abstraction of linear subspaces of a vector space with distinguished basis or, equivalently, a set of labeled set of vectors in a vector space.  Alternatively, they are a generalization of graphs and are therefore amenable to a structure theory similar to that of graphs.  For algebraic geometers, they are a source of bizarre counterexamples in studying moduli spaces, a combinatorial way of labelling strata of a Grassmannian, and a testing ground for theorems about representability of cohomology classes. 

Matroids were introduced by Whitney \cite{Whitney35} as an an abstraction of linear independence.  If $\k$ is a field, one can study an $(n+1)$-tuple of vectors $(v_0,\dots,v_n)$ of $\k^{d+1}$ by defining a {\em rank function}
\[r:2^{\{0,\dots,n\}}\rightarrow \Z_{\geq 0}\]
by, for $S\subset \{0,\dots,n\}$, 
\[r(S)=\dim(\Span(\{v_i \mid i\in S\})).\]
This rank function satisfies certain natural properties, and one can consider rank functions that satisfy these same properties without necessarily coming from a set of vectors.  This rank function is what Whitney called a matroid.
Whitney noticed that there were matroids that did not come from a set of vectors over a particular field $\k$.  Such matroids are said to be {\em non-representable} over $\k$.  Matroids can also be obtained from graphs as a sort of combinatorial abstraction of the cycle space of a graph.  In fact, it is a piece of folk wisdom that any theorem about graph theory that makes no reference to vertices is a theorem in matroid theory.    The important structure theory of matroids that are representable over particular finite fields (or over all fields) was initiated by Tutte.  In fact, Tutte was able to demarcate the difference between graphs and matroids in a precise way.
 The enumerative theory of matroids as partially ordered sets was initiated by Birkhoff, Whitney, and Tutte and systematized and elaborated by Rota \cite{RotaFoundations}.  From this enumerative theory, there were associated polynomial invariants of matroids, among them the characteristic and Tutte polynomials. The theory of matroids was enlarged and formulated in more categorical terms by a number of researchers including Brylawski, Crapo, Higgs, and Rota \cite{CrapoRota}.  Matroids were found to have applications to combinatorial optimization problems by Edmonds \cite{Edmonds} who introduced a polytope encoding the structure of the matroid.

Throughout the development of the subject, many alternative formulations of matroids were found.  They were combinatorial abstractions of notions like the span of a subset of a set of vectors, independent sets of vectors, vectors forming a basis of the ambient space, or minimally dependent sets of vectors.  Each of these definitions made different structures of matroids more apparent.  The multitude of non-obviously equivalent definitions goes by the name of {\em cryptomorphism}.  Matroid theory is, in fact, sometimes forbidding to beginners because of the frequent switching between definitions.

The point of view of matroids in this survey is one initiated in the work of Gelfand-Goresky-MacPherson-Serganova, which related representable matroids to certain subvarieties of a Grassmannian.   One views the vector configuration spanning $\k^{n+1}$ as a surjective linear map
\[\k^{n+1}\rightarrow \k^{d+1},\ \ (x_0,\dots,x_n)\mapsto x_0v_0+\dots x_nv_n.\]
 By dualizing, one has an injective linear map 
 $(\k^{d+1})^*\rightarrow (\k^{n+1})^*$ where we consider the image of the map as a subspace $V\subseteq (\k^{n+1})^*\cong \k^{n+1}$.  Now, we may scale this subspace by the action of the algebraic torus $(\k^*)^{n+1}$ acting on the coordinates of $\k^{n+1}$.  The closure of the algebraic torus orbit containing $V$defines a subvariety of $\Gr(d+1,n+1)$.  The set of characters of the algebraic torus acting on $V$ leads to the notion of matroid polytopes and a new perspective on matroids.  It is this perspective that we will explore in this survey.  In particular, we will study how the matroid polytope perspective leads to a class of valuative invariants of matroids, how the subvarieties of $\Gr(d+1,n+1)$ that correspond to subspaces $V$ representing a particular matroid are interesting in their own right and give insight to representability, and finally how the study of an object called the Bergman fan parameterizing degenerations of the matroid sheds light on the enumerative theory of matroids.  In fact, we will use the Bergman fan to present a proof of a theorem of Huh and the author \cite{HuhKatz} of a certain set of inequalities among coefficients of the characteristic polynomial of matroids, called log-concavity, addressing part of a conjecture of Rota, Heron, and Welsh \cite{Rota}.

Another theme of this survey is cryptomorphism.  The new ways of thinking about matroids introduced by algebraic geometry have introduced two new definitions of matroids that are quite different from the other, more classical characterizations: matroid polytopes and Bergman fans.  The definition of matroids in terms of matroid polytopes comes out of the work of Gelfand-Goresky-MacPherson-Serganova.  The Bergman fan definition was motivated by valuation theory \cite{Bergman}, rephrased in terms of tropical geometry, and can be described here as studying matroids as  cohomology classes on a particular toric variety called the permutohedral variety.  In this survey, we advocate for the combinatorial study of a slight enlargement of the category of matroids, that of Minkowski weights on permutohedral varieties.  This enlargement allows a new operation called $(r_1,r_2)$-truncation introduced by Huh and the author which is essential to the proof of log-concavity of the characteristic polynomial.

We have picked topics to appeal to algebraic geometers.
We have put some emphasis on representability, which through Mn\"{e}v's theorem and Vakil's work on Murphy's Law plays a central role in constructing pathological examples in algebraic geometry.

This survey's somewhat bizarre approach and assumptions of background reflect how the author learned the subject.  We assume, for example, that the reader is familiar with toric varieties and $K$-theory but provide an introduction to M\"{o}bius inversion.  We include just enough of the highlights of the structure theory of matroids to give readers a sense of what is out there.  The literature on matroids is vast and the author's ignorance keeps him from saying more.  The more purely combinatorial research in matroid theory has a quite different flavor from our survey.  Also, there are a number of topics that would naturally fit into this survey that we had to neglect for lack of expertise.  Such topics include Coxeter matroids \cite{CoxeterMatroids}, oriented matroids \cite{Orientedmatroids}, matroids over a ring as defined by Fink and Moci \cite{FinkMoci}, hyperplane arrangements \cite{OrlikTerao}, and tropical linear subspaces \cite{SpeyerLinear} .
This survey is rather ahistorical.   We neglect nearly all the motivation coming from graph theory.

We make no claims towards originality in this survey.  The presentation of Huh-Katz's proof of log-concavity of the characteristic polynomial differs from that of the published paper \cite{HuhKatz} but is similar to the exposition in Huh's thesis \cite{Huhthesis}.

We would like to acknowledge Matthew Baker, Graham Denham, Michael Falk, Alex Fink, June Huh, Sam Payne, Margaret Readdy, Hal Schenck, and Frank Sottile for valuable conversations.  This survey arose from an expository talk given at the Simons Symposium on Non-Archimedean and Tropical Geometry.  We'd like to thank the organizers of ths symposium for their invitation and encouragement.  

There are a number of references that we can recommend enthusiastically and which were used extensively in the writing of this survey.  Oxley's textbook \cite{Oxley} is invaluable as a guide to the combinatorial theory.  Welsh's textbook \cite{Welshbook} is very broad and geometrically-oriented.  Wilson \cite{WilsonMonthly} gives a nice survey with many  examples.  Reiner's lectures \cite{Reinerlectures} explain the theory of  matroids and oriented matroids in parallel while also providing historical background.  The three Encyclopedia of Mathematics and its Applications volumes, \emph{Combinatorial Geometries} \cite{Combinatorialgeometries}, \emph{Theory of Matroids} \cite{TheoryofMatroids}, and \emph{Matroid Applications} \cite{Matroidapplications} are collections of valuable expository articles.  In particular, we found \cite{BrylawskiBook} very helpful in the writing of this survey.  Denham's survey on hyperplane arrangements \cite{Denham} is a useful reference for more advanced topics.

 \subsection{Notation}
 
 We will study algebraic varieties over a field $\k$.  We will use $\k$ to denote affine space over $\k$, $\A^1_\k$ and $\k^*$ to denote the multiplicative group over $\k$, $(\G_m)_\k$.  For a vector space $V$, $V^*$ will be the dual space.  Consequently, $(\k^n)^*$ will be a vector space and $(\k^*)^n$ will be a multiplicative group.   We will refer to such $(\k^*)^n$ as an algebraic torus.  Our conventions are geared towards working in projective space: matroids will usually be rank $d+1$ on a ground set $E=\{0,1,\dots,n\}$.

This survey is organized largely by the mathematical techniques employed.  The first six sections are largely combinatorial while the next six are increasingly algebraic.  Section 2 provides motivation for the definition of matroids which is given in section 3.  Section 4 provides examples while section 5 explains constructions in matroid theory with an emphasis on what the constructions mean for representable matroids viewed as vector configurations, projective subspaces, or hyperplane arrangements.  Section 6 discusses representability of matroids.  Section 7 introduces polynomial invariants of matroids, in particular the Tutte and characteristic polynomial.  Section 8 reviews the matroid polytope construction of Gelfand-Goresky-MacPherson-Serganova and the valuative invariants that it makes possible.  Section 9 reviews constructions involving the Grassmannian, describing the relationships between Pl\"{u}cker coordinates and the matroid axioms and between the matroid polytope and torus orbits, and then it discusses realization spaces and finally, the $K$-theoretic matroid invariants of Fink and Speyer.  Section 10 is a brief interlude reviewing toric varieties.  Section 11 introduces Bergman fans and shows that they are Minkowski weights.  Section 12 gives a proof of log-concavity  of the characteristic polynomial through intersection theory on toric varieties.  Section 13 points out some future directions.
 
\section{Matroids as Combinatorial Abstractions}

A matroid is a combinatorial object that captures properties of vector configurations or equivalently, hyperplane arrangements.  We will informally discuss different ways of thinking about vector configurations as motivation for the rest of the survey.  This section is provided solely as motivation and will not introduce any definitions needed for the rest of the paper.

Let $\k$ be a field, and let $v_0,v_1,\dots,v_n$ be vectors in $\k^{d+1}$ that span $\k^{d+1}$.  We can study the dimension of the span of a subset of these vectors.  Specifically, for $S\subset\{0,1,\dots,n\}$, we set
\[V_S=\Span(\{v_i\mid i\in S\}),\]
 and we define a \emph{rank function} $r:2^E\rightarrow\Z$, by
for $S\subseteq \{0,1,\dots,n+1\}$,
\[r(S)=\dim(V_S).\]
There are some obvious properties that $r$ satisfies: we must have $0\leq r(S)\leq |S|$; $r$ must be non-decreasing on subsets (so $S\subseteq U$ implies $r(S)\leq r(U)$); and $r(\{0,1,\dots,n\})=\dim(\k^{d+1})=d+1$.  There is a less obvious property: because for $S,U\subseteq \{0,1,\dots,n\}$, $V_{S\cap U}\subseteq V_S\cap V_U$, we must have
\[r(S\cap U)\leq \dim(V_S\cap V_U)=\dim(V_S)+\dim(V_U)-\dim(V_{S\cup U})=r(S)+r(U)-r(S\cup U).\]
We can take this one step further and study all rank functions that satisfy these properties.  Such a rank function, we will call a \emph{matroid}.  Not all matroids will come from vector configurations.  Those that do will be said to be \emph{representable}.  Remarkably, a number of geometric constructions will work for matroids regardless of their representability.

Instead of studying rank functions, we can study certain collections of subsets of $\{0,1,\dots,n\}$ that  capture the same combinatorial data.  We can study \emph{bases} which are $(d+1)$-element subsets of $\{0,1,\dots,n\}$  corresponding to subsets of $\{v_0,v_1,\dots,v_n\}$ which span $\k^{d+1}$.  Or we can study \emph{independent sets} which are subsets of $\{0,1,\dots,n\}$ corresponding to linearly independent sets of vectors.  We can study \emph{circuits} which are minimal linearly dependent sets of vectors.  Or we can study \emph{flats} which correspond to subspaces spanned by some subset $\{v_0,v_1,\dots,v_n\}$.  Each of these collections of subsets can be used to give a definition of a matroid.

Alternatively, we can consider linear subspaces instead of vector configurations.  Let $V\subseteq \k^{n+1}$ be a $(d+1)$-dimensional subspace that is not contained in any coordinate hyperplane. If $e_0,e_1,\dots,e_n$ is the standard basis of $\k^{n+1}$, then the dual basis $e_0^*,e_1^*,\dots,e_n^*$ induce linear forms on $V$.  This gives a vector configuration in $V^*$.  We can projective $V$ to obtain a projective subspace $\P(V)\subset \P^n$.  We can rephrase the data of the rank function in terms of a hyperplane arrangement of $\P(V)$.  Indeed, if the $e_i^*$'s are all non-zero,
each linear form $e_i^*$ vanishes on a hyperplane, $H_i$ on $\P(V)$, giving a hyperplane arrangement.  Note that $H_i$ is the intersection of $\P(V)$ with the coordinate hyperplane in $\P^n$ cut out by $X_i=0$ where $X_i$ is a homogeneous coordinate.  Now, we can define the rank function by, for $S\subset\{0,1,\dots,n\}$,
\[r(S)=\codim\left((\bigcap_{i\in S} H_i)\subset\P(V)\right).\]
There are interpretations of bases, independent sets, circuits, and flats in this language as well.

There are a number of invariants of matroids that correspond to remembering the matroids up to a certain equivalence class, analogous to passing to a Grothendieck ring.  One invariant, the Tutte polynomial can be related to the class of the matroid in a Grothendieck ring whose equivalence relation comes from deletion and contraction operations.  In the hyperplane arrangement language, deletion corresponds to forgetting a hyperplane on the projective subspace, and contraction corresponds to a particular projection onto a lower-dimensional subspace.

There are geometric constructions that can be performed on the linear subspace $V$.  These constructions can be studied not just for linear subspaces but for matroids.   Some of these constructions, we shall see, can be used to give new combinatorial abstractions of linear subspaces and therefore, new definitions of matroids.

One construction involves the Grassmannian.  A linear subspace $V$ corresponds to a point in the Grassmannian $\Gr(d+1,n+1)$ parameterizing $(d+1)$-dimensional subspaces of $\k^{n+1}$.  Certain information about this point is equivalent to the data of the matroid.  To speak of it, we have to study the geometry of the Grassmannian.  The Grassmannian has a Pl\"{u}cker embedding,
\[i:\Gr(d+1,n+1)\hookrightarrow \P^N\]
where $N=\binom{n+1}{d+1}-1$.  The homogeneous coordinates $p_B$ on $\P^N$ are labelled by $B\subset\{0,1,\dots,n\}$ with $|B|=d+1$ and are called Pl\"{u}cker coordinates.  The data of the matroid is captured by which Pl\"{u}cker coordinates are nonzero.  Alternatively, the data can be phrased in terms of a certain group acting on the Grassmannian.  Let $T=(\k^*)^{n+1}$ act on $\k^{n+1}$ by dilating the coordinates.  This induces an action on the Grassmannian by, for $t\in T$, taking $V\in \Gr(d+1,n+1)$ to $t\cdot V=\{tv \mid v\in V\}$.  Note that the diagonal torus of $T$ acts trivially on $\Gr(d+1,n+1)$.  This group action extends to the ambient $\P^N$.  Given $V\in Gr(d+1,n+1)\subset \P^N$, we can lift $V$ to a point $\tilde{V}\in\k^{N+1}$ and ask what are the characters of $T$ of the smallest sub-representation of $T$ containing $\tilde{V}$.  This set of characters captures exactly the data of the matroid.  Alternatively, we can rephrase this data in terms of group orbits.  The closure of the $T$-orbit containing $V$, $\overline{T\cdot V}\subset\P^N$ is a polarized projective toric variety.  By well-known results in toric geometry, this toric variety corresponds to a polytope.  The data of this polytope also corresponds to the matroid.  Moreover, one can study polytopes that arise in this fashion combinatorially and even associate them to non-representable matroids.  These matroids polytopes can be used to produce an interesting class of invariants of matroids, called valuative invariants.  These invariants are those that are well-behaved under subdivision of the matroid polytope into smaller matroid polytopes.  The $K$-theory class of the structure sheaf of the closure of the torus orbit, $\O_{\overline{T\cdot V}}\in K_0(\Gr(d+1,n+1))$ is a valuative invariant of the matroid introduced by Speyer.  By a combinatorial description of $K$-theory of the Grassmannian, the invariant can be extended to describe non-representable matroids. 

Given a matroid $M$, one can study the set of points on the Grassmannian that have $M$ as their matroid.  This locally closed subset of $\Gr(d+1,n+1)$ is called a \emph{thin Schubert cell}.  Such sets have arbitrarily bad singularities (up to an equivalence relation) and can be used to construct other pathological moduli spaces.  The pathologies of thin Schubert cells are responsible for a number of the difficulties in understanding representability of matroids.

Given a subspace $\P(V)\subset\P^n$, one can blow up the ambient $\P^n$ to understand $\P(V)$ as a homology class.  Homogeneous coordinates on $\P^n$ provide a number of distinguished subspaces.  In fact, we can consider all subspaces that occur as $\bigcap_{i\in S} H_i$ for all proper subsets $S\subset \{0,1,\dots,n\}$. If we intersect $n$ distinct coordinate hyperplanes of $\P^n$, we get a point all but one of whose coordinates are $0$.  There are $n+1$ such points.  If we intersect $n-1$ distinct coordinate hyperplanes, we get a coordinate line between two of those points.  If we intersect $n-2$ distinct coordinate hyperplanes, we get a coordinate plane containing three of those points, and so on.  We can produce a new variety, called the permutohedral variety, $X$ by first blowing up the $n+1$ points, then blowing up the proper transform of the coordinate lines, then blowing up the proper transform of the coordinate planes, and so on.  The proper transform $\widetilde{\P(V)}$ of $\P(V)$ is an iterated blow-up of intersections of hyperplanes from the induced arrangement.  The homology class of $\widetilde{\P(V)}$ in $X$ depends only on the matroid of $V$.  It has been studied as the Bergman fan by Sturmfels and Ardila-Klivans.   Moreover, it can be defined for non-representable matroids.  The characteristic polynomial, which is a certain specialization of the Tutte polynomial can be phrased as the answer to an intersection theory problem on the Bergman fan.  In the representable case, one may apply intersection-theoretic inequalities derived from the Hodge index theorem to prove inequalities between the coefficients of the characteristic polynomial, resolving part of the Rota-Heron-Welsh conjecture.

We will discuss all this and more below.

\section{Matroids}

There are many definitions of a matroid.  The equivalence of these definitions go by the buzzword of {\em cryptomorphism}.  They are all structures on a finite set $E$ which will be called the ground set.

The rank formulation of matroids is one that will be most useful to us in the sequel:

\begin{definition} A matroid on $E$ of rank $d+1$ is a function
\[r:2^E\rightarrow\Z\]
satisfying   
\begin{enumerate}
\item $0\leq r(S)\leq |S|$,
\item $S\subseteq U$ implies $r(S)\leq r(U)$,
\item \label{i:subadd} $r(S\cup U)+r(S\cap U)\leq r(S)+r(U)$, and 
\item $r(\{0,\dots,n\})=d+1.$
\end{enumerate}
\end{definition}

\begin{definition} Two matroids $M_1,M_2$ with $M_i$ on $E_i$ with rank function $r_i$ are said to be {\em isomorphic} if there is a bijection $f:E_1\rightarrow E_2$ such that for any $S\subseteq E_1$,$r_2(f(S))=r_1(S)$.
\end{definition}

The definition of matroids makes sense from the point of view of vector configuration in a vector space.  Let $E=\{0,1,\dots,n\}$.
Let $\k$ be a field.  Consider the vector space $\k^{d+1}$ together with $n+1$ vectors $v_0,\dots,v_n\in\k^{d+1}$ spanning $\k^{d+1}$.  For $S\subseteq E$, set
\[r(S)=\dim(\Span(\{v_i \mid i\in S\})).\]  Then $r$ is a matroid.  If we write 
\[V_S=\Span(\{v_i \mid i\in S\}),\]
then item (\ref{i:subadd}) is equivalent to
\[\dim(V_{S\cap U})\leq \dim(V_S\cap V_U).\]
Note that this inequality may be strict because there may be no subset of $E$ exactly spanning $V_S\cap V_U$.
Several of the matroid axioms are simple-minded and obvious while one is non-trivial.   This is very much in keeping with the flavor of the subject.

\begin{definition} A matroid is said to be {\em representable over $\k$} if it is isomorphic to a matroid arising from a vector configuration in a vector space over $\k$.  A matroid is said to be {\em representable} if it is representable over some field.  A matroid is said to be {\em regular} if it is representable over every field.
\end{definition}

Regular matroids are much studied in combinatorics.  They have a characterization due to Tutte (see \cite{Oxley} for details).  Given a representable matroid, we may form the matrix whose columns are the coordinates of the vectors in the vector configuration.  A matrix is said to be totally unimodular if each square submatrix has determinant $0$,$1$, or $-1$.  It is a theorem of Tutte that regular matroids are those representable by totally unimodular matrices with real entries \cite{Oxley}.

Representable matroids are an important class of matroids but are not all of them.  In fact, it is conjectured that they are asymptotically sparse among matroids.  We will discuss non-representable matroids at length in this survey.

Instead of considering a rank function, we may consider instead the set of {\em flats} of the matroids.

\begin{definition} A {\em flat} of $r$ is a subset $S\subseteq E$ such that for any $j\in E$, $j\not\in S$, $r(S\cup \{j\})>r(S)$.
\end{definition}

We think of flats as the linear subspaces of $\k^{n+1}$ spanned by vectors labeled by a subset of $E$.
We may also axiomatize matroids as a set of flats.

\begin{definition} \label{d:flats} A matroid is a collection of subsets $\cf$ of a set $E$ that satisfy the following conditions
\begin{enumerate}
\item $E\in\cf$,
\item if $F_1,F_2\in\cf$ then $F_1\cap F_2\in\cf$, and
\item if $F\in\cf$ and $\{F_1,F_2,\dots,F_k\}$ is the set of minimal members of $\cf$ properly containing $F$ then the sets $F_1\setminus F,F_2\setminus F,\dots, F_k\setminus F$ partition $E\setminus F$.  \label{iflat:3}
\end{enumerate}
\end{definition}

Note that axiom \eqref{iflat:3} implies that for any flat $F$ and $j\not\in F$, there is a unique flat $F'$ containing $F\cup \{j\}$ that does not properly contain any flat containing $F$.  We can also encode the data of the flats in terms of a {\em closure operation} where the closure of a set $S\subseteq E$, $\operatorname{cl}(S)$ is the intersection of the flats containing $S$.  
The set of flats form a lattice which is a poset equipped with operations that abstract intersection and union.  The lattice of flats of $M$ is denoted by $L(M)$.  We will let $\hat{0}$ be the minimal flat.
Given a collection of flats $\cf$, we may recover the rank function of a set $S\subseteq E$ by setting it to be the length of the longest chain of non-trivial flats properly contained in $\operatorname{cl}(S)$.

We can also axiomatize matroids in terms of their bases.  A basis for a vector configuration labeled by $E$ is a subset $B\subseteq E$ such that $\{v_i\mid i\in B\}$ is a basis for $\k^{d+1}$. In terms of the rank function, a basis is a $(d+1)$-element set $B\subseteq E$ with $r(B)=d+1$.
 
\begin{definition} A matroid is a collection $\cB$ of subsets of $E$ such that
\begin{enumerate}
\item $\cB$ is nonempty and
\item If $B_1,B_2\in \cB$ and $i\in B_1\setminus B_2$ then there is an element $j$ of $B_2\setminus B_1$ such that $(B_1\setminus i)\cup \{j\}\in \cB$.
\end{enumerate}
\end{definition}

The second axiom is called {\em basis exchange}.  It is a classical property of pairs of bases of a vector space and is due to Steinitz.  By applying it repeatedly, we may show that all bases have the same number of elements.  This is the rank of the matroid.
It is straightforward to go from a rank function to a collection of bases and vice versa.

Another axiomatization comes from the set of independent subsets which should be thought of subsets of $E$ labelling linearly independent subsets.  These would be subets $I\subset E$ such that $r(I)=|I|$.

\begin{definition} A matroid is a collection of subsets $\cI$ of $E$ such that
\begin{enumerate}
\item $\cI$ is nonempty, 
\item Every subset of a member of $\cI$ is a member of $\cI$, and
\item If $X$ and $Y$ are in $\cI$ and $|X|=|Y|+1$, then there is an element $x\in X\setminus Y$ such that $Y\cup  \{x\}$ is in $\cI$.  
\end{enumerate}
\end{definition}

\begin{definition} A {\em loop} of a matroid is an element $i\in E$ with $r(\{i\})=0$.  A {\em pair of parallel points} $(i,j)$ of a matroid are elements $i,j\in E$ such that $r(\{i\})=r(\{j\})=r(\{i,j\})=1$.  A matroid is said to be {\em simple} if it has neither loops nor parallel points.
\end{definition}

For vector configurations, a loop corresponds to the zero vector while parallel points correspond to a pair of parallel vectors.

\begin{definition} A {\em coloop} of a matroid is an element $i\in E$ that belongs to every basis.    
\end{definition}

In terms of vector configurations, a coloop corresponds to a vector not in the span of the other vectors.

A \emph{circuit} in a matroid is a minimal subset of $E$ that is not contained in a basis.  For a set of vectors, this should be thought of as a subset $C\subseteq E$ such that the vectors labeled by $C$ are linearly dependent but for any $i\in C$, $C\setminus \{i\}$ is linearly independent.  Circuits can be axiomatized to give another definition of matroids.

\section{Examples}

In this section, we explore difference classes of matroids arising in geometry, graph theory, and optimization.

\begin{example}
The uniform matroid $U_{d+1,n+1}$ of rank $d+1$ on $n+1$ elements is defined for $E=\{0,1,\dots,n\}$ by a rank function $r:2^E\rightarrow \Z_{\geq 0}$ given by
\[r(S)=\min(|S|,d+1).\]
It corresponds to a vector configuration $v_0,v_1,\dots,v_n$ given by $n+1$ generically chosen vectors in a $(d+1)$-dimensional vector space.  Any set of $d+1$ vectors is a basis.  
Note that if the field $\k$ does not have enough elements, then the vectors cannot be chosen generically.  For example, because $\F_2^2$ has only three non-zero elements, $U_{2,4}$ does not arise as a vector configuration over $\F_2$.  Equivalently, it is not representable over $\F_2$.  

The matroids on a singleton set will be important below.  The matroid $U_{0,1}$ is represented by a vector configuration consisting of $0\in\k$.  The single element of the ground set of $U_{0,1}$ is a loop.  On the other hand, $U_{1,1}$ is represented by any non-zero vector in $\k$.  The single element of the ground set of $U_{1,1}$ is a coloop.
\end{example}

\begin{example} \label{e:vector} Let $x_0,x_1,\dots,x_n$ be points in $\P_{\k}^d$ not contained in any proper projective subspace.  Pick a vector $v_i\in\k^{n+1}\setminus \{0\}$ on the line described by $x_i$.  Then $v_0,v_1,\dots,v_n$ gives a vector configuration in $\k^{d+1}$ and therefore a matroid on $\{0,1,\dots,n\}$ of rank $d+1$.  Specifically, the rank function, for $S\subset E=\{0,1,\dots,n\}$ is given by
\[r(S)=\dim(\Span(v_i \mid i\in S)).\]
This vector configuration can be thought of as a surjective map:
\begin{eqnarray*}
\k^{n+1}&\to&\k^{d+1}\\
(t_0,\ldots,t_n)&\mapsto&t_0v_0+\ldots+t_nv_n.
\end{eqnarray*}
Alternatively, we can dualize this map to get an injective map $(\k^{d+1})^*\rightarrow (\k^{n+1})^*$ whose image is a subspace $V$.
Note here that every non-empty set has positive rank so the matroid has no loops.  Coloops are elements $i$ such that the minimal projective subspace containing $\{x_0,x_1,\dots,x_n\}\setminus\{x_i\}$ is of positive codimension.  Flats correspond to minimal projective subspaces containing some subset of $\{x_0,x_1,\dots,x_n\}$.  Bases are $(d+1)$-element subsets of $\{x_0,x_1,\dots,x_n\}$ that are not contained a proper projective subspace of $\P^d$.

If $x_0,x_1,\dots,x_n$ are contained in a proper projective subspace, we may replace the ambient subspace by the minimal projective subspace containing $x_0,x_1,\dots,x_n$ in order to define a matroid.
\end{example}

\begin{example} Let $V\subseteq \k^{n+1}$ be an $(d+1)$-dimensional subspace.   Let $e_0,e_1,\dots,e_n$ be a basis for $\k^{n+1}$.  The inclusion $i:V\hookrightarrow \k^{n+1}$ induces a surjection $i^*:(\k^{n+1})^*\rightarrow V^*$.  The image of the dual basis, $\{i^*e_0^*,i^*e_1^*,\dots,i^*e_n^*\}$ gives a vector configuration in $V^*$.  We can take its matroid.

Let us use the basis to put coordinates $(x_0,x_1,\dots,x_n)$ on $\k^{n+1}$.  We can define subspaces of $V$ as follows: for $I\subseteq E$, set
\[V_I=\{x\in V\mid x_i=0\ \text{for all}\ i\in I\}.\]
Then $r(I)=\codim(\dim V_I\subset V)$.
A flat in this case is a subset $F\subset E$ such that for any $G\supset F$, $V_G\subsetneq V_F$.  The lattice of flats is exactly the lattice of subsets of $V$ of the form $V_S$, ordered under reverse inclusion.  An element $j$ is a loop if and only if $V$ is contained in the coordinate hyperplane $x_j=0$.  A basis is a subset $B\subseteq \{0,1,\dots,n\}$ such that $V_B=\{0\}$.
An element $j$ is a coloop if and only if $V$ contains the basis vector $e_j$.

If we replace $V$ and $\k^{n+1}$ by their projectivizations, and $V$ is not contained in any coordinate hyperplane, then the definition still makes sense.  In this case we consider a $r$-dimensional projective subspace $\P(V)$ in $\P^n$.  We consider $\P(V_i)$ as an arrangement of hyperplanes on $\P^n$ \cite{OrlikTerao}.  In this case, the subsets $P(V_F)$ as $F$ ranges over all flats corresponds to the different possible intersections of hyperplanes (including the empty set).
\end{example}

\begin{example} \label{e:quotient} Given $V\subset \k^{n+1}$ as above, we may define a matroid by considering the quotient
\[\pi:\k^{n+1}\rightarrow \k^{n+1}/V.\]  For $I\subset E$, let $\k^I\subset \k^{n+1}$ be the subspace given by
\[\k^I=\Span(\{e_i\mid i\in I\}).\]
We set 
\[r(I)=\dim(\pi(\k^I)).\]
This is the matroid given by the vector configuration $\{\pi(e_0),\pi(e_1),\dots,\pi(e_n)\}$.

This example is related to the previous one by matroid duality which we will investigate in Subsection \ref{ss:duality}.
\end{example}

\begin{figure}
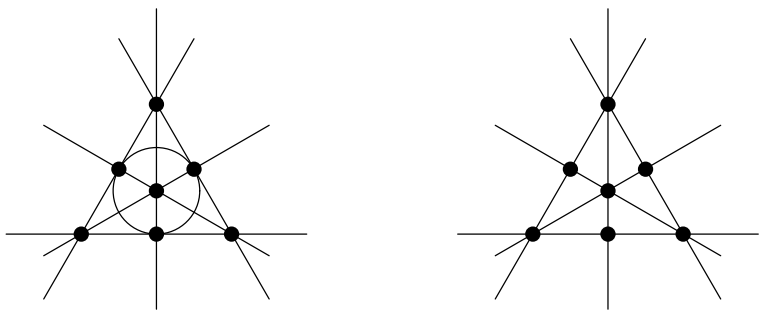
 \label{f:fano}
\begin{center}
\begin{texdraw}
       \drawdim cm  \relunitscale 0.5
       \linewd 0.03
        \move(-1 1) \lvec (7 1)
        \move(3 -1)\lvec(3 7)
        \move(0 -.73)\lvec(4 6.2)
        \move(6 -.73)\lvec(2 6.2)
        \move(0 .422)\lvec(6 3.89)
        \move(6 .422)\lvec(0 3.89)
        \move(3 2.15)  \lcir r:1.15
        \move(1 1) \fcir f:0 r:0.2
       \move(3 1) \fcir f:0 r:0.2
       \move(5 1) \fcir f:0 r:0.2
       \move(3 4.46) \fcir f:0 r:0.2
      \move(4 2.73) \fcir f:0 r:0.2 
      \move(2 2.73) \fcir f:0 r:0.2 
       \move(3 2.15) \fcir f:0 r:0.2

        \move(11 1) \lvec (19 1)
        \move(15 -1)\lvec(15 7)
        \move(12 -.73)\lvec(16 6.2)
        \move(18 -.73)\lvec(14 6.2)
        \move(12 .422)\lvec(18 3.89)
        \move(18 .422)\lvec(12 3.89)
        \move(13 1) \fcir f:0 r:0.2
       \move(15 1) \fcir f:0 r:0.2
       \move(17 1) \fcir f:0 r:0.2
       \move(15 4.46) \fcir f:0 r:0.2
      \move(16 2.73) \fcir f:0 r:0.2 
      \move(14 2.73) \fcir f:0 r:0.2 
       \move(15 2.15) \fcir f:0 r:0.2

\end{texdraw}
\end{center}
\caption{The Fano and non-Fano matroids}
\end{figure}

\begin{example}
One can draw simple matroids as  as point configurations.  We imagine the points as lying in some projective space and if the matroid were representable, they would give a vector configuration as in Example~\ref{e:vector}.  For example, if we have a rank $3$ matroid, we view the points as spanning a projective plane. We specify the rank $2$ flats that contain more than two points by drawing lines containing points.  These lines together with all lines between pairs of points are exactly the rank $2$ flats.  Higher rank matroids can be described by point configurations in higher dimensional spaces where we specify the $k$-flats that contain more points than what is predicted by rank considerations (e.g. for $3$-flats which correspond to planes, we would show the planes that are not merely those containing $3$ non-collinear points or a line and a non-incident point).

The Fano and non-Fano matroids denoted by $F_7$ and $F_7^-$, respectively, are pictured in  Figure 1.  The Fano matroid is a rank $3$ matroid consisting of $7$ points together with $7$ lines passing through particular triples of points.  A line through three of those points is drawn as a circle.  The Fano matroid is representable exactly over fields of characteristic $2$.  In fact, it is the set of all points and lines in the projective plane over $\F_2$, $\P^2_{\F_2}$.
The non-Fano matroid is given by the same configuration but with the center line removed.  We see that as meaning that those three points on that line are no longer collinear.  This matroid is representable exactly over fields whose characteristic is different from $2$.
\end{example}

\begin{figure}
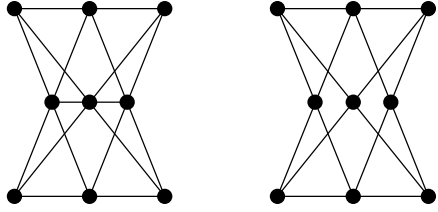
 \label{f:pappus}
\begin{center}
\begin{texdraw}
       \drawdim cm  \relunitscale 0.5
       \linewd 0.03
        \move(1 0) \lvec (5 0)
        \move(1 5)\lvec(5 5)
        \move(1 0)\lvec(3 5)
        \move(1 0)\lvec(5 5)
        \move(3 0)\lvec(5 5)
        \move(3 0)\lvec(1 5)
       \move(5 0)\lvec(1 5)
        \move(5 0)\lvec(3 5)
        \move(2 2.5)\lvec(4 2.5)
        \move(1 0) \fcir f:0 r:0.2 
        \move(3 0) \fcir f:0 r:0.2 
        \move(5 0) \fcir f:0 r:0.2 
        \move(1 5) \fcir f:0 r:0.2 
        \move(3 5) \fcir f:0 r:0.2 
        \move(5 5) \fcir f:0 r:0.2 
        \move(2 2.5) \fcir f:0 r:0.2 
        \move(3 2.5) \fcir f:0 r:0.2 
        \move(4 2.5) \fcir f:0 r:0.2 
        
        \move(8 0) \lvec (12 0)
        \move(8 5)\lvec(12 5)
        \move(8 0)\lvec(10 5)
        \move(8 0)\lvec(12 5)
        \move(10 0)\lvec(12 5)
        \move(10 0)\lvec(8 5)
       \move(12 0)\lvec(8 5)
        \move(12 0)\lvec(10 5)
        \move(8 0) \fcir f:0 r:0.2 
        \move(10 0) \fcir f:0 r:0.2 
        \move(12 0) \fcir f:0 r:0.2 
        \move(8 5) \fcir f:0 r:0.2 
        \move(10 5) \fcir f:0 r:0.2 
        \move(12 5) \fcir f:0 r:0.2 
        \move(9 2.5) \fcir f:0 r:0.2 
        \move(10 2.5) \fcir f:0 r:0.2 
        \move(11 2.5) \fcir f:0 r:0.2 
\end{texdraw}
\end{center}
\caption{The Pappus and non-Pappus matroids}
\end{figure}

\begin{example}
The Pappus and non-Pappus matroids are pictured in Figure 2.  The non-Pappus matroid is obtained from the Pappus matroid by mandating that the three points in the middle not be collinear.  
This is in violation of Pappus's theorem which is a theorem of projective geometry.   Therefore, the non-Pappus matroid is not representable over any field.
\end{example}

\begin{example}
Let $G$ be a graph  with $n+1$ edges labeled by $E=\{0,1,\dots,n\}$.  We can define the graphic matroid $M(G)$ to be the matroid whose set of bases are $G$'s spanning forests.  The flats of this matroid are the set of edges $F$ such that $F$ contains any edge whose endpoints are connected by a path of edges in $F$.  Note that the loops of this matroid are the loops of the graphs while the coloops are the bridges.  In fact, coloops are sometimes called bridges or isthmuses in the literature.

This matroid comes from a vector configuration.  Pick a direction for each edge.  Let $C_1(G,\k)$ be the vector space of simplicial $1$-chains on $G$ considered as a $1$-dimensional simplicial complex.  There is a basis of $C_1(G,\k)$ given by the edges $e_0,\dots,e_n$ with their given orientation.  Let $\partial:C_1(G,\k)\rightarrow C_0(G,\k)$ be the differential.  Then $M(G)$ is given by the vector configuration $\partial(e_0),\dots,\partial(e_n)$.  Here, the rank of the matroid is  
\[\dim(\partial C_1(G,\k))=\dim C_0(G,\k)-\dim H_0(G,\k)=|V(G)|-\kappa(G)\]
where $\kappa(G)$ is the number of connected components of $G$.

By Whitney's $2$-isomorphism theorem \cite[Thm 5.3.1]{Oxley}, one can reconstruct a connected graph $G$ from $M(G)$ up to two moves, vertex cleaving and Whitney twists.  If $G$ is $3$-connected, it can be uniquely reconstructed from $M(G)$.  In fact, matroids can be considered to be generalizations of graphs.  Tutte \cite{TutteKings} stated, ``If a theorem about graphs can be expressed in terms of edges and circuits only it probably exemplifies a more general theorem about matroids.''

As a special case, consider $K_{n+1}$, the complete graph on $n+1$ vertices.  Let us denote its vertices by $w_0,\dots,w_n$.  The edges are denoted by $e_{ij}$ for $0\leq i<j\leq n$.  The differential is given by $\partial e_{ij}=w_i-w_j$.  The associated subspace as in Example~\ref{e:vector} is $dC^0(G,\k)\subset C^1(G,\k)$.  We can put coordinates $y_0,\dots,y_n$ on $C^0(G,\k)$ by taking as a basis the characteristic functions of vertices $\delta_{w_0},\dots,\delta_{w_n}$.  Therefore, 
\[dC_0(G,\k)\cong \k^{n+1}/\k\]
where we quotient by the diagonal line.  The  hyperplane arrangement induced by the coordinate subspaes of $C^1(G,\k)$ is the \emph{braid arrangement}
\[\{y_i-y_j=0\mid 0\leq i<j\leq n\}.\]
\end{example}

\begin{example}
Let $G$ be a graph with $n+1$ edges labeled by $E=\{0,1,\dots,n\}$.  We can also define another matroid, the cographic matroid $M^*(G)$ of $G$.  The bases of $M^*(G)$ are complements of the spanning forests of $G$.

The cographic matroid of a graph also comes from a vector configuration.  We pick a direction for each edge as before.  Let $C^1(G,\k)$ be the $1$-cochains of $G$. It has a basis given by $\delta_{e_0},\delta_{e_1},\dots,\delta_{e_n}$, the characteristic function of each edge with given orientation
 Let $d:C^0(G,\k)\rightarrow C^1(G,\k)$ be the differential.  Let $V=C^1(G,\k)/dC^0(G,\k)$.  The matroid is given by the image of $\delta_{e_0},\delta_{e_1},\dots,\delta_{e_n}$ in $V$.
\end{example}

Matroids that are isomorphic to $M(G)$ for some graph $G$ are said to be graphic.  Cographic matroids are defined analogously.  Observe that because cycle and cocycle spaces can be defined over any field, graphic and cographic matroids are regular.
The uniform matroid $U_{2,4}$, because it is not regular, is neither graphic nor cographic.  There are examples of regular matroids that are neither graphic nor cographic.

\begin{example} Transversal matroids arise in combinatorial optimization.  Let $E=\{0,\dots,n\}$.  Let $A_1,\dots,A_m$ be subsets of $E$.  A {\em partial transversal} is a subset $I\subseteq E$ such that there exists an injective $\phi:I\rightarrow \{1,\dots,m\}$ such that $i\in A_{\phi(i)}$.  A transversal is a partial transerval of size $m$. We can view elements of $E$ as people and elements of $\{1,\dots,m\}$ as jobs where $A_i$ is the set of people qualified to do job $i$.  A partial transversal is a set of people who can each be assigned to a different job.  The transversal matroid is defined to be the matroid on $E$ whose independent sets are the sets of partial transversals.

An important generalization of Hall's theorem is due to Rado.  See \cite[Ch. 7]{Welsh} for details:
\begin{theorem} Let $M$ be a matroid on a set $E$.   A family of subsets $A_1,\dots,A_m\subseteq E$ has a transversal that is an independent set in $M$ if and only if for all $J\subseteq \{1,\dots,m\}$,
\[r(\cup_{j\in J} A_j)\geq |J|.\]
\end{theorem}

This has applications to finding a common transversal for two collections of subsets.
\end{example}

\begin{example} Algebraic matroids come from field extensions.  Let $\F$ be a field and let $\K$ be an extension of $\F$ generated by a finite subset $E\subset \K$.  We define a rank function as follows: for $S\subseteq E$,
\[r(S)=\operatorname{tr. deg}(\F(S)/\F),\]
the transcendence degree of $\F(S)$ over $\F$.  It turns out that $r$ defines a matroid.  It can be shown that every matroid representable over some field is an algebraic matroid over that same field \cite[Prop 6.7.11]{Oxley}.  There are examples of non-algebraic matroids and of algebraic, non-representable matroids.

Algebraic matroids can be interpreted geometrically.  Let $V\subset \A_{\F}^{n+1}$ be an algebraic variety.  Set $\K=\K(V)$, the function field of $V$.  Let $E=\{x_0,x_1,\dots,x_n\}$ be the coordinate functions on $V$.   The rank of the algebraic matroid given by $\K$ is the transcendence degree of $\K$ over $\F$, which is equal to $\dim V$.  A subset $S\subseteq E$ is a basis when $\operatorname{tr. deg}(\K/\F(S))=0$ which happens when the projection onto the coordinate space $\pi_S:V\rightarrow \A_{\F}^S$ is generically finite.
\end{example}

\begin{example} An important class of matroids are the paving matroids.  Conjecturally, they form almost all matroids \cite{MNWW}. A {\em paving} matroid of rank $r+1$ is a matroid such that any set $I\subset E$ with $|I|\leq r$ is independent.  Paving matroids can be specified in terms of their {\em hyperplanes}, that is, their rank $r$ flats as there is a cryptomorphic axiomatization of matroids in terms of their hyperplanes.  We use the following proposition where an $r$-partition of a set $E$ is a collection of subsets $\cT=\{T_1,T_2,\dots,T_k\}$ where for all $i$, $|T_i|\geq r$ and  each $r$ element subset of $E$ is a subset of a unique $T_i$:
\begin{proposition} \cite[Prop 2.1.24]{Oxley} If $\mathcal{T}=\{T_1,T_2,\dots,T_k\}$ is an $r$-partition of a set $E$, then $\cT$ is the set of hyperplanes of a rank $r+1$ matroid on $E$.  Moreover, for $r\geq 1$, the set of hyperplanes of every rank $r+1$ paving matroid on $E$ is an $r$-partition of $E$.
\end{proposition}

One example of a paving matroid is the V\'{a}mos matroid.  Here we follow the definition of \cite{Oxley}.  We set $E=\{1,2,3,4,1',2',3',4'\}$.  We set 
\[\cT_1=\{\{1,2,1',2'\},\{1,3,1',3'\},\{1,4,1',4'\},\{2,3,2',3'\},\{2,4,2',4'\}\]
 and 
\[\cT=\cT_1\cup\{T\subset E|\  |T|=3\text{\ and $T$ is not contained in any element of}\ \cT_1\}\]
Then $\cT$ is a $3$-partition and the set of hyperplanes of a rank $4$ matroid $V_8$. This matroid is not representable over any field.  One can view $V_8$ as the set of vertices of a cube whose bottom and top faces are labeled $\{1,3,2,4\}$ and $\{1',3',2',4'\}$ where we mandate that the points $\{3,4,3',4'\}$ are not coplanar.  This matroid turns out to be non-algebraic \cite{IngletonMain}.

For details on paving matroids, see \cite[Sec 2.1]{Oxley}.
\end{example}

\begin{example}  {\em Schubert matroids} are the matroids whose linear subspaces correspond to the generic point of a particular Schubert cell.    They were introduced by Crapo \cite{CrapoSingleelement}.  See \cite{ArdilaFinkRincon} for more details.

Schubert cells form an open stratification of the Grassmannian $\Gr(d+1,n+1)$ of $(d+1)$-dimensional subspaces of $\k^{n+1}$.  The Schubert cells consist of all the subspaces that intersect a flag of subspaces in particular dimensions.  Specifically, we have a flag of subspaces
\[\{0\}\subsetneq W_1\subsetneq W_2\subsetneq\ldots\subsetneq W_n\subsetneq W_{n+1}=\k^{n+1}\]
with $\dim W_i=i$.  For a subspace $V\in \Gr(d+1,n+1)$, this flag induces a nested sequence of subspaces
\[\{0\}\subseteq V\cap W_1\subseteq V\cap W_2\subseteq\ldots\subseteq V\cap W_n\subseteq V\cap W_{n+1}=V\]
The dimensions of these subspaces increase by $0$ or $1$ with each inclusion, so we can mandate where the jump occurs.  In the most generic situation, the sequence of dimensions would be $0,\ldots,0,1,2,\ldots,d,d+1$, that is, $\dim(V\cap W_{n-d+i})=i$.  We consider cases where where the sequence the jumps differs from that situation.   Specifically, we let $a_1,\dots,a_{d+1}$ be a non-increasing sequence of integers with $a_1\leq n-d$.  The Schubert cell is cut out by the open conditions that $\dim(V\cap W_{n-d+i-a_i})=i$.  Its closure, the corresponding Schubert variety is cut out by replacing the equality by ``$\geq$''.

To put Schubert cells into a matroid context, we must pick a flag.  Let $W_i=\Span(e_0,\ldots,e_{i-1})$.  In other words, $W_i$ is cut out by the system $x_i=\ldots=x_{n}=0$.  Therefore, we expect to have $\{n-a_{d+1},n-1-a_d,\ldots,n-d-a_1\}$ as a basis.  However, we need to impose more conditions to specify a matroid.  We will suppose that $V$ is generic with respect to the flag apart from these conditions.  We declare the bases of the matroid to be exactly the subsets $\{s_0,s_1,\ldots,s_d\}\subseteq \{0,\ldots,n\}$ such that $s_i\leq (n-i)-a_{d+1-i}$.  A point of view that will be taken up later is that matroids allow one to specify points in Schubert cells more precisely.
\end{example}

\section{Operations on Matroids} \label{s:operations}

In this section, we survey some of the operations for constructing and relating matroids.  Our emphasis is on explaining these constructions in the representable case.  For a more complete reference we recommend \cite{BrylawskiBook,Oxley}.

\subsection{Deletion and Contraction}

Let $M$ be a matroid of rank $d+1$ on a finite set $E$.  For $X\subset E$, we we may define the {\em deletion}  $M\setminus X$.  The ground set of $M\setminus X$ is $E\setminus X$ with rank function given as follows: for $S\subseteq E\setminus X$,
\[r_{M\setminus X}(S)=r_M(S).\]
If $M$ is represented by a vector configuration $v_0,\dots,v_n$, $M\setminus X$ is represented by $\{v_i\mid i\not\in X\}$.  Similarly, if $M$ is represented by a vector space $V\subseteq \k^{n+1}$, $M\setminus i$ is represented by $\pi(V)\subseteq \k^n$ where $\pi:\k^{n+1}\rightarrow \k^{n+1-|X|}$ is given by projecting out the coordinates corresponding to elements of $X$.  
Deletion on graphic matroids corresponds to deleting an edge from the graph.

For $T\subseteq E$, we define the \emph{restriction} by
\[M|_T=M\setminus(E\setminus T).\]
 If $F$ is a flat of $M$, then it is easy to see that the lattice of flats $L(M|_F)$ is the interval $[\hat{0},F]$ in $L(M)$.

A special case which will be of interest below is the deletion of a single element.  Let $i\in E$.  If $i$ is not a coloop, then there is a basis $B$ of $E$ not containing $i$.  This is rank $d+1$ subset of $M\setminus i$, and so $M\setminus i$ is a matroid of rank $d+1$.  The bases of $M\setminus i$ are the bases of $M$ that do not contain $i$.  By considering the matroid as a linear subspace, we see that  $\dim \pi(V)=\dim V$.  

There is a dual operation to deletion called \emph{contraction}.  Let $X\subset E$.  We define the contraction $M/X$ to be the matroid on ground set $E\setminus X$ given by for $S\subseteq E\setminus X$,
\[r(S)=r(S\cup X)-r(X).\]
Consequently, $M/X$ is  of rank $n+1-r(X)$.  
If $F$ is a flat of $M$, then $L(M/F)$ is the interval $[F,E]$ in $L(M)$.
If $M$ is represented by a vector configuration, $v_0,\dots,v_n$ in $\k^{d+1}$, $M/X$ is represented by $\{\pi(v_i)\mid i\not\in X\}$ where $\pi:\k^{d+1}\rightarrow \k^{d+1}/W$ is the projection and
\[W=\Span(\{v_i\mid i\in X\}).\]
Likewise if $M$ is represented by a subspace $V\subseteq \k^{n+1}$, $M/X$ is given by $V\cap L_X\subseteq L_X$ where $L_X$ is the coordinate subspace given by 
\[L_X=\{x\mid x_i=0 \text{\ if\ }i\in X\}.\]

If $i\in E$ that is not a loop, then the rank of $M/i$ is $d$.  In this case, if $M$ is represented by a subspace $V$ then $V$ intersects $L_i$ transversely.  If $i$ is a loop, the rank of $M/i$ is $d+1$.  In this case, $V$ is contained in the subspace $L_i$.
If $i$ is not a loop, the bases of $M/i$ are
\[\cB(M/i)=\{B\setminus\{i\}\mid i\in B,\ B\ \text{is a basis for\ } M\}.\] 

\begin{definition} A matroid $M'$ is said to be a {\em minor} of $M$ if it is obtained by deleting and contracting elements of the ground set of $M$.
\end{definition}
Note that a minor of a representable matroid is representable because we have a geometric interpretation of deletion and contraction on a vector arrangement or linear subspace.

\subsection{Direct sums of matroids}

Given two matroids $M_1,M_2$ on disjoint sets $E_1,E_2$, we may produce a direct sum matroid $M_1\oplus M_2$ on $E_1\sqcup E_2$.  Specifically, for $S_1\subseteq E_1,S_2\subseteq E_2$, we define
\[r(S_1\sqcup S_2)=r_1(S_1)+r_2(S_2).\]
The bases of $M$ are of the form $B_1\sqcup B_2$ for $B_1\in\cB_1,B_2\in\cB_2$. The circuits of $M_1\oplus M_2$ are the circuits of $M_1$ together with the circuits of $M_2$.
The lattice of flats obeys
\[L(M_1\oplus M_2)=L(M_1)\times L(M_2)\]
where the underlying set is the Cartesian product and $(F_1,F_2)\leq (F'_1,F'_2)$ if and only if $F_1\leq F'_1$ and $F_2\leq F'_2$.
 If $M_1$ is represented by vectors $v_0,\dots,v_m\in\k^{d_1+1}$ and $M_2$ is represented by vectors $w_0,\dots,w_n\in\k^{d_2+1}$, then $M_1\oplus M_2$ is represented by
\[(v_0,0),\ldots,(v_m,0),(0,w_0),\ldots(0,w_n)\in \k^{d_1+1}\oplus \k^{d_2+1}.\]
If $M_1,M_2$ are represented by subspace $V_1\subseteq \k^{n_1+1}, V_2\subseteq \k^{n_2+1}$, then $M_1\oplus M_2$ is represented by 
\[V_1\oplus V_2\subseteq \k^{n_1+1}\oplus \k^{n_2+1}.\]
For graphic matroids, direct sum corresponds to disjoint union of graphs.  

Every matroid has a decomposition analogous to the decomposition of a graph into connected components.  

\begin{definition} A matroid is {\em connected} if for every $i,j\in E$, there exists a circuit containing $i$ and $j$.
\end{definition}

We can define connected components of the matroid by saying that two elements $i,j$ are in the same connected component if and only if there exists a circuit containing $i$ and $j$. 
This is an equivalence relation.  It can be stated in terms of bases in the following form which will be important when we study matroid polytopes: two elements $i$,$j$ are in the same connected component if and only if there exists bases $B_1$ and $B_2$ such that $B_2=(B_1\setminus\{i\})\cup\{j\}.$  If $T_1,\dots,T_\kappa$ are the connected components of $M$, then we have a direct sum decomposition
\[M\cong M|_{T_1}\oplus\dots \oplus M|_{T_\kappa}.\]
Connected matroids are indecomposable under direct sum, and matroids have a unique direct sum decomposition into connected matroids.

Loops and coloops play a particular role in direct-sum decompositions.
Because the only circuit in which a loop $i$ occurs is  $\{i\}$, $\{i\}$ is a connected component.  Similarly, because a coloop $j$ does not occur in any circuit, $\{j\}$ is a connected component.  Therefore, loops and coloops may be split off from the matroid as in the following proposition:

\begin{proposition} Any matroid $M$ can be written as a direct sum
\[M\cong M'\oplus (U_{0,1})^{\oplus l}\oplus (U_{1,1})^{\oplus c}\]
where $M'$ has neither loops nor coloops.
\end{proposition}

\subsection{Duality} \label{ss:duality}
Duality is a a natural operation on matroids that generalizes duality of planar graphs.  

\begin{definition} The dual of a matroid $M$ on $E$ with rank function $r$ is defined to be the matroid on $E$ with rank function given by
\[r^*(S)=r(E\setminus S)+|S|-r(E).\]
\end{definition}

This rank function satisfies the axioms of a matroid.  If $M$ is rank $d+1$ on $E=\{0,1,\dots,n\}$, then $M^*$ is rank $n-d$.  
The bases of $M^*$ can be seen to be the set
\[\cB^*=\{E\setminus B\mid B\in \cB\}.\]
Duality interchanges loops and coloops, commutes with direct sum, and takes deletion to contraction:
$(M\setminus i)^*=M^*/i.$
As an example, we have $(U_{r+1,n+1})^*=U_{n-r,n+1}$.

The dual of a representable matroid is representable.  Let $M$ be represented by a vector configuration $v_0,v_1,\dots,v_n$ spanning $\k^{d+1}$.  This configuration can be thought of as a surjection $p:\k^{n+1}\rightarrow \k^{d+1}$ that fits into an exact sequence as
\[\xymatrix{0\ar[r]&\ker(p)\ar[r]^j&\k^{n+1}\ar[r]^p&\k^{d+1}\ar[r]&0}.\]
 We can take duals to get an injection
$p^*:(\k^{d+1})^*\hookrightarrow (\k^{n+1})^*.$
The dual of $M$ is given by the vector configuration corresponding to the quotient 
\[(\k^{n+1})^*\rightarrow (\k^{n+1})^*/p^*((\k^{d+1})^*)=\ker(p)^*.\]

Indeed, let $w_0,w_1,\dots,w_n$ be the dual basis of $(\k^{n+1})^*$.  For $S\subseteq \{0,\ldots,n\}$, the span of $\{w_i\mid i\in S\}$ induces a linear projection $\pi_S:\k^{n+1}\rightarrow \k^S.$  Write $\k^{E\setminus S}\subset\k^{n+1}$ for the kernel of that projection.
The rank of $S$ in the vector configuration induced by $\{w_0,\ldots,w_n\}$ is 
\begin{eqnarray*}
r^*(S)&=&\dim(j^*\pi_S^*((\k^S)^*))\\
&=&\dim(\pi_S(\ker(p)))\\
&=&\dim((\ker(p)+\k^{E\setminus S})/\k^{E\setminus S})\\
&=&\dim(\k^{E\setminus S}+\ker(p)/\ker(p))+\dim(\ker(p))-\dim(\k^{E\setminus S})\\
&=&\dim(p(\k^{E\setminus S}))+\dim(\ker(p))-\dim(\k^{E\setminus S})\\
&=&r(E\setminus S)+(n-d)-((n+1)-|S|)\\
&=&r(E\setminus S)+|S|-r(E).
\end{eqnarray*}
If a matroid is represented by a subspace $V\subset\k^{n+1}$, its dual is represented by $V^\perp\subset(\k^{n+1})^*$ where $V^\perp=\ker((\k^{n+1})^*\rightarrow V^*)$.

By interpreting of $M(G)$ and $M^*(G)$ as vector configurations given by chain and cochain groups, we see that these are dual matroids.   If $G$ is a planar graph, it turns out that $M^*(G)=M(G^*)$ where $G^*$ is the planar dual of $G$.  Because one can take the dual of any graphic matroids, the theory of matroids allows one to take the dual of a non-planar graph.

\subsection{Extensions} \label{ss:extensions}

Single-element extension is an operation on matroids inverse to single-element deletion.  Its properties were worked out by Crapo.  Given a matroid $M$ on a ground set $E$, it produces a new matroid $M'$ on a ground set $E'=E\sqcup \{p\}$ such that $M=M'\setminus p$.  Here we will follow the exposition of \cite{BrylawskiConstructions}.  We can partition the flats of $M$ into three sets based on how they change under extension:
\begin{eqnarray*}
\mathcal{K}_1&=&\{F\in \cf\mid\text{$F$ and $F\cup \{p\}$ are both flats of $M'$}\}\\
\mathcal{K}_2&=&\{F\in \cf\mid\text{$F$ is a flat of $M'$ but $F\cup\{p\}$ is not}\}\\
\mathcal{K}_3&=&\{F\in \cf\mid F\cup\{p\}\ \text{is a flat of $M'$ but $F$ is not}\}.
\end{eqnarray*}

This partition can be understood in the case where $M'$ is given by a vector configuration $\{v_0,\dots,v_n,v_p\}\subseteq \k^{d+1}$ and $M=M'\setminus p$.  The flats in $\mathcal{K}_3$ correspond to subspaces that contain $v_p$ and are spanned by a subset of $\{v_0,\dots,v_n\}$.  A flat $F$ of $M$ is in $\mathcal{K}_1$ when $v_p\not\in\Span(F)$ and $\Span(F\cup v_p)$ is not among the linear subspaces corresponding to flats of $M$.  Finally, $F\in \mathcal{K}_2$ when $v_p\not\in\Span(F)$ but $\Span(F\cup v_p)$ corresponds to some flat of $M$.  This flat must be an element of $\mathcal{K}_3$.  
If one knows $\mathcal{K}_3$, one can determine $\mathcal{K}_2$ and therefore $\mathcal{K}_1$: elements of $\mathcal{K}_2$ are exactly those that are contained in an element of $\mathcal{K}_3$.  Now, let us determine what properties that $\mathcal{K}_1$ and $\mathcal{K}_3$ should have.   If $F_1\subseteq F_2$ and $F_2\in\mathcal{K}_1$, then $F_1\in\mathcal{K}_1$.
Also, if $F_1\subseteq F_2$ and $F_1\in\mathcal{K}_3$, then $F_2\in\mathcal{K}_3$.  A more subtle property can be seen by considering intersections: if $F_1,F_2\in\mathcal{K}_3$, then $v_p\in\Span(F_1)\cap\Span(F_2)$;  therefore, if $\Span(F_1\cap F_2)=\Span(F_1)\cap \Span(F_2)$, then $F_1\cap F_2\in\mathcal{K}_3$.  These properties can all be established in the abstract combinatorial setting.
Translated into the matroid axioms, these properties say that $\mathcal{K}_3$ is a modular cut:

\begin{definition} Let $M$ be a matroid.  A subset $\mathcal{M}\subseteq \cf$ is said to be a {\em modular cut} if
\begin{enumerate}
\item If $F_1\in \mathcal{M}$ and $F_2\in\cf$ with $F_1\subseteq F_2$, then $F_2\in\mathcal{M}$ and
\item if $F_1,F_2\in\mathcal{M}$ satisfy
\[r(F_1)+r(F_2)=r(F_1\cap F_2)+r(F_1\cup F_2)\]
then $F_1\cap F_2\in\mathcal{M}$ \label{i:mc2}
\end{enumerate}
\end{definition}

Note that the condition on ranks in \eqref{i:mc2} above says that $\Span(F_1\cap F_2)=\Span(F_1)\cap \Span(F_2)$.

This characterization of $\mathcal{K}_3$ holds for all extensions of matroids and is fact a sufficient condition for an extension to exist:

\begin{proposition} For any single-element extension $M'$ of $M$, the set $\mathcal{K}_3$ is a modular cut.  Moreover, given any modular cut $\mathcal{M}$ of a matroid $M$, there is a unique single element extension $M'$, denoted by $M+_{\mathcal{M}} p$ such that $\mathcal{K}_3=\mathcal{M}$.
\end{proposition}

An extension is a composition of single element extensions.
Single-element extensions may leave the class of representable matroids: the matroid $M$ may be representable but $M'$ may not be.  Indeed, given a matroid $M$, there may not exist a set of vectors $\{v_0,\ldots,v_n\}\subseteq \k^{d+1}$ representing $M$ such that there is a vector $v_p$ contained in the subspaces corresponding to flats in the modular cut.  In fact, one can produce such an $M$ by beginning with a non-representable matroid and deleting elements until it becomes representable.

Single-element extensions have a geometric interpretation when one considers matroids represented by a subspace $V\subseteq \k^{n+1}$.  Specifically, one looks for a subspace $V'\subset\k^{n+2}$ such that $\pi(V')=V$ where $\pi:\k^{n+2}\rightarrow\k^{n+1}$ is projection onto the first $n+1$ factors.  We can interpret such a $V'$ as the graph of a linear function $l:V\rightarrow\k$.  Then $\mathcal{K}_3$ can be interpreted as the flats contained in $l^{-1}(0)$.

A special case of single element extension is that of a {\em principal extension}.  Specifically, one takes the modular cut to be all flats containing a given flat $F$.   In terms of vector configuration, this corresponds to adjoining a generic vector in the subspace corresponding to $F$.  The extension is denoted by $M+_F p$.
In the case where $F=E$, this is called the {\em free extension} and it corresponds to extending by a generic vector.

There is an operation inverse to contraction, called coextension.  In other words, given a matroid $M$, one produces $M'$ such that $M=M'/p$. Note that the rank of $M'$ will be one greater than the rank of $M$.  Because deletion is dual to contraction, coextension can be defined in terms of extension.  Specifically, let $\mathcal{M}$ be a modular cut in $M^*$.  Then the coextension associated to $\mathcal{M}$ is
\[M'=(M^*+_{\mathcal{M}} p)^*.\]
One can similarly define principal and free coextension.  

\subsection{Quotients and Lifts}

There is a natural quotient operation on matroids.  Specifically, one extends a matroid by an element and then one contracts the element.  The quotient given by a modular cut $\mathcal{M}$ is defined to 
$M+_{\mathcal{M}}p/p$.  In the case of vector configuration, this corresponds to taking a quotient of the ambient subspace: if $M'=M+_{\mathcal{M}}p$ is represented by $\{v_0,\ldots,v_n,v_p\}$ in $\k^{n+1}$, then $M'/p$ is represented by the image of $\{v_0,\ldots,v_n\}$ in $\k^{n+1}/\k v_p$.  One can define principal quotients by taking modular cuts associated to a flat.  If $\mathcal{M}=\{E\}$, then the quotient is given by a free extension by $v_p$ followed by a contraction by $v_p$.  This 
is called the {\em truncation} $\Trunc^d(M)$ of $M$.  It corresponds to taking the quotient of the ambient  space by a generic vector.   In general, for $0\leq k\leq d+1$, we define the $k$-truncation as the matroid on $E$ with rank function given by
\[r_{\Trunc^k(M)}(S)=\max(r(S),k).\]
If we view the matroid as a linear subspace $V\subset\k^{n+1}$, then truncation has a geometric interpretation as follows:

\begin{lemma} \label{l:trunc} Let $M$ be a rank $d+1$ matroid on $\{0,1,\dots,n\}$ represented by a $(d+1)$-dimensional subspace $V\subseteq\k^{n+1}$.  Let $H$ be a hyperplane in $\k^{n+1}$ that intersects $V_I$ transversely for all $I\subset\{0,1,\dots,n\}$.
Then $V\cap H$ represents $\Trunc^d(M)$.
\end{lemma}

\begin{proof}
We see that the rank function associated to $V\cap H$ obeys
\[r_{V\cap H}(S)=\codim(\dim (V_S\cap H)\subset (V\cap H))=\max(\codim(V_S\subset V)),d)=r_{\Trunc^d(M)}(S).\]
\end{proof}

Consequently, the $k$-truncation corresponds to intersecting with a generic $n+1-(d+1-k)$-dimensional subspace.

This will be important in the sequel.  There is also a dual notion to quotient, that of {\em lifts.} 

\subsection{Maps} There are several rival notions of morphisms between matroids.  We briefly review the notion of {\em strong maps} following Kung \cite{KungStrongmaps} noting that there is also a notion of weak maps.

\begin{definition} Let $M_1$ and $M_2$ be matroids on sets $E_1$ and $E_2$, respectively.  Let the direct sums $M_i\oplus U_{0,1}$ be given by extending by a loop $o_i$.  A strong map $\sigma$ from $M_1$ to $M_2$ is a function $\sigma:E_1\cup\{o_1\}\rightarrow E_2\cup \{o_2\}$ taking $o_1$ to $o_2$ such that the preimage of any flat of $M_2\oplus U_{0,1}$ is a flat in $M_1\oplus U_{0,1}$.
\end{definition}

A strong map turns out to be the composition of an extension and a contraction.  

\begin{definition} An \emph{embedding} of a matroid $M$ on $E$ into a matroid $M^+$ on $E^+$ is an inclusion $i:E\hookrightarrow E^+$ such that $M^+|_{i(E)}=M$ where we identify elements of $E$ with their images under $i$.  
\end{definition}

Embeddings are strong maps.  Contractions are strong maps as well.  If we contract by a set $U\subseteq E$, then the map is given by $\sigma:E\cup \{o_1\}\rightarrow (E\setminus U)\cup \{o_2\}$ where $\sigma$ is the identity on $E\setminus U$ and takes $U\cup \{o_1\}$ to $o_2$.  We see that we need to add a loop to define strong maps because we will need a zero element as a target for elements.
We have the following factorization theorem:
\begin{theorem} Let $\sigma:M_1\rightarrow M_2$ be a strong map.  Then there is a matroid $M^+$ such that $\sigma$ can be factored as
\[\xymatrix{M_1\ar[r]^i&M^+\ar[r]^p&M_2}\]
where $i$ is an embedding and $p$ is a contraction.
\end{theorem}

\subsection{Relaxation}

New matroids can be obtained from old by the technique of relaxation.  Recall that a circuit of a matroid is a minimal dependent set.  Matroids can be axiomatized in terms of circuits.  Let $X\subseteq E$ be a circuit that is also a hyperplane.  We call such a set a \emph{circuit-hyperplane}.  We construct a new matroid by mandating that $X$ be a basis.

\begin{proposition} \cite[Prop 1.5.14]{Oxley} Let $X$ be a circuit-hyperplane of a matroid $M$ on $E$.  Let $\cB$ be the set of bases of $M$ and set $\cB'=\cB\cup \{X\}$.  Then $\cB'$ is the set of bases for a matroid $M'$
\end{proposition}

By relaxing a representable matroid, one may obtain a non-representable matroid.
In fact, the non-Pappus matroid is obtained from the Pappus matroid by relaxing one line through three points.  The non-Fano matroid is similarly a relaxion of the Fano matroid.  The fields over which a matroid is representable may change by relaxation as the non-Fano matroid is a relaxation of the Fano matroid.

\section{Representability and Excluded Minor Characterizations}

\subsection{Introduction to Representability}
Representability is a central part of the combinatorial study of matroids.  Here, one would like a combinatorial description of representable matroids.  This is probably too much to ask.  However, one can study matroids that are representable over a fixed field.  Here, we need to discuss which classes of matroids might have a good structure theory.

The analogy with graph theory is particularly strong here.  The prototypical structural result that one would like to generalize is  Wagner's characterization of planar graphs, which states that a graph is planar if and only if it does not contain a $K_{3,3}$ or $K_5$ minor.   Note that the class of planar graphs is minor closed, that is, any minor of a graph in this class is also in this class.  Wagner's theorem produces a finite list of forbidden minors for this class of graphs.  A far-reaching generalization of this theorem is the Robertson-Seymour graph minors theorem which states that any minor-closed class of graphs has a finite list of excluded minors.  Examples of minor-closed classes of graphs include trees, linklessly embeddable graphs, and graphs with embedding genus at most $g$ for a fixed $g$.  The theorem gives a structural decomposition of an arbitrary minor-closed class of graphs and takes up more than $500$ journal pages.  While it shows that the number of forbidden minors is finite, it does not construct an explicit list.  See \cite{Diestel} for a nice summary.

Just as one might study particular classes of graphs like planar graphs, one can study particular classes of matroids.  The most natural classes of matroids would be those that are closed under taking minors, i.e. deleting and contracting elements.  One would hope for forbidden minor characterizations of these classes.

There are a number of natural minor-closed classes of matroids.  For a fixed field $\k$, the class of matroids that are representable over $\k$ is minor-closed.  This follows from the explicit construction of deletion and contraction on matroids represented by, say, a vector configuration.  Other minor-closed classes of graphs are graphic matroids (those of the form $M(G)$ for some graph $G$), cographic matroids ($M^*(G)$) which are minor-closed because deletion and contraction correspond to operations on the graphs.  Also, regular matroids, that is, matroids that are represented over every field, form a minor closed class.

There are all sorts of interesting containments between these classes of matroids.  All graphic matroids are regular, but the converse is not true.  For example, $M^*(K_5)$, being cographic, is regular, but it is not graphic while by duality $M(K_5)$ is not cographic.  The direct sum, $M(K_5)\oplus M^*(K_5)$ is regular but is neither graphic nor cographic since it contains $M(K_5)$ and $M^*(K_5)$ as minors.   Not all matroids are representable over $\F_2$.  For example,
the uniform matroid $U_{2,4}$ is not representable over $\F_2$ because there do not exist four vectors in general position in $\F_2^2$.  Moreover, not all $\F_2$-representable matroids are regular.  For example, the Fano matroid, $F_7$ is representable over $\F_2$ without being regular.  Indeed, it is representable over a field $\k$ if and only if its characteristic is $2$.  The non-Fano matroid, on the other hand, turns out to representable over ever field of characteristic not equal to $2$.

There are well-known forbidden minor characterizations of certain classes of matroids, a line of inquiry initiated by Tutte \cite{TutteHomotopy}.  Matroids representable over $\F_2$ are characterized by not having a $U_{2,4}$-minor by a theorem of Tutte.  Graphic matroids are exactly the matroids that do not have a $U_{2,4}, M^*(K_{3,3}),$ or  $M^*(K_5)$ minor.  This fact is a closely related to Wagner's theorem: because the dual of a graphic matroid is the graphic matroid of the dual graph, if it exists, it turns out that $M^*(K_{3,3})$ and $M^*(K_5)$ are not graphic.  One can even phrase Wagner's theorem in the language of matroids: a matroid is the graphic matroid of a planar graph if and only if it is does not contain any of the following matroids as minors: $U_{2,4}, M(K_{3,3}), M^*(K_{3,3}),M(K_5), M^*(K_5)$.  Matroids representable over $\F_3$ are those without minors isomorphic to  $U_{2,5},U_{3,5}, F_7,$ or $F_7^*$ by a theorem due independently to Bixby \cite{Bixby79} and Seymour \cite{Seymour79}.    By \cite{TutteHomotopy}, regular matroids are exactly the matroids not containing $U_{2,4}, F_7,$ or $F_7^*$.

These forbidden minor theorems can be proved by looking at possible matrices whose columns are the vectors of a representation.  One figures out the pattern of zero and non-zero entries in the matrix by considering the {\em fundamental circuits} of the matroid.  Specifically, one fixes a basis $B$, say $B=\{0,1,\dots,d\}$ and supposes that in the representation, the elements of the basis are given by the standard  basis vectors.  Then given $i\not\in B$, one can look at the circuit containing $i$ and some elements of $B$.  The vector representing $i$ must be a linear combination of the standard basis vectors in the circuit.  One considers different possibilities for the non-zero entries.  If one begins with a matroid that is minor-minimal among non-representable matroids, then the matrix turns out not to represent the matroid.  By applying matrix operations organized by combinatorial operations on the matroid, one can put the matrix and matroid in a standard form.  For example, because $\F_2$ has only one non-zero element, if a matroid is $F_2$-representable, once one fixes a basis, only one matrix needs to be considered.  Many of these representability theorems were originally proved using Tutte's Homotopy Theorem \cite{TutteHomotopy}, a deep and difficult theorem about the structure of a particular polyhedral complex on the flats of low corank.  Current research in representability relies on quite difficult theorems on decomposing matroids.  See \cite[Chapter 14]{Oxley} for a survey.

These excluded minor characterizations are over a fixed finite field, not over a field of fixed characteristic.  In particular, one is not allowed to take algebraic extensions of a given field.  This in essence creates two ways in which a matroid may fail to be representable: that there are not enough elements; or there is a contradiction in the linear conditions.  For example, the uniform matroid $U_{2,4}$ is not representable over $\F_2$ because there are no generic planes in $\F_2^4$.  The second problem is much more serious as can be seen in the matroids built by applying Mn\"{e}v's theorem as we will discuss in Subsection \ref{ss:mnev}.
  
The question of representability over infinite fields is much more difficult and is connected to Hilbert's Tenth problem, the question of algorithmic decidability of a system of polynomial equations.  We will discuss this connection below.  It turns out that there are infinitely many excluded minors for real representability.    
The hopelessness of the situation can be distilled into the following theorem of Mayhew, Newman, and Whittle \cite{MNWReal} which gives an explicit construction: 
\begin{theorem} Let $\K$ be an infinite field.  Let $N$ be a matroid that is representable over $\K$.  Then there exists a matroid $M$ that is an excluded minor for $\K$-representability, is not representable over any field, and has $N$ as a minor.
\end{theorem}
There has been research on the question of whether it is possible at all to give a finite, combinatorial condition for a matroid to be representable.  This is the so-called missing axiom of matroid theory.  A paper of V\'{a}mos \cite{VamosMissing} from 1978 asserted that the ``missing axiom of matroid theory is lost forever.''  In other words, there is no way of axiomatizing representability in second-order logic.  More recently, however, a paper of Mayhew, Newman, and Whittle \cite{MNWMissing} criticized V\'{a}mos's paper by stating that the matroids in it differed from matroids as commonly studied and asked again ``Is the missing axiom of matroid theory lost forever?'' conjecturing that representability cannot be axiomatized in a specific second-order logical language.

The general situation of representability over finite fields is the subject of Rota's conjecture that given any finite field $\F_q$, there are finitely many forbidden minors for $\F_q$-representability \cite{Rota}.  A solution has been announced by Geelen, Gerards, and Whittle \cite{GGWNotices}.  In addition, they have announced an excluded minor theorem for classes of $\F_q$-representable matroids:
\begin{theorem} Any minor-closed class of $\F_q$-representable matroids has a finite set of excluded $\F_q$-representable minors.  
\end{theorem}
Their work relies on a deep structure theory of matroids which generalizes the work of Robertson and Seymour.  See the survey \cite{GGWStructure} for an earlier account of their progress on Rota's conjecture.

\subsection{Inequivalent representations}

A matroid may have several representations that are essentially different.  Here, we follow the exposition of Oxley \cite{Oxley}.  Suppose a matroid is represented by two vector configurations $\{v_0,v_1,\dots,v_n\}$ and $\{v'_0,v'_1,\dots,v'_n\}$ in $\k^{n+1}$.  The two representations are said to be inequivalent if there does not exist an element of $\Gl_{n+1}(\k)$ taking one vector configuration to the other.  We may also consider representations of simple matroids as point configurations in $\P^n$.  Then we say two representations are projectively inequivalent if they are not related by an element of $\PGL_{n+1}(\k)$.  

It is not too hard to find matroids with inequivalent representations.  For example, $U_{3,5}$ has projectively inequivalent representations over $\F_5$.  Perhaps surprisingly, there are cases where matroids have projectively unique representations.  There is this theorem of Brylawski and Lucas:
\begin{theorem}\label{t:binaryunique}  Let $M$ be a simple matroid that is representable over $\F_2$.  Then for any field $\k$, any two representations of $M$ over $\k$ are projectively equivalent. 
\end{theorem}
Similarly, matroids have at most one representations over $\F_3$ up to projective equivalence. There is much research in trying to bound the number of inequivalent representations of certain classes of matroids. 

The existence of inequivalent representations is important for questions of representability: given a matroid $M$ that we wish to represent, we may delete different elements, try to find representations of various deletions, and then glue the representations of the deletions together; the existence of inequivalent representations is an obstruction to gluing the smaller representations together.  This phenomenon seems, to this author at least, reminiscent of the issues with automorphisms in moduli theory.

\subsection{Ingleton's criterion}

By linear algebraic arguments, Ingleton provided a necessary condition for a matroid to be representable over some field \cite{Ingleton}.

\begin{theorem} Let $M$ be a representable matroid on $E$.  Then for any subsets $X_1,X_2,X_3,X_4\subseteq E$, the rank function obeys the following inequality:
\begin{eqnarray*}
&&r(X_1)+r(X_2)+r(X_1\cup X_2\cup X_3)+r(X_1\cup X_2\cup X_4)+r(X_3\cup X_4)\\
&\leq&r(X_1\cup X_2)+r(X_1\cup X_3)+r(X_1\cup X_4)+r(X_2\cup X_3)+r(X_2\cup X_4).
\end{eqnarray*}
\end{theorem}

More recently, Kinser discovered an infinite family of independent inequalities whose leading member is Ingleton's \cite{Kinser}.

\begin{theorem} Let $k\geq4$ and let $M$ be a representable matroid on $E$.  Let $X_1,X_2,\dots,X_k\subseteq E$ be subsets.  Then the rank function obeys the following:
\begin{eqnarray*}
&&r(X_1\cup X_2)+r(X_1\cup X_3\cup X_k)+r(X_3)+ \sum_{i=4}^{k} \left(r(X_i)+r(X_2\cup X_{i-1}\cup X_i)\right)\\
&\leq&
r(X_1\cup X_3)+r(X_1\cup X_n)+r (X_2\cup X_3)+\sum_{i=4}^{k} \left(r(X_2\cup X_i)+r(X_{i-1}\cup X_i)\right).
\end{eqnarray*}
\end{theorem}

The Vam\'{o}s matroid, $V_8$ violates Ingleton's criterion with the choice of $X_i=\{i,i'\}$ and is therefore not representable over any field \cite[Exercise 6.1.7]{Oxley}.  The Vam\'{o}s matroid turns out to be a non-representable matroid on the minimum number of elements.

\subsection{Are most matroids  not representable?}

In \cite{BrylawskiKelly}, Brylawski and Kelly claimed 
\begin{quote}
``It is an exercise in random matroids to show that most
matroids are not coordinatizable [representable] over any field (or even any
division ring).''
\end{quote}
Unfortunately, to this date no one has been able to complete this exercise.  However, one can study matroid properties asymptotically: one examines the proportion of matroids on an $n$ element ground set that have a given property and take the limit as $n$ goes to infinity.   It is suspected that asymptotically most matroids are non-representable paving matroids.  A good reference for what is known is  \cite{MNWW}.   It is mentioned there that by counting arguments, for a fixed finite field $\F$, asymptotically, most matroids are not representable over $\F$.

\section{Polynomial invariants of matroids}

\subsection{The Characteristic and Tutte polynomials}

The characteristic polynomial is an invariant of matroids that generalizes the chromatic polynomial of graphs as introduced by Birkhoff \cite{Birkhoff}.  Let $G$ be a  graph.  The characteristic polynomial of $G$ is the function $\chi_G$ given by setting $\chi_G(q)$ to be the number of proper colorings of $G$ with $q$ colors for $q\in\Z_{\geq 1}$.  A proper coloring with $q$ colors is a function $c:V(G)\rightarrow \{1,2,\ldots,q\}$ such that adjacent vertices are assigned different values.  The function $\chi_G(q)$ can easily be shown to be a polynomial.  Note that a graph with a loop has no colorings.  Recall that for $e=v_1v_2$, an edge of $G$, the deletion $G\setminus e$ is the graph $G$ with the edge $e$ removed while the contraction $G/e$ is the graph $G$ with  $e,v_1,v_2$ contracted to a single vertex.  The chromatic polynomial obeys the deletion/contraction relation
\[\chi_G(q)=\chi_{G\setminus e}(q)-\chi_{G/e}(q).\]
This can be seen by interpreting $\chi_{G\setminus e}(q)$ as the count of colorings where we do not impose the condition that the colors of the end-points of $e$ are different and interpreting $\chi_{G/e}(q)$ as colorings where the colors of the end-points of $e$ are mandated to be the same.  One can show that $\chi_G(q)$ is a polynomial by applying deletion/contraction repeatedly and then noting that the chromatic polynomial is multiplicative over disjoint union of graphs and that if $G$ consists of a single vertex, then $\chi_G(q)=q$.

The chromatic polynomial can be extended to matroids in a straightforward fashion.  One looks for a polynomial $\chi_M(q)$, called \emph{the characteristic polynomial} for a matroid $M$ which satisfies
\begin{enumerate}
\item (loop property) if $M$ has a loop, then $\chi_M(q)=0$,
\item (normalization) the characteristic polynomial of the uniform matroid $U_{1,1}$ satisfies
\[\chi_{U_{1,1}}(q)=q-1.\]
\item (direct sum) If $M=M_1\oplus M_2$ then
\[\chi_M(q)=\chi_{M_1}(q)\chi_{M_2}(q),\]
\item (deletion/contraction) if $i$ is not a loop or coloop of $M$ then 
\[\chi_M(q)=\chi_{M\setminus i}(q)-\chi_{M/i}(q),\]
\end{enumerate}
By inducting on the size of the ground set, we easily see that an invariant satisfying these properties is unique.  Note that loops and coloops split off from a matroid in a direct sum decomposition.
It turns out that $\chi_M(q)$ is well-defined.  In fact, we will construct it explicitly below by using M\"{o}bius functions.

The characteristic polynomial is a generalization of the chromatic polynomial in the following sense:
\begin{proposition} If $G$ is a graph and $M(G)$ is its matroid, then
\[\chi_G(q)=q^\kappa\cdot\chi_{M(G)}(q)\]
where $\kappa$ is the number of components of $G$. 
\end{proposition}
This theorem is proved by showing that both sides obey the same deletion-contraction relation and then checking the equation on the seed value of a graph with two vertices and a single edge between them whose matroid is $U_{1,1}$.

One can generalize the characteristic polynomial further by  removing the restriction that it vanishes on loops.  Tutte \cite{TutteRing} introduced his eponymous polynomial for graphs by studying all possible topological invariants of graphs that were well-behaved under deletion and contraction.  The definition was extended to matroids by Crapo \cite{CrapoTutte}.
We consider all invariants that are well-behaved under deletion/restriction and multiplicative under direct sum.  Let $\Mat$ be the class of all matroids.  We define the Tutte-Grothendieck ring $K_0(\Mat)$ as the commutative ring given by formal sums of isomorphism classes of matroids subject to the relations
\begin{enumerate}
\item if $i$ is neither a loop nor a coloop of a matroid $M$, then
$[M]=[M\setminus i]+[M/i],$
\item $[M_1\oplus M_2]=[M_1][M_2],$
\end{enumerate}
This Tutte-Grothendieck group was used implicitly by Tutte for graphs before Grothendieck's introduction of the Grothendieck group.  Because deletion and contraction commute, a matroid can be written uniquely in $K_0(\Mat)$ as a polynomial in the rank $1$ matroids.  Consequently, the Grothendieck group $K_0(\Mat)$ is the free polynomial ring over $\Z$ generated by $U_{0,1}$ (a loop) and $U_{1,1}$ (a coloop).

\begin{definition} Let $R$ be a commutative ring.  A Tutte-Grothendieck invariant valued in $R$ is a homomorphism 
\[h:K_0(\Mat)\rightarrow R.\]
\end{definition}

\begin{definition} The {\em rank generating} polynomial of a matroid $M$ is defined by
\[R_M(u,v)=\sum_{I\subseteq E} u^{r(M)-r(I)}v^{|I|-r(I)}.\]
The {\em Tutte polynomial of $M$} is defined by
\[T_M(x,y)=R_M(x-1,y-1).\]
\end{definition}

The definition is justified by the following:

\begin{proposition}
The Tutte polynomial is the unique Tutte-Grothendieck invariant $T:K_0(\Mat)\rightarrow\Z[x,y]$ satisfying
$T_{U_{1,1}}(x,y)=x$ and $T_{U_{0,1}}(x,y)=y$.
\end{proposition}
This result naturally generalizes a formula for the characteristic polynomial that we will explore later.
 
This theorem shows that $U_{0,1}$ and $U_{1,1}$ are algebraically independent in $K_0(\Mat)$.
The Tutte polynomial specializes to the characteristic polynomial:
\[\chi_{M}(q)=(-1)^{r(M)}T_M(1-q,0).\]
This can be seen by observing that the characteristic polynomial has the following properties: it takes the value $q-1$ on coloops; it vanishes on loops; its deletion/restriction relation has a sign, which is accounted for by the fact that $r(M/i)=r(M)-1$ for a non-coloop $i$.

The Tutte polynomial is well-behaved under duality: $T_M(x,y)=T_{M^*}(y,x)$.

\subsection{Motivic definition of characteristic polynomial}

Consider the Grothendieck ring of varieties over $\k$, $K_0(\Var_\k)$ where $\k=\C$ or $\k=\F_{p^r}$.  This is the ring of formal sums of varieties where addition is disjoint union and multiplication is Cartesian product subject to the scissors relation: for $Z\subset X$,
\[[X]=[Z]+[X\setminus Z].\]

One can show that there is a homomorphism
\[e:K_0(\Var_k)\rightarrow \Z[q]\]
such that 
\[e(\k)=q.\]

\begin{definition} For $V^{d+1}\subseteq \k^{n+1}$, a $(d+1)$-dimensional linear subspace, the {\em characteristic polynomial} of $V$ is 
\[\chi_V(q)\equiv e([V\cap (\k^*)^{n+1}]).\]
\end{definition}

\begin{example}
By inclusion/exclusion for a generic subspace, we have
\[\chi_V(q)=q^{d+1}-\binom{d+1}{1}q^r+\binom{d+1}{2}q^{d-1}-\dots+(-1)^{d+1}\binom{d+1}{0}.\]

Here, we begin with $V$ remove the $d+1$ coordinate hyperplanes, add back in $\binom{d+1}{2}$ codimension $2$ coordinate flats, and so on.
\end{example}

\begin{theorem} \label{t:motivic} Let $V\subseteq\k^{n+1}$ be a linear subspace with matroid $M$.  Then
\[\chi_V(q)=\chi_M(q).\]
\end{theorem}

\begin{proof}
We verify that $\chi_V$ obeys the axioms for the characteristic polynomial.  We note that the direct sum decomposition, deletion, and contraction of a representable matroid only involve representable matroids, so a function obeying the axioms on representable matroids must agree with the characteristic polynomial.

If $M$ has a loop then $V$ is contained in a coordinate subspace, and each side of the equation is $0$.  Therefore, we may suppose that $E$ is loopless.

If $M=U_{1,1}$ then $V=\k\subseteq\k$.  Therefore, $V\cap(\k^*)=[\k^1]-[\operatorname{pt}]$ and $\chi_V(q)=q-1$.

If $M=M_1\oplus M_2$, then $V=V_1\oplus V_2\subseteq \k^{n_1}\oplus \k^{n_2}$.  Consequently,
\[[V\cap (\k^*)^{n+1}]=[V_1\cap (\k^*)^{n_1}][V_2\cap (\k^*)^{n_2}].\] 
Applying $e$, we get the direct sum relation.

Let $i\in E$ be neither a loop nor a coloop. Let $H_i\subset \k^{n+1}$ be the hyperplane cut out by $x_i=0$.  Let $A_i$ be the set given by 
\[A_i=\{(x_0,x_1,\dots,x_n)\mid x_j\neq 0 \text{\ if \ } i\neq j\}.\]
Then, we have the following motivic equation
\[[V\cap (\k^*)^{n+1}]=[V\cap A_i]-[V\cap A_i\cap H_i].\]
Since $i$ is not a coloop, the coordinate projection $\pi:\k^{n+1}\rightarrow \k^n$ forgetting the $i$th component is injective on $V$.  Moreover, $\pi$ takes $V\cap A_i$ bijectively to $\pi(V)\cap (\k^*)^n$.   Now, $\pi(V)\subseteq \k^n$ is the subspace representing $M\setminus i$.  On the other hand $V\cap A_i\cap H_i$ is  the intersection of $V\cap H_i$ with the algebraic torus in $H_i$.  Now, $V\cap H_i$ is the subspace representing $V/i$, and the deletion/contraction relation follows.
\end{proof}

Note that $\chi_M(q)$ has degree $d+1$ with the leading coefficient equal to $1$ as the only $d+1$-dimensional variety that contributes to the motivic expression is $V$ itself. 

We now relate the characteristic polynomial to geometric invariants of $V\cap (\k^*)^{n+1}$.  If $V\subset\k^{n+1}$ is a subspace, the motivic class of $[V\cap(\k^*)^{n+1}]$ lies in the commutative subring with unity $\<[\k]\>\subseteq K_0(\Var_\k)$ generated by $[\k]$.   This can be shown by
expressing $V\cap (\k^*)^{n+1}$ motivically in terms of $V$ and its intersections with coordinate flats.  Consequently, for any ring $R$, if $e':K_0(\Var_\k)\rightarrow R$ is any homomorphism, $e'([V\cap(\k^*)^{n+1}])$ is determined by $e'([\k])\in R$.  It follows that 
\[e'([V\cap (\k^*)^{n+1})=\chi_V(e'[\k]).\]

If $\k=\C$, we can choose the homomorphism $\chi^c_y:K_0(\Var_\k)\rightarrow \Z[u]$ to be the compactly-supported $\chi_y$-characteristic given on varieties by  
\[\chi_y^c(X)=\sum_{p,q} \sum_m (-1)^m h^{p,q}(H^m(X)) u^p\]
where the Hodge numbers are taken with respect to Deligne's mixed Hodge structure \cite{PetersSteenbrink}.  Since $\chi_y^c([\k])=u-1$, we derive the following result of Orlik-Solomon \cite{OrlikSolomon}  as proved in \cite{Aluffi,KatzStapledon}

\begin{theorem}
Let $V\subseteq \C^{n+1}$ be a subspace, then we have
\[\chi_y^c([V\cap (\C^*)^{n+1}])=\chi_V(u).\]
\end{theorem}

Similarly, if $\k=\F_{p^r}$, counting the number of $\F_{p^r}$-points gives a homomorphism $K_0(\Var_k)\rightarrow \Z$.  Since $|\F_p^r|=p^r$, we have the following theorem of Athanasiadis \cite{AthanasiadisAdvances}:

\begin{theorem}
Let $V\subseteq \F_{p^r}^{n+1}$ be a subspace, then we have
\[|V\cap (\F_{p^r}^*)^{n+1}|=\chi_V(p^r).\]
\end{theorem} 

\subsection{Mobius inversion and the characteristic polynomial}

The characteristic polynomial has a description given by inclusion/exclusion along the poset of flats.  This is phrased in the language of M\"{o}bius inversion which is a combinatorial abstraction of the above motivic arguments.  We will only make use of the M\"{o}bius function evaluated from the minimal flat.  Let $L(M)$ be the poset of flats of $M$ ordered by inclusion.  We define $\mu(E,F)$ for every pair of flats $E\subseteq F$ recursively by
\begin{enumerate}
\item $\mu(E,E)=1,$
\item $\mu(E,F)=-\sum_{E\subseteq F'\subset F} \mu(E,F')$.
\end{enumerate}

We will suppose that $M$ is loopless and concentrate on $\mu(\emptyset,F)$.    For $I\subseteq E=\{0,1,\dots,n\}$, let $L_I$ be the coordinate flat 
\[L_I=\{x\mid x_i=0 \text{\ if\ }i\in I\}.\]
Note that $\codim(V\cap L_I)=r(I)$ and $V\cap L_I=V\cap L_F$ where $F$ is the smallest flat containing $I$. Let
\[L^*_I=\{x\mid x_i=0 \text{\ if and only if } i\in I\}.\]
The following lemma, which is a special case of M\"{o}bius inversion, is the motivation for the above definition.
\begin{lemma} \label{l:mobius}
\[[V \cap (\k^*)^{n+1}]=\sum_{F\in L(M)} \mu(\emptyset,F)[V\cap L_F].\]
\end{lemma}

\begin{proof}
We note that for any flat $F$, 
\[[V\cap L_F]=\sum_{\substack{F'\in L(M)\\F'\supseteq F}} [V\cap L^*_{F'}].\]
Consequently,
\begin{eqnarray*}
\sum_{F\in L(M)} \mu(\emptyset,F)[V\cap L_F]&=&\sum_{\substack{F,F'\in L(M)\\F\subseteq F'}} \mu(\emptyset,F)[V\cap L^*_{F'}]\\
&=&\sum_{F'\in L(M)} \left(\sum_{\substack{F\in L(M)\\F\subseteq F'}} \mu(\emptyset,F)\right)[V\cap L^*_{F'}]\\
&=&[V\cap L^*_{\emptyset}].
\end{eqnarray*}
\end{proof}

As a consequence of Weisner's theorem \cite[Section 3.9]{StanleyBook}, \cite[Theorem 3.10]{StanleyNotes} for loopless matroids, we have the following result which will be important in the sequel:

\begin{lemma} \label{l:weisner}
Let $F$ be a flat of a loopless matroid $M$.  
For any $a \in F$,
\[\mu(\emptyset,F)=-\sum_{a\notin F' \lessdot F} \mu(\emptyset, F')\]
where $F'\lessdot F$ means that $F'\subset F$ and $r(F')=r(F)-1$. 
\end{lemma}

We have the following description of the characteristic polynomial:
\begin{theorem} \label{t:charpoly}
The characteristic polynomial for a loopless matroid $M$ of rank $d+1$ is given by 
\[\chi_M(q)=\sum_{F\in L(M)} \mu(\emptyset,F)q^{d+1-r(F)}.\]
If $M$ has loops, then $\chi_M(q)=0$.
\end{theorem}

This theorem can be proven for representable matroids by combining Lemma \ref{l:mobius} and Theorem \ref{t:motivic}.
Before we give the proof of this theorem, we will need the following lemma:

\begin{lemma} Let $F$ be a flat of a matroid $M$.  Then
\[U_F:=\sum_{\substack{S\subset E\\ \cl(S)=F}} (-1)^{|S|}=
\begin{cases}
\mu(\emptyset,F)& \text{if $M$ is loopless}\\
0   &\text{otherwise.}
\end{cases}\]
\end{lemma}

\begin{proof}
First, suppose $M$ is loopless.  We will show that $U_F$ agrees with $\mu(\emptyset,F)$ inductively.  It is clear that $U_\emptyset=1$ and so it suffices to show (where $\leq$ means containment of flags)
\[U_F=-\sum_{\emptyset\leq F< F'} U_{F'}.\]
Now,
\[\sum_{\emptyset\leq F'\leq F} U_{F'}=\sum_{\emptyset\leq F'\leq F} \sum_{\substack{S\subset E\\ \cl(S)=F'}} (-1)^{|S|}=\sum_{S\subseteq F} (-1)^{|S|}=(1-1)^{|S|}=0.
\]
Suppose that $M$ has loops.  Let $F\subseteq E$ be the minimal flat of $M$, which consists of all the loops.  Then
\[U_F=\sum_{S\subseteq F} (-1)^{|S|}=(1-1)^{|S|}=0.\]
The same inductive argument shows that $U_F=0$ for all flats.
\end{proof}

Note that in the notation of the above proof, Theorem~\ref{t:charpoly} is equivalent to 
\[\chi_M(q)=\sum_{F\in L(M)} U_F q^{d+1-r(F)}.\]
Now we continue with the proof of Theorem~\ref{t:charpoly}.

\begin{proof}
We verify the axioms for the characteristic polynomial.  

By definition, the formula is true if $M$ has a loop.

If $M=U_{1,1}$, it has two flats.  We have $\mu(\emptyset,\emptyset)=1$, $\mu(\emptyset,\{0\})=-1$.  Consequently, our formula gives $q-1$.

Suppose $M=M_1\oplus M_2$.  If either $M_i$ has a loop, then so does $M$ and the formula is verified.
Otherwise, the lattice of flats obeys $L(M)=L(M_1)\times L(M_2)$.  It is easily seen that for $F_1\in L(M_1),F_2\in L(M_2)$, the M\"{o}bius function obeys
\[\mu_M(\emptyset,F_1\times F_2)=\mu_{M_1}(\emptyset,F_1)\mu_{M_2}(\emptyset,F_2).\]
Therefore, $\chi_M(q)=\chi_{M_1}(q)\chi_{M_2}(q).$

Now, we verify the deletion/contraction relation using an approach due to Whitney.  Let $i\in E$ be neither a loop nor coloop.    Then we have
\begin{eqnarray*}
\sum_{F\in L(M)} U_F q^{d+1-r(F)}&=&\sum_{F\in L(M)}\left(\sum_{\substack{S\subseteq F\\ \cl(S)=F}} (-1)^{|S|}\right)q^{d+1-r(F)}\\
&=&\sum_{S\subseteq E} (-1)^{|S|} q^{d+1-r(S)}\\
&=&\sum_{S\subseteq E\setminus\{i\}} (-1)^{|S|} q^{d+1-r(S)}
+\sum_{S\subseteq E\setminus\{i\}} (-1)^{|S\cup\{i\}|} q^{d+1-r(S\cup\{i\})}\\
&=&\sum_{S\subseteq E\setminus\{i\}} (-1)^{|S|} q^{d+1-r_{M\setminus i}(S)}
-\sum_{S\subseteq E\setminus\{i\}} (-1)^{|S|} q^{d-r_{M/i}(S)}\\
&=&\sum_{F\in L(M\setminus  i)} (U_{M\setminus i})_F q^{d+1-r_{M\setminus i}(F)}
-\sum_{F\in L(M/i)}(U_{M/i})_F q^{d-r_{M/i}(S)}\\
&=&\chi_{M\setminus i}(q)-\chi_{M/i}(q)
\end{eqnarray*}
where we have used the fact that because $i$ is neither a loop nor a coloop,
\[r(M\setminus i)=d+1,\ r(M/i)=d.\]
\end{proof}

We see from Lemma~\ref{l:weisner} that the coefficients of the characteristic polynomial alternate and the degree of the characteristic polynomial is equal to the rank of the matroid.

\subsection{The reduced characteristic polynomial}

In this section, we introduce the reduced characteristic polynomial which will be important for the proof of the log-concavity of the characteristic polynomial of representable matroids.  We will only consider loopless matroids.
We begin by noting that $\chi_M(q)$ is divisible by $q-1$.  This can be proved combinatorially.  In the representable case, we interpret $\chi_M(q)$ as the image of the motivic class of $V\cap (\k^*)^{n+1}$.  Because $\k^*$ acts freely on $V\cap(\k^*)^{n+1}$ we have
\[e([V\cap (\k^*)^{n+1}])=e(([V\cap (\k^*)^{n+1})/\k^*])e([k^*])\]
with $e([k^*])=q-1$.

\begin{definition} The reduced characteristic polynomial is
\[\overline{\chi}_M(q)=\frac{\chi_M(q)}{q-1}.\]
Define the numbers $\mu^0,\mu^1,\dots,\mu^d$ by
\[\overline{\chi}_M(q)=\sum_{i=0}^d (-1)^i\mu^iq^{d-i}.\]
\end{definition}

Note that $\mu^0=1$.  We follow the convention that $\mu^{-1}=0$.

\begin{lemma} \label{l:reducedcoeff} Let $a\in E$.  The coefficients of $\overline{\chi}_M(q)$ are given by
\[\mu^k=(-1)^k\sum_{\substack{F\in L(M)_k\\a\not\in F}} \mu(\emptyset,F)=(-1)^{k+1}\sum_{\substack{F\in L(M)_{k+1}\\ a\in F}} \mu(\emptyset,F)\]
where the sum over the rank $k$ flats not containing $a$.
\end{lemma}

\begin{proof} We begin by proving the second equality by applying Lemma~\ref{l:weisner}
\begin{eqnarray*}
\sum_{\substack{F\in L(M)_{k+1}\\ a\in F}} \mu(\emptyset,F)&=&
-\sum_{\substack{F\in L(M)_{k+1}\\ a\in F}} \sum_{\substack{F'\lessdot F\\a\not\in F'}} \mu(\emptyset,F')\\
&=&-\sum_{\substack{F'\in L(M)_{k}\\ a\not\in F}} \sum_{\substack{F\gtrdot F'\\a\in F}} \mu(\emptyset,F')\\
&=&-\sum_{\substack{F'\in L(M)_{k}\\ a\not\in F}} \mu(\emptyset,F')
\end{eqnarray*}
where the last equality follows from the fact that for $F'$ not containing $a$, there is a unique flat $F$ with $r(F)=r(F')+1$ and $a\in F$. 
If we write 
\[\chi_M(q)=\mu_0q^{r+1}-\mu_1q^r+\dots+(-1)^{d+1}\mu_{r+1}\]
then
\[\mu_k=\sum_{F\in L(M)_k} \mu(\emptyset,F)\]
and by equating coefficients of $\overline{\chi}_M(q)$ and $\frac{\chi_M(q)}{q-1}$,
\[\mu_k=\mu^k+\mu^{k-1}.\]
The theorem is true for $k=0$.  It now follows by induction using the above formula.
\end{proof}

\subsection{Log-concavity of Whitney numbers}

We now discuss some conjectured inequalities for characteristic polynomials.

\begin{definition} A polynomial 
\[\chi(q)=\mu_0q^{r+1}+\mu_1q^r+\dots+\mu_{r+1}\]
is said to be {\em unimodal} if the coefficients are unimodal in absolute value, i.e. there is a $j$ such that
\[|\mu_0|\leq |\mu_1|\leq\dots\leq|\mu_j|\geq|\mu_{j+1}|\geq\dots\geq|\mu_{r+1}|.\]
\end{definition}

The chromatic polynomial of a graph was conjectured to be unimodal by Read \cite{Read}.  This conjecture was proved by Huh in 2010 \cite{Huh}.  In fact, Huh did more.  He proved the characteristic polynomial of a matroid representable over of a field of characteristic $0$ is log-concave.

\begin{definition} The polynomial $\chi(q)$ is said to be {\em log-concave} if
for all $i$, 
\[|\mu_{i-1}\mu_{i+1}|\leq\mu_i^2.\]
\end{definition}

If the inequality is strict then the polynomial is said to be strictly log-concave.
Log-concavity means that the logarithms of the absolute values of the (alternating) coefficients form a concave sequence.  Log-concavity implies unmodality.  Below we will give a streamlined version of the proof of Huh and the author \cite{HuhKatz} that the characteristic polynomial of any representable matroid is log-concave.  This proof is very similar to Huh's original proof except that it replaces singularity theory with intersection theory on toric varieties. This theorem is a special case of the still-open Rota-Heron-Welsh conjecture:

\begin{conjecture} For any matroid, $\chi_M(q)$ is log-concave.
\end{conjecture}

The characteristic polynomial of a matroid representable over a field of characteristic $0$ was proved to be strictly log-concave by Huh in \cite{HuhLogconcave} by studying the variety of critical points that arises in maximum likelihood estimation.

A closely related conjecture is Mason's conjecture.  Let $M$ be a matroid of rank $d+1$.  Let $\cI$ be the set of independent sets of $M$.  We define the $f$-vector by setting $f_i$ to be the number of independent sets of cardinality $i$.  it was conjectured by Mason the $f$-vector was log-concave, that is the polynomial
\[f(q)=\sum_{i=0} f_iq^{r+1-i}\]
is log-concave. 
In fact, Mason strengthened his conjecture to these statement for every $k$ to the following:
\begin{enumerate}
\item $f_k^2\geq \frac{k+1}{k}f_{k-1}f_{k+1}$
\item $f_k^2\geq \frac{k+1}{k}\frac{f_1-k+1}{f_1-k} f_{k-1}f_{k+1}$
\end{enumerate}

Mason's original log-concavity conjecture was proved in the representable case by Lenz:

\begin{theorem} \cite{Lenz} The $f$-vector of a realizable matroid is strictly log concave.
\end{theorem} 

The $h$-polynomial is defined by 
\[h(q)=f(q-1).\]
Dawson \cite{Daw84} conjecture this to be log-concave.  Huh \cite{HuhLogconcave} proved that $h$ is log-concave and without internal zeroes if $M$ is representable over a field of characteristic $0$.

A much less tractable sequence of numbers conjectured to be log-concave are the Whitney numbers of the second kind.  They are $W_k$,  the number of flats of $M$ of rank $k$.   There are strengthenings of the log-concavity conjecture for $W_k$ by Mason analogous to those of the $f$-vector.   A special case is the Points-Lines-Planes conjecture:

\begin{conjecture} If $M$ is a rank four matroid, then
\[W_2^2\geq \frac{3}{2}W_1W_3.\]
\end{conjecture}

In the representable case, we image having $W_1$ points in $3$-space determining $W_2$ lines and $W_3$ planes.  A slightly strengthening of this conjecture was proved by Seymour for all matroids in which no line has more than four points \cite{Seymour82}.  See \cite{Aigner} for a survey.

\section{Matroid polytopes}

\subsection{Definition of matroid polytope}

We first give a historical introduction to matroid polytopes using the independence polytope of Edmonds \cite{Edmonds}, which he introduced as a tool in combinatorial optimization.  Specifically, Edmonds was interested in maximizing a weight function along independent subsets.  Let $c:E\rightarrow \R$ be a function on the ground set.  For $I\subseteq E$, an independent subset, set the weight function on independent sets to be 
\[c(I)=\sum_{i\in I} c(i).\]
Now, a natural problem is to find a basis for the matroid of maximum weight.  Matroids have the property that this problem is solvable by a greedy algorithm:
\begin{enumerate}
\item Set $B:=\emptyset$,
\item While $B$ is not a basis, pick $i\not\in B$ such that $B\cup\{i\}$ is independent and $i$ is of maximum weight and replace $B$ by $B\cup\{i\}$.
\end{enumerate}
The resulting set $B$ is a basis of maximum weight.  Moreover, matroids can be characterized as a structure for which the greedy algorithm succeeds.

The basis optimization problem can also be phrased in terms of maximizing a linear functional on the independence polytope $P(\cI(M))\subset \R^E$ given by
\[P(\cI(M))=\Conv(\{e_I \mid I\in \cI\})\]
where $e_I$ is the characteristic vector of an independent set $I$,
\[e_I=\sum_{i\in I} e_i.\]
We can define a {\em weight vector} $w\in \R^n$ by
\[w=\sum c(i)e_i.\]
The set of points $x\in P$ that maximize the function $\<x,w\>$ is a face of the independence polytope.  The vertices of that faces are the independent sets of maximum weight.

Edmonds studied the structure of independence polytope and was able to prove the following important theorem in combinatorial optimization: 
\begin{theorem} Let $M_1,M_2$ be two matroids on the same ground set where intersection is given by intersecting the collection of independent subsets. Then the independence polytope is well-behaved under intersections in the following sense:
\[P(\ci(M_1)\cap \ci(M_2))=P(\ci(M_1))\cap P(\ci(M_2))\]
\end{theorem}
This theorem has relevance to finding common transversals.  
We recommend \cite{BixbyCunningham} as a survey on matroid polytopes in optimization.

Matroid polytopes received renewed attention due to work of Gelfand-Goresky-MacPherson-Serganova \cite{GGMS}.  In this case, they investigated the convex hull of the characteristic vectors of the bases of a matroid.

\begin{definition} The {\em matroid polytope} of a matroid $M$ is given as the following convex hull:
\[P(M)=\Conv(\{e_B \mid B\in \cB\}).\]
\end{definition}

\begin{lemma} The vertices of the matroid polytope $P(M)$ are exactly the characteristic vectors of the bases of $M$.
\end{lemma}

\begin{proof} If $e_B$ was a convex combination of $\{e_{B_1},\dots,e_{B_k}\}$ then we could write 
\[e_B=\sum_j \lambda_j e_{B_j}\]
for $0<\lambda_j$, $\sum \lambda_j=1$.  For each $i\in B$, by extracting the component of $e_i$ from the above equation, we would have $i\in B_j$ for all $j$.  Therefore, we conclude $B_1=\dots=B_k=B$.  Consequently, $e_B$ cannot be written as a non-trivial convex combination.
\end{proof}

\begin{example} The matroid polytope of the uniform matroid $U_{d+1,n+1}$ is $\Delta(d+1,n+1)$, the $(d+1,n+1)$-hypersimplex which is defined to be the convex hull of the $\binom{n+1}{d+1}$ vectors $e_B$ as $B$ ranges over all $(d+1)$-element subsets of $E=\{0,1,\dots,n\}$. 
\end{example}

\begin{example} The matroid polytope of $U_{2,4}$ is an octahedron in $\R^4$ whose vertices are 
\[\{(1,1,0,0),(1,0,0,1),(1,0,1,0),(0,1,1,0),(0,1,0,1),(0,0,1,1)\}\]
where there is an edge between a pair of vertices if and only if they differ in exactly two coordinates.
\end{example}

Remarkably, matroid polytopes can be abstractly characterized leading to a cryptomorphic definition of matroids by the work of Gelfand-Goresky-MacPherson-Serganova:

\begin{definition} A {\em matroid polytope} $P$ is a convex lattice polytope in $\R^{n+1}$ which is contained in $\Delta(d+1,n+1)$ for some $d$ such that all vertices of $P$ are vertices of $\Delta(d+1,n+1)$ and each edge is a translate of $e_i-e_j$ for $i\neq j$.
\end{definition}

\begin{proposition} Let $M$ be a rank $d+1$ matroid on $E=\{0,1,\dots,n\}$.  Then $P(M)$ is a matroid polytope.
\end{proposition}

\begin{proof}
We sketch the proof. Clearly every vertex of $P(M)$ is a vertex of $\Delta(d+1,n+1)$.  Let $L$ be the line segment between $e_{B_1}$ and $e_{B_2}$.  If we have $|B_1\triangle B_2|\neq 2$, we must show that $L$ is not an edge of $P(M)$.
By the repeated use of the basis exchange axiom, one can show that the midpoint of the edge $\frac{e_{B_1}+e_{B_2}}{2}$ lies in the convex hull of other vertices of the matroid polytope.
\end{proof}

We give the proof of the converse following \cite{GGMS}.

\begin{theorem} \label{t:matroidpolytope} Let $P$ be a matroid polytope.  Then there exists a unique matroid $M$ such that $P=P(M)$.
\end{theorem}

\begin{proof}
Let $\cB$ be the set of $(d+1)$-element subsets of $E=\{0,1,\dots,n\}$ such that $\{e_B\mid B\in \cB\}$ is the set of vertices of $P$.   We will show that $\cB$ is the collection of bases for a matroid by verifying the basis exchange axiom.

Let $B_1, B_2\in \cB$.  Let $i\in B_1\setminus B_2$.  We must find $j\in B_2\setminus B_1$ such that $(B_1\setminus i)\cup {j}\in \cB$.  Let $B'_1,\dots,B'_k$ be the bases such that $\{e_{B_1}e_{B'_l}\}_l$ are the edges of $P$ at $e_{B_1}$.  If $B_2$ is among the $B'_l$'s, then we are done because then $|B_1\triangle B_2|=2$.
Otherwise, because the midpoint between $e_{B_1}$ and $e_{B_2}$ is contained in $P$, it is contained in the convex cone at $e_{B_1}$ spanned by the edges containing $e_{B_1}$.  Therefore, we may write
\begin{eqnarray*}
\frac{e_{B_1}+e_{B_2}}{2}&=&e_{B_1}+\sum_l \lambda_l(e_{B'_l}-e_{B_1})
\end{eqnarray*}
or
\begin{eqnarray} \label{e:exchangecone}
\frac{e_{B_2}-e_{B_1}}{2}&=&\sum_l \lambda_l(e_{B'_l}-e_{B_1})
\end{eqnarray}
for $\lambda_l\geq 0$.  Extracting the coefficient of $e_i$, we get
\[-\frac{1}{2}=-\sum_{l \mid i\not\in B'_l} \lambda_l.\]
 Let $l'$ be such that $i\not\in B'_{l'}$ and $\lambda_{l'}>0$.  Let $j$ be the unique element of $B_{l'}\setminus B_1$.   
 Because the coefficient of $e_j$ on the right side of \eqref{e:exchangecone} is positive, we have that $j\in B_2\setminus B_1$.
\end{proof}

It turns out that the dimension of $P(M)$ is $n+1-\kappa(M)$ where $\kappa(M)$ is the number of connected components of $M$.  In fact, if $M$ is connected, then the dimension of $P(M)$ is $n$ for the following reasons: all vertices of the matroid polytope lie in the hyperplane described by $x_0+x_1+\dots+x_n=d+1$ corresponding to the fact that all of the bases have the same number of elements; and for all pairs $i\neq j$, $e_i-e_j$ is parallel to an edge of $P(M)$.  If $E$ is partitioned as $E=E_1\sqcup\ldots\sqcup E_\kappa$ where $M=M_1\oplus\ldots\oplus M_\kappa$ and $M_i$ is a matroid on $E_i$, then $P(M)=P(M_1)\oplus\ldots\oplus P(M_\kappa)\subset\R^{E_1}\oplus\ldots\oplus\R^{E_\kappa}$.

\subsection{Faces of the matroid polytope}

In this section, we study faces of the matroid polytope following \cite{AKBergman}.  They correspond to bases maximizing a certain weight vector.  Further details can be found in \cite{FeichtnerS}.  We first review faces of polytopes. 

\begin{definition}
Let $P$ be a convex polytope in $\R^{n+1}$.  For $w\in \R^{n+1}$, let $P_w$ be the set of points $x\in P$ that minimize the function $\<x,w\>$.  
\end{definition}

The set $P_w$ is a face of the polytope $P$. 
 
 Observe that any face of the matroid polytope is itself a matroid polytope.  We will identify its matroid.
  
\begin{definition} Let $M$ be a matroid on $E=\{0,\ldots,n\}$ and let $w\in\R^E$.  For a basis $B\in\cB(M)$, let the $w$-weight of $B$ be
\[w_B=\sum_{i\in B} w_i.\]
Let $M_w$ be the matroid whose bases are the bases of $M$ of minimal  $w$-weight.
\end{definition}

We know that $M_w$ is a matroid for the following reason: the convex hull of the characteristic polynomial of its bases is the matroid polytope $P(M)_w$ allowing us to apply Theorem \ref{t:matroidpolytope}.

We give a description of the matroid $M_w$ following \cite{AKBergman}.  Let $S_\bullet$ be the flag of subsets
\[S_\bullet=\{\emptyset=S_0\subsetneq S_1\subsetneq \ldots \subsetneq S_k\subseteq F_{k+1}=E\},\]
where $w$ is constant on $S_i\setminus S_{i-1}$ and $w|_{S_i\setminus S_{i-1}}$ is strictly increasing with $i$.
\begin{lemma} We have the following equality:
\[M_w=\bigoplus_{i=1}^{k+1} (M|_{S_i})/S_{i-1}.\]
\end{lemma}

\begin{proof}
Every basis of minimal $w$ weight can be constructed by the greedy algorithm.  In others words, it will start with the empty set, add elements of $S_1$ as long as the result is independent, move on to $S_2$, and so on.
Consequently, a basis of $M_w$ consists of $r(S_i)-r(S_{i-1})$ elements of $S_i\setminus S_{i-1}$ for all $i$.   Such a set of $r(S_i)-r(S_{i-1})$ elements is exactly a basis of $(M|_{S_i})/S_{i-1}$.
\end{proof}

We will make use of this description in Section \ref{s:bergman} when we discuss the Bergman fan of a matroid.

\subsection{Valuative invariants}

The polytope definition of matroids allows one to think about invariants of matroids in a different light.  One can study invariants of matroids that are well-behaved under subdivisions of the matroid polytope. This is the approach taken by Fink and Derksen \cite{Derksen,DerksenFink}. 

\begin{definition} A {\em matroidal subdivision} of a matroid polytope $P(M)\subset\R^{n+1}$ is a polyhedral complex $\mathcal{D}$ whose cells are polytopes in $\R^{n+1}$ such that
\begin{enumerate}
\item Each cell of $\mathcal{D}$ is a matroid polytope, and 
\item $\bigcup_{P\in \mathcal{D}} P=P(M)$
\end{enumerate}
\end{definition}

\begin{example} As an example of a matroid subdivision, consider $U_{2,4}$ whose matroid polytope is an octahedron.  It has a matroid subdivision into two pyramids that meet along the square whose vertices are $\{(1,0,0,1),(1,0,1,0),(0,1,1,0),(0,1,0,1)\}$ \cite{SpeyerLinear}.  In fact, let $M_{\operatorname{top}}$ be the rank $2$ matroid on $E=\{0,1,2,3\}$ where $0$ and $1$ are taken to be parallel but all other two element subsets are bases.  Its matroid polytope is the top pyramid whose vertices are $\{(0,0,1,1),(1,0,0,1),(1,0,1,0),(0,1,1,0),(0,1,0,1)\}$.  Similarly, the bottom pyramid is the matroid polytope corresponding to $M_{\operatorname{bot}}$ where $2$ and $3$ are mandated to be parallel.  The middle square is the matroid polytope of $M_{\operatorname{mid}}$ where $0$ and $1$ are mandated to be parallel as are $2$ and $3$ while the bases are $\{0,2\},\{1,2\},\{0,3\},\{1,3\}$.
\end{example}

Write $P_1,\dots,P_k$  for the top-dimensional cells of $\mathcal{D}$.  For $J\subseteq \{1,\dots,k\}$, let $P_J=\bigcap_{j\in J} P_j$.  Let $\mathcal{M}$ be the set of all matroid polytopes in $\R^{n+1}$ and let $A$ be an abelian group.

\begin{definition}  A function $f:\mathcal{M}\rightarrow A$ is a {\em matroid valuation} if for all matroidal subdivisions of a matroid polytope $P(M)$,
\[f(P(M))=\sum_{\emptyset\neq J\subseteq\{1,\dots,k\}} (-1)^{|J|-1}f(P_J).\]
\end{definition}

In other words, $f$ obeys an inclusion/exclusion relation for matroidal subdivisions.  Consequently,  valuative invariant of lattice polytopes like the Ehrhart polynomial immediately give valuative invariants of matroids.
When the subdivision is regular, the matroidal subdivision can be interpreted as a tropical linear subspace \cite{SpeyerLinear} which is a combinatorial abstraction of a  family of linear subspaces defined over a disk that degenerates into a union of linear subspaces.

The Tutte polynomial along with a quasi-symmetric function-valued invariant introduced by Billera-Jia-Reiner \cite{BilleraJiaReiner} are valuative.  Derksen \cite{Derksen} introduced a valuative invariant which was proved to be universal among valuative invariants \cite{DerksenFink}. It takes values in the ring of quasi-symmetric functions, whose underlying vector space has a basis indexed by compositions $\underline{\alpha}=(\alpha_1,\alpha_2,\dots,\alpha_l)$.  Let $U_{\underline{\alpha}}$ be some basis for the ring of quasi-symmetric functions over $\Z$.  Derksen's invariant $\mathcal{G}$ is defined as
\[\mathcal{G}=\sum_{\underline{E}} U_{r(\underline{E})}\]
where $\underline{E}=\{\emptyset=E_0\subset E_1\subset E_2\subset\dots\subset E_{n+1}=E\}$
ranges over all $(n+1)!$ maximal chains of subsets of $E$ and 
\[r(\underline{E})=(r(E_1)-r(E_0)+1,r(E_2)-r(E_1)+1,\dots,r(E_{n+1})-r(E_n)+1).\]
This invariant specializes to any valuative invariant under homomorphisms from quasi-symmetric functions to an abelian group \cite{DerksenFink}.  This is proved by showing that Schubert matroids (after allowing permutations of the ground set) are a basis for the dual space to valuative invariants.

\section{Grassmannians and Matroid Polytopes}
\subsection{Pl\"{u}cker coordinates and basis exchange}

The Grassmannian $\Gr(d+1,n+1)$ is an algebraic variety whose points correspond to the $(d+1)$-dimensional linear subspaces of $\k^{n+1}$.  Equivalently, it is the set of all $d$-dimensional projective subspaces of $\P^n$.  It generalizes projective space in the sense that $\Gr(1,n+1)$ is isomoprhic to $\P^n$.

We present the Grassmannian as a quotient of the Stiefel variety following \cite{GKZ}.  The Stiefel variety $\St(d+1,n+1)$ parameterizes  linearly independent $(d+1)$-tuples of vector in $\k^{n+1}$.  We denote their span as the $(d+1)$-dimensional subspace $P\subset \k^{n+1}$ and write the $(d+1)$-tuple $(v_1,\dots,v_n)$ as a $(d+1)\times(n+1)$-matrix $X$.  Now, $\Gl_{d+1}$ acts on the $(d+1)$-tuple without changing $P$.  This is the same thing as the left-action of $\Gl_{d+1}$ on the matrix.  Because the action is free and acts transitively on bases of $P$, the Grassmannian is the quotient
\[\Gr(d+1,n+1)=\St(d+1,n+1)/\Gl_{d+1}.\]
There are homogeneous coordinates, the Pl\"{u}cker coordinates that embed this quotient into projective space: let $B$ be a $(d+1)$-tuple of distinct  elements of $\{1,\dots,n+1\}$, the Pl\"{u}cker coordinate $p_B$ is the determinant of the $(d+1)\times(d+1)$-matrix formed by the columns indexed by $B$.  The Pl\"{u}cker coordinates are indeed coordinates on the Grassmannian.  Indeed, because $X$ is rank $d+1$, we may permute the columns of $M$ so that the first $d+1$ columns form a non-singular matrix.  By applying an element of $\Gl(d+1)$, we may suppose $X$ is the form 
\[X=\left[I_{d+1}|A\right]\]
where $A$ is a $(d+1)\times (n-d)$-matrix.  Let $B$ consist of $d$ elements of $\{1,\dots,d+1\}$ together with an element $i$ of $\{d+2,\dots,n+1\}$.  Then $p_B$ is equal (up to sign) to an entry of the $i$th column of $X$.  Therefore, from the Pl\"{u}cker coordinates, we can recover $X$.  

Now, because if $B$ and $B'$ are related by a permutation, $p_B$ and $p_{B'}$ are equal up to a sign, it suffices to consider only the Pl\"{u}cker coordinates corresponding to $B=(i_1,\dots,i_{d+1})$ with $i_1<i_2<\dots<i_{d+1}$.  Therefore, there are $\binom{n+1}{d+1}$ Pl\"{u}cker coordinates.
The Pl\"{u}cker coordinates give an embedding
\[\Gr(d+1,n+1)\hookrightarrow \P^{\binom{n+1}{d+1}-1}.\]
The Pl\"{u}cker coordinates obey natural quadratic relations, called the Pl\"{u}cker relations: for $1\leq i_1<\dots<i_d\leq n+1$, $1\leq j_1<\dots< j_{d+2} \leq n+1$, we have, for every point of the Grassmannian,
\[\sum_{a=1}^{d+1} (-1)^a p_{i_1\dots i_d j_a} p_{j_1\dots \hat{j}_a\dots j_{d+2}}=0\]
where $\hat{j}_a$ means $j_a$ is deleted.  Moreover, these relations generate the ideal defining $\Gr(d+1,n+1)$.

Given a point of the Grassmannian whose Pl\"{u}cker coordinates are $(p_B)$, we can define a matroid by setting the bases to be
\[\cB=\left\{B\in \binom{\{1,\dots,n+1\}}{d+1}\ \Bigg |\ p_B\neq 0\right\}.\]
\begin{lemma} \label{l:grassmatroid} The set $\cB$ defined above is the set of bases for a matroid.
\end{lemma}

\begin{proof}
Because the matrix $X$ is of rank $d+1$, the set $\cB$ is non-empty.  We need only show that it satisfies the basis exchange axiom.  Let $B_1,B_2\in \cB$, $i\in B_1\setminus B_2$.  We must show that there is an element $j\in B_2\setminus B_1$ such that $(B_1\setminus i)\cup \{j\}\in\cB$.  
Write $B_1\setminus\{i\}=\{i_1,\dots,i_d\}$ and $B_2\cup \{i\}=\{j_1,\dots,j_{d+2}\}$.  The left side of the Pl\"{u}cker relation for this choice of $i$'s and $j$'s has $p_{B_1}p_{B_2}$ as a summand.  Because this monomial is nonzero, there must be another nonzero monomial of the form $p_{(B_1\setminus\{i\})\cup \{j\}}p_{(B_2\setminus\{j\})\cup\{i\}}$.
\end{proof}

We can view this matroid as coming from a vector configuration.  Specifically, pick a matrix $X\in \St(d+1,n+1)$ representing that linear subspace.  Take the vectors to be the columns of $X$.  The set of $d+1$ vectors is a basis for the matroid if and only if the corresponding determinant is nonzero.  This occurs exactly when these vectors give a basis of $\k^{d+1} $. 

We note that $\Gr(d+1,n+1)$ is isomorphic to $\Gr(n-d,n+1)$.  Specifically if $V\subset \k^{n+1}$ is a $d$-dimensional subspace, there is a map dual to inclusion $i^*:(\k^{n+1})^*\rightarrow V^*$.  The map is surjective so its kernel is $(n-d)$-dimensional.  This isomorphism is analogous to matroid duality.

\subsection{Projective toric varieties}

We give a review of not-necessarily-normal (or Sturmfeldian) projective toric varieties.  They are important for relating the matroid polytope to the Grassmannian  For more details, see \cite{CLS}.  Let $\cA\subset \Z^{n+1}$ be a finite subset.  Write $\cA=\{\omega_0,\omega_1,\dots,\omega_N\}$.  We can define a morphism $i:(\k^*)^{n+1}\rightarrow \P^N$ by
\begin{eqnarray*}
x &\mapsto &[x^{\omega_0}:\ldots:x^{\omega_N}]
\end{eqnarray*}
where if $x=(x_0,\dots,x_n)$ and $\omega_i=(\omega_{i0},\dots,\omega_{in})$,
\[x^{\omega_i}=x_0^{\omega_{i0}}\dots x_{n}^{\omega_{in}}.\]
Let $A$ be the $(n+1)\times (N+1)$-matrix whose columns are the $\omega$'s.  If $A$ is unimodular and full rank, then $i$ is an inclusion.  
The toric variety $\P_\cA$ associated to $\cA$ is the closure $\overline{i((\k^*)^{n+1})}$.  The weight polytope of $X_\cA$ is the convex hull of $\cA$.  The torus orbits in $\P_\cA$ are in bijective correspondence with the faces of the weight polytope.

This definition can be extended to include closures of torus orbits in projective spaces.  Suppose that $\k$ is an algebraically closed field.  Let $H=(\k^*)^{n+1}$ act on a vector space $V$.  There is a character decomposition 
\[V=\bigoplus_\chi V_\chi\]
where for a character $\chi$ of $H$,
\[V_\chi=\{v\in V \mid t\cdot v=\chi(t)v \text{\ for all\ } t\in H\}.\]
Let $v\in \P(V)$.  We will study the $H$-orbit closure $\overline{H\cdot v}\subseteq \P(V)$.  Lift $v$ to an element of $\tilde{v}\in V$, and write 
\[\tilde{v}=\sum_\chi \tilde{v}_\chi\]
for $\tilde{v}_\chi\in V_\chi$.  Let $\cA(v)=\{\chi\mid \tilde{v}_\chi\neq 0\}$.  By unpacking definitions, we have $\overline{H\cdot v}\cong X_{\cA(v)}$.  The \emph{weight polytope} of $v$ is the convex hull of $\cA(v)$ in $H^\wedge\otimes \R$ where $H^\wedge$ is the character lattice of $H$.  If the image of $v$ in $\P(V)$ is stabilized by a subtorus $H'\subseteq H$, there is an induced linear map $H^\wedge\rightarrow (H')^\wedge$ and the weight polytope lies in a translate of $\ker(H^\wedge\rightarrow (H')^\wedge)$.

\subsection{Torus orbits and matroid polytopes}

In this section, we consider a torus acting on the Grassmannian.  Note that the algebraic torus  $(\k^*)^{n+1}$ acts on $\P^n$ by dilating the homogeneous coordinates.  This induces an action on $\Gr(d+1,n+1)$: for $t\in (\k^*)^{n+1}$, $P\subset \P^n$, a $d$-dimensional subspace, $t\cdot P$ is also a $d$-dimensional subspace.  We can see this action in terms of the Stiefel variety:  $(\k^*)^{n+1}$ dilates the columns of the matrix $X$.  Note that the diagonal $\k^*$ acts trivially on the Grassmanian since it preserves any  subspace $P$.

For $P\in \Gr(n-d,n)$, we may consider the closure of the $(\k^*)^{n+1}$-orbit of $P$, 
\[\P_P=\overline{(\k^*)^{n+1}\cdot P}.\]
As observed in \cite[Prop~12.1]{SpeyerMatroid}, this orbit closure is a normal toric variety by arguments of White \cite{White77}.    Now, by Lemma \ref{l:grassmatroid}, we may produce a rank $d+1$ matroid $M$ on $\{0,1,\dots,n\}$ from $P$.

We have the following combinatorial description of the toric variety by Gelfand-Goresky-MacPherson-Serganova \cite{GGMS}:

\begin{theorem} The weight polytope of $\P_P$ is the matroid polytope $P(M)$.
\end{theorem}

\begin{proof}
The $(\k^*)^{n+1}$-action extends to the ambient space of the Pl\"{u}cker embedding $\P^{\binom{n+1}{d+1}-1}$.  In fact, by the matrix description of the group action, if $t\in (\k^*)^{n+1}$ is given by $(t_0,\dots,t_{n})$ then
\[t\cdot p_B=\left(\prod_{i\in B} t_i\right) p_B\]
Consequently, $p_B$ lies in the character space $e_B$ given by the characteristic vector of $B$.     The weight polytope is the convex hull of $e_B$ for $B\in\cB$.  
\end{proof}
Note that the weight polytope lies in an affine hyperplane corresponding to the fact that all bases have $d+1$ elements.  This reflects the fact that the diagonal torus of $(\k^*)^{n+1}$-stabilizes $P\in\Gr(d+1,n+1)$.

\subsection{Thin Schubert cells}

We can use matroids to define thin Schubert cells, locally closed subsets of the Grassmannian that refine Schubert cells.   A thin Schubert cell consists of all subspaces with a given matroid.   Thin Schubert cells were introduced in \cite{GGMS} by Gelfand-Goresky-MacPherson-Serganova.  Prior to their introduction in the algebraic geometric context, White \cite{WhiteBracket} studied their coordinate rings.
Following the notation of \cite{GGMS}, we will think of matroids as vector configuration given by the  projection of the coordinate vectors.

The closure of the usual Schubert cells, called Schubert varieties, can be described by the vanishing of some Pl\"{u}cker coordinates.  The data of the matroid corresponds to imposing exactly which Pl\"{u}cker coordinates vanish.  We will use this data to define a thin Schubert cell.

\begin{definition} Let $M$ be a rank $d+1$ matroid on $E=\{0,\dots,n\}$.  The \emph{thin Schubert cell} $\Gr_{M}$ in $\Gr(d+1,n+1)$ is the locally closed subset given by 
\[\Gr_M=\{P\in \Gr(d+1,n+1) \mid p_B(P)\neq 0\ \text{if and only if $B$ is a basis of $M$}\}.\]
\end{definition}

Another convention is to consider the thin Schubert cell associated to $M$ in $\Gr(n-d,n+1)=\Gr(d+1,n+1)$ as follows: for an $(n-d)$-dimensional subspace $P\in \Gr(n-d,n+1)$, there is a projection $\pi_P:\k^{n+1}\rightarrow\k^{n+1}/P$.  For $S \subset\{0,\dots,n\}$, 
let $\k^S$ be the linear space spanned by $e_i$ for $i\in S$.  
We may consider the matroid given by the vector configuration $\pi_P(e_0)\dots\pi_P(e_n)$ which has the rank function $r(I)=\dim(\pi_P(\k^I))$.  

 We can describe the stabilizer of $\Gr_M$ under the torus action in terms of $M$.  Let $M=M_1\sqcup\ldots\sqcup M_\kappa$ be a connected component decomposition of $M$ where where each $M_i$ is a matroid on $F_i$ for a partition $E=F_1\sqcup \ldots \sqcup F_\kappa$.  Then we have a decomposition of $P$,
\[\k^{n+1}=\k^{F_1}\oplus\ldots\oplus \k^{F_\kappa}\] 
given by
\[P=P_1\oplus\ldots\oplus P_{\kappa}\]
for $P_i\subseteq \k^{E_i}$.  We have an inclusion $(\k^*)^{F_i}\hookrightarrow \k^{n+1}$. 
The image of the diagonal torus in $(\k^*)^{F_i}$ stabilizes $P_i$.  One can show that the stabilizer of $P$ is $(\k^*)^\kappa$ corresponding to the direct product of these diagonal tori.

The thin Schubert cells do not provide a stratification of the Grassmannian.  In fact, the closure of a thin Schubert cell is not always a finite union of thin Schubert cells. 
However, thin Schubert cells do occur as the intersection of finitely many Schubert cells with respect to different flags.

We have the following proposition which describes how the thin Schubert cels behave under direct sums of matroids:

\begin{proposition}  If the matroid $M$ has a unique decomposition into connected components 
\[
M = M|_{F_0} \oplus \cdots \oplus M|_{F_c},
\]
for flats, $F_i$, then  there is a natural isomorphism
\[\Gr_M\cong\Gr_{M|_{F_0}}\times\dots\times\Gr_{M|_{F_c}}\]
\end{proposition}

\subsection{Realization spaces}

Given a simple matroid $M$ of rank $d+1$ on $\{0,1,\dots,n\}$, 
we can define \emph{realization spaces}, space whose points correspond to representations of $M$.  These spaces exist as moduli spaces, but  we will focus our attention on their $\k$-points and treat them as parameter spaces.  The results of this subsection, however, can be stated in moduli-theoretic language.

Thin Schubert cells can be related to a universal family of representations.  Here, we say that  a $d$-dimensional subspace $V\subset\P^n$ represents $M$ if the matroid associated to $V$ is $M$.
Recall that the Grassmannian $\Gr(d+1,n+1)$ has a universal family of subspaces over it: there is a subvariety $\mathcal{V}\subset \P^n_{\Gr(d+1,n+1)}$ that fits in a commutative diagram
\[\xymatrix{
\mathcal{V}\ar[d]\ar@{^{(}->}[r]&\P^n_{\Gr(d+1,n+1)}\ar[dl]\\
\Gr(d+1,n+1)}\]
such that $\mathcal{V}$ is flat over $\Gr(d+1,n+1)$ and the fiber $\mathcal{V}_V$ for $V\in \Gr(d+1,n+1)$ is the subspace of $\P^n$ corresponding to $V$.
 
 \begin{definition} A family of representations of $M$ over a scheme $X$ is a flat family $V\rightarrow X$ of $d$-dimensional subspaces in $\P^n_X$ such that for every closed point $x$, the fiber $V_x$ is a representation of $M$.
\end{definition}

The base-change $\mathcal{V}_{\Gr^M}$ over the thin Schubert cell $\Gr^M$ is a family of realizations.  Its $\k$-points, $\Gr^M(\k)$ are exactly the representations of $M$ over $\k$. 

One can also consider the space of representations of $M$ as a hyperplane arrangement.  Recall that if $V\subset\P^n$ is a representation of $M$ as a subspace, intersecting $V$ with each coordinate hyperplane gives an arrangement of $n+1$ hyperplanes on $\P^d\cong V$ representing $M$.  We can alternatively consider an arrangement of $n+1$ hyperplanes on $\P^d$ by considering an $(n+1)$-tuple of non-zero linear forms on $\P^d$ up to constant multiples.  Let $\Gamma=\A^{d+1}\setminus\{0\}$.  Therefore, an $n+1$-tuple of non-zero linear forms $(s_0,s_1,\ldots,s_n)$ is described by a point of $\Gamma^{n+1}$.  

\begin{definition} The \emph{linear form realization space} for $M$ is the subscheme $\Gamma_M$ of $\Gamma^{n+1}$ consisting of hyperplane arrangements realizing $M$.
\end{definition}

Two tuples of sections cut out the same family of hyperplane arrangements if and only if they differ by an element of $(\k^*)^{n+1}$ scaling the sections $(s_0,s_1,\ldots,s_n)$ and two families of hyperplane arrangements are isomorphic, by definition, if they differ by an element of $\PGL_{d+1}$.
These actions commute, and the product $\PGL_{d+1} \times (\k^*)^{n+1}$ acts freely on $\Gamma_M$. 

The quotient
\[
\mathbf R_M = \Gamma_M / (\PGL_{d+1} \times (\k^*)^{n+1})
\]
is called the \emph{realization space} of the matroid $M$.
Points of this realization space correspond to isomorphism class of representations of $M$ as hyperplane arrangements.

The thin Schubert cells are naturally torus torsors over realization spaces of hyperplane arrangements which we will now describe using the subtori $(\k^*)^{F_i}$ from the above subsection  following \cite{KPReal}, which summarizes the discussion in \cite[Sec 1.6]{Lafforgue03}.

\begin{proposition} \cite[1.6]{Lafforgue03} 
If $M$ is connected then $\mathbf{R}_M$ is the quotient of $\Gr_M$ by the free action of $(\k^*)^{n+1}/\k^*$.  If $M$ is not connected, there is a natural isomorphism
\begin{eqnarray*}
\mathbf{R}_M&\cong&\mathbf{R}_{M|_{F_0}}\times\dots\times \mathbf R_{M|_{F_\kappa}}
\end{eqnarray*}
In this case, the $(\k^*)^{n+1}/(\k^*)$-action naturally factors through the projection
\[(\k^*)^{n+1}/\k^*\rightarrow (\k^*)^{F_0}/\k^*\times\dots\times(\k^*)^{F_\kappa}/\k^*.\]
\end{proposition}

\subsection{Matroid representability and universality} \label{ss:mnev}

In a certain sense, the question of the representability of matroids is maximally complicated as a consequence of  Mn\"{e}v's universality theorem which states that any possible singularity defined over $\Z$ occurs on some thin Schubert cell (up to a natural equivalence).   This has applications to moduli problems and logic.  Here we state Mn\"{e}v's theorem in the form applied by Lafforgue \cite{Lafforgue03}.

\begin{theorem}
Let $X$ be an affine scheme of finite type over $\Spec \Z$.  Then there exists a connected rank $3$ matroid $M$, an integer $N\geq 0$ and an open set $U\subseteq X\times \A^N$ projecting onto $X$ such that $U$ is isomorphic to the realization space $\mathbf R_M$.
\end{theorem}

A very brief proof is given in \cite{Lafforgue03}.  More details in a slightly different context can be found in \cite{LeeVakil}.  An accessible treatment that only considers the set-theoretic case is \cite{RichterGebertUniversality}.  The central idea of most proofs of this theorem, following Shor \cite{Shor}, is to add auxiliary variables and present $X$ as being cut out by a very large system of equations each involving three variables and a single arithmetic operation.  This new system is encoded as a line arrangement by using the von Staudt constructions (see \cite{RichterGebertUniversality} for an illustration).

This theorem quickly gives interesting counterexamples.  If $X$ is disconnected, then the realization space $\mathbf R_M$ will be disconnected and the matroid will have representations (as hyperplane arrangements) that cannot be continuously deformed into each other.  One can also control the fields over which $X$ has a representation:  a matroid $M$ is representable over $\k$ if and only $\mathbf R_M$ has a $\k$-point.  If $X=\Spec \F_2$ and $\k$ is some infinite field, then $X$ and consequently $\mathbf{R}_M$ only has $\k$-points if and only if $\charact \k=2$.  Similarly, if $X=\Spec \Z[x]/(2x-1)$, then $\mathbf{R}_M$ only has $\k$-points if and only if $\charact \k\neq 2$.   

As observed by Sturmfels \cite{SturmfelsDecidability}, a decision procedure to determine if a matroid is representable over $\Q$ is equivalent to a decision procedure to determine whether a system of polynomial equations has a solution in $\Q$.  The existence of the decision procedure is a famous unsolved problem, Hilbert's Tenth Problem over $\Q$.  See Poonen's survey \cite{PoonenSurvey} for more details about this problem. 

For finite fields $\k$, the situation is a bit different.  Because $\mathbf R_M$ is isomorphic to an open set $U\subseteq X\times\A^N$, and  it is possible for some $\k$-points to be missing from $U\subseteq X\times\A^N$, it is not necessarily true that the existence of $\k$-points in $X$ is equivalent to the representability of $M$ over $\k$.  Moreover, the number of points of $X(\k)$ and $\mathbf R_M(\k)$ may be drastically different.  For this reason, it's possible to have unique representability theorems similar to Theorem \ref{t:binaryunique} for finite fields. 

Mn\"{e}v's thorem is the primordial example of Murphy's Law in algebraic geometry.  Murphy's Law, due to Vakil \cite{Vakil06}, is a theorem that states that certain classes of moduli spaces contain every singularity type over $\Z$ up to a particular natural equivalence.  In fact, Vakil was able to prove that a large number of moduli spaces obey Murphy's Law by relating them to  realization spaces of matroids.

\subsection{Matroid representability and semifields}

One can view matroids as points of thin Schubert cells over field-like objects like blueprints or partial fields.  This is the point of view taken in Dress's matroids over fuzzy rings \cite{Dress} as studied by Dress with Wenzel \cite{DressWenzel} .  Here, we use the notation of Lorscheid's blueprints and blue schemes \cite{Lorscheid} for the following highly speculative subsection.

\begin{definition} A {\em blueprint} B is a multiplicatively-written monoid $A$ with neutral element $1$ and absorbing element $0$ together with an equivalence relation $\cR$ (given by $\equiv$) on the semiring $\N[A]$ of finite formal sums of elements of $A$ such that
\begin{enumerate}
\item The relation $\cR$ is closed under addition  and multiplication,
\item The absorbing element is equivalent to the empty sum, and
\item For $a,b\in A$ with $a\equiv b$, then $a=b$.
\end{enumerate}
\end{definition}

 For example, one can define an object called the field of one element $\F_1$ to be the blueprint with $A=\{0,1\}$ and $\cR$ to be the empty relation.  Moreover, given any semiring $R$, one can produce a blueprint by setting $A$ to be $R^\times$, where $R^\times$ is $R$, considered as a multiplicative monoid, and setting 
 \[\cR=\left\{\sum a_i \equiv \sum b_i\mid \sum a_i=\sum b_i \text{\ in\ } R\right\}.\]
 In particular every field and idempotent semiring can be interpreted as a blueprint.  An important blueprint is the none-some semifield $\B_1$ where we  set $A=\{0,1\}$ and $\cR=\{1+1\equiv 1\}$.  Here we imagine a primitive culture with addition and multiplication but only numbers for $0$ and more than $0$.  This is isomorphic to the min-plus semifield $\{0,\infty\}$ by taking $-\log$.   

For a blueprint $B$, we can view a linear subspace defined over $B$ as a $B$-point of $\Gr(d+1,n+1)$ provided that we can treat $\Gr(d+1,n+1)$ in its Pl\"{u}cker embedding as a model of $\Gr(d+1,n+1)$ over $B$.   While there are various rival notions for the definition of a $B$-point, one notion is an assignment of the Pl\"{u}cker coordinates to elements of $A$ such that Pl\"ucker relations are satisfied.   Under that definition, a matroid is a $\B_1$-point.  The value of a Pl\"ucker coordinate in $\B_1$ determines whether or not it is zero.  The Pl\"{u}cker relations in $\B_1$ become the basis exchange axiom.  In the case that $B$ is induced from a field $\k$, then a $B$-point of $\Gr(d+1,n+1)$ is the same thing as a $\k$-point.

Regular matroids should perhaps be thought of as $\F_{1^2}$-points of $\Gr(d+1,n+1)$.  Here $\F_{1^2}$, the degree $2$ cyclotomic extension of $\F_1$, is the blueprint generated by $A=\{0,1,-1\}$ subject to the relation $1+(-1)\equiv0$.  This is because regular matroids are those representable by totally unimodular matrices and therefore,  all the Pl\"{u}cker coordinates are $0,1,$ or $-1$.  
An important idea in the study of the field of one element is the numerical prediction that the number $\F_1$-points is the limit of the number of $\F_q$-points as $q\to 1$.  Is there a similar  prediction for the number of regular matroids?

Excluded minor characterizations of matroids can be thought of as matroid lifting results.  Specifically, we have a morphism of blueprints $B'\to B$, and a $B$-point of the Grassmannian, and we ask when this $B$-point is in the image of the induced map $\Gr(d+1,n+1)(B')\rightarrow \Gr(d+1,n+1)(B)$.
For example, there is a natural blueprint morphism $\F_{1^2}\rightarrow \F_2$.
If $M$ is a matroid that is representable over $\F_2$, the question of whether it is is regular corresponds to whether it lifts from a $\F_2$-point to a $\F_{1^2}$-point.  This is the point of view taken in the work of Pendavingh and van Zwam \cite{PVLifting} phrased in the language of partial fields which include the usual fields and $\F_{1^2}$. They are able to give unified proofs of a number of representability results including the following theorem of Tutte:
\begin{theorem} A matroid is regular if and only if it is representable over both $\F_2$ and $\F_3$.
\end{theorem}  
Here, they manipulate matrix representations of matroids by matrix operations, but it would be very interesting to rephrase their work in the language of deformation theory where, for examples, one tries to extend the matroid from the partial field $F_2\otimes F_3$ to $F_{1^2}$.  The category of partial fields is more interesting than that of fields because of the existence of tensor products and many more morphisms.

\subsection{$K$-theoretic matroid invariants}

Speyer \cite{SpeyerMatroid} introduced a way of producing valuative invariants of matroids by using the $K$-theory of the Grassmannian.  He begins with an invariant valued in the $K$-theory of the Grassmannian.  This invariant can be specialized by particular geometric operations to take values in less exotic rings.  In fact, it is the main result of \cite{FinkSpeyer} that this invariant specializes to the Tutte polynomial.

We first review algebraic $K$-theory.  See \cite[Sec. 15.1]{FultonIntersectiontheory} for a very brief summary.  For a scheme $X$, let $K_0(X)$ be the Grothendeick group of coherent sheaves on $X$.  It is the group of formal linear combinations of coherent sheaves on $X$ subject to the relation
\[[\cf]=[\cf']+[\cf'']\]
whenever there is an exact sequence of coherent sheaves
\[\xymatrix{
0\ar[r]&\cf'\ar[r]&\cf\ar[r]&\cf''\ar[r]&0.
}\]
Given a morphism of schemes $f:X\rightarrow Y$, there is a pushforward homomoprhism $f_*:K_0(X)\rightarrow K_0(Y)$ given by
\[f_*\cf=\sum_{i=0}^\infty (-1)^i[R^if_*\cf]\]
On the other hand, we may define $K^0(X)$ as the Grothendieck group of vector bundles on $X$, the group of formal linear combinations of vector bundles on $X$ subject to the relation
\[[E]=[E']+[E'']\]
whenever $E'$ is a sub-bundle of $E$ with quotient $E''=E/E'$.  
Given a moprhism $f:X\rightarrow Y$, there is a pullback homomorphism $f^*:K^0(Y)\rightarrow K^0(X)$ given by pullback of vector bundles.
The group $K^0(X)$ can be made into a ring by introducing tensor product as multiplication.
The natural group homomorphism $K^0(X)\rightarrow K_0(X)$ taking a vector bundle to its sheaf of sections is an isomorphism if $X$ is a non-singular variety.  We will freely switch between $K_0(X)$ and $K^0(X)$ as all of our schemes will be non-singular varieties.  We will make use of the fact that $K_0(\P^n)=\Q[x]/(x^{n+1})$ where $x=1-\O(-1)$ which corresponds to the structure sheaf of a hyperplane.

Let  $T=(\k^*)^{n+1}$ be an algebraic torus acting on a scheme $X$.   We may define equivariant $K$-groups $K_0^T(X)$ and $K^0_T(X)$ by considering equivariant coherent sheaves and vector bundles.  There are non-equivariant restriction maps
\[K_0^T(X)\rightarrow K_0(X),\ K^0_T(X)\rightarrow K^0(X)\]
forgetting the $T$-action.  
Now, the equivariant $K$-theory of a point is particularly simple: $K_T^0(\operatorname{pt})$ is the Grothendieck group of representations of $T$.  By taking characters, we see $K_T^0=\Z[t_0^\pm,\dots,t_n^\pm]$, the ring of Laurent polynomials.

The Grassmannian has a natural group action which allows us to consider equivariant $K$-theory \cite{KnutsonRosu}.  By localization, this equivariant $K$-theory has a quite combinatorial flavour.  Let $T$ act on $\Gr(d+1,n+1)$ as follows: $t\in T$ dilates the coordinates of ambient $\k^{n+1}$ taking a subspace $P\in \Gr(d+1,n+1)$ to $t\cdot P$.  Note that the diagonal subtorus in $T$ acts trivially.  The equivariant $K$-theory of the Grassmannian can be expressed in terms of its fixed points and $1$-dimensional orbits under $T$.   The following is easily verified:

\begin{lemma}
The fixed points of the $T$-action on $\Gr(d+1,n+1)$ are in one-to-one correspondence with subsets $B\subset\{0,1,\dots,n\}$ with $|B|=d+1$ with the fixed points $x_B$ given by
the $(d+1)$-dimensional subspaces of the form $\Span(e_i\mid i\in B)$.  The $1$-dimensional torus orbits of the action are in one-to-one correspondence with pairs $B_1,B_2\subset  \{0,1,\dots,n\}$ with $|B_1|=|B_2|=d+1$ and $|B\triangle B'|=2$ with the closure of the torus orbit parameterized by $\P^1$ as
\[[s:t]\mapsto \Span(e_i'\mid i'\in B_1\cap B_2)+\k(se_i+te_j)\]
where $i\in B_1\setminus B_2$ and $j\in B_2\setminus B_1$.
\end{lemma}
Because $\Gr(d+1,n+1)$ is a smooth projective variety, has finitely many fixed points, and the closure of each $1$-dimensional torus orbits is isomorphic to $\P^1$, $\Gr(d+1,n+1)$ is equivariantly formal \cite[Cor 5.12]{VV}, \cite[Cor A.5]{KnutsonRosu} which means that its equivariant $K$-theory is controlled by the fixed points and $1$-dimensional orbits.  Specifically, the restriction map to the fixed points
\[K^0_T(\Gr(d+1,n+1))\rightarrow K^0_T(\Gr(d+1,n+1)^T)=\bigoplus_B K^0_T(x_B)\]
is an injection whose image is given by $f_B\in K_0^T(x_B)$ for all $B$ satisfying
\[f_{S\cup \{i\}}\equiv f_{S\cup \{j\}}\ \operatorname{mod}\ 1-t_i/t_j\]
where $S\in \binom{\{0,\dots,n\}}{d}$, $i,j\not\in S$.

Now, we describe Speyer's matroid invariant.  Let $P\in \Gr(d+1,n+1)$.  The invariant is produced from the $K$-theory class of the orbit closure $\overline{T\cdot P}$, 
\[\O_{\overline{T\cdot P}}\in K_0(\Gr(d+1,n+1)).\]
Because $\overline{T\cdot P}$ is $T$-invariant, its structure sheaf is $T$-equivariant.  The corresponding equivariant $K$-theory class, $y(P)\in  K^T_0(\Gr(d+1,n+1))$ can be given in terms of the matroid polytope.  We will give the description of $y(P)_B\in K_0^T(x_B)$.  For $B$, a basis of the matroid, let $\Cone_B(P(M))$ be the cone of $P(M)$ at $B$, in other words the real non-negative span of all vectors of the form $v-e_B$ for $v\in P(M)$.  The Hilbert series of a cone $C$ is given by 
\[\Hilb(C)=\sum_{a\in C\cap \Z^{n+1}} t^a.\]
We define $y(P)$ by
\[y(P)_B=\begin{cases} \Hilb(\Cone_B(P(M)))\prod_{i\in B}\prod_{j\not\in B} (1-t_i^{-1}t_j) &\text{if $B$ is a basis of $P$}\\
0 & \text{else}
\end{cases}\]
The rational function $y(P)_B$ turns out to be a Laurent polynomial.  Note that $y(P)$ depends only on the matroid of $P$ and so can be written $y(M)$.  This definition gives a well-defined class of $K_0^T(\Gr(d+1,n+1))$ even when $M$ is not representable.  This invariant turns out to be valuative although it is not universal among valuative invariants.

The invariant $y(M)$ specializes to the Tutte polynomial in a geometric fashion.  Recall that $\Fl(d_1,\dots,d_r;n+1)$ is the flag variety whose points parameterize flags of projective subspaces $F_{d_1-1}\subset \ldots\subset F_{d_r-1}\subset \P^n$ where $\dim F_i=i-1$.    There is a natural inclusion 
\[\Fl(1,n;n+1)\hookrightarrow \P^n\times\P^n\]
 taking $F_1\subset F_{n-1}$ to $(F_1,F_{n-1})$ where the two $\P^n$'s parameterize points and hyperplanes in $\P^n$.
There is a diagram (borrowed from \cite{FinkSpeyer}) of morphisms forgetting various stages of the flags:
\begin{equation}\xymatrix{
 & \Fl(1,d+1,n; n+1)\ar[dl]_{\pi_{d+1}}\ar[dr] \ar[ddr]_{\pi_{1n}} & \\
\Gr(d+1,n+1) & & \Fl(1,n;n+1)\ar@{^(->}[d] \\
 & & \P^{n} \times \P^{n}
}\end{equation}
If we write $K^0(\P^{n}\times\P^n)=\Q[x,y]/(x^{n+1},y^{n+1})$, one has
\begin{theorem}\cite{FinkSpeyer} Let $\O(1)$ be the pullback of $\O(1)$ on $\P^{\binom{n+1}{d+1}-1}$ to $\Gr(d+1,n+1)$ by the Pl\"{u}cker embedding.  Interpreting $y(M)$ as a non-equivariant class, we have
\[(\pi_{1n})_*\pi_{d+1}^*(y(M)\cdot [\O(1)])=T_M(x,y).\]
\end{theorem}
The proof works by relating the $K$-theoretic class to the rank-generating polynomial.

\section{Review of Toric Varieties}

We give a quick review of toric varieties in preparation for the next section.  More complete references are \cite{CLS,FultonInt}.

\subsection{Rational fans and Minkowski weights}
A toric variety $X=X(\Delta)$ is an algebraic variety specified by a rational fan $\Delta$ in $N_\R=N \otimes_\Z \R$ for a lattice $N \cong \mathbb{Z}^n$. They are normal varieties compactifying the algebraic torus $T=(\k^*)^n$ such that the natural multiplication of $T$ on itself extends to a $T$-action.  
A rational fan is  a particular set $\Delta$ of strongly convex rational polyhedral cones.  A strongly convex rational polyhedral cone is the non-negative span of finitely many vectors in $N$ and which contains no lines through the origin.  A rational fan is defined as a polyhedral complex whose cells are strongly convex rational polyhedral cones, that is:
\begin{enumerate}
\item if $\sigma\in\Delta$, then every face of $\sigma$ is an element of $\Delta$, and
\item if $\sigma,\sigma'\in\Delta$, then $\sigma\cap \sigma'\in\Delta$.
\end{enumerate}
We say that a fan is pure of dimension $d$ if every maximal cone is $d$-dimensional.
For a cone $\sigma$, let $\sigma^\circ$ denote the relative interior of $\sigma$.
Associated to a fan $\Delta$ is a toric variety $X(\Delta)$ defined over a field $\k$.  

The toric variety is stratified by torus orbits.   Write $N^\vee$ for the lattice dual to $N$.  Write $N_\sigma$ for the sublattice of $N$  given by by $\Span(\sigma) \cap N$.
For a cone $\sigma$, let $\sigma^\perp\subseteq N^\vee$ be the kernel of the projection $N^\vee\rightarrow N_\sigma^\vee$ that is dual to the natural inclusion.  
Given a cone $\sigma$ with $\dim(\sigma)=k$, there is a torus orbit $\O_\sigma$ with $\dim(\O_\sigma)=n-k$.  This torus orbit is canonically isomorphic to $\Hom(\sigma^\perp,\k^*)$ . The torus orbit $\O_0=\Hom(N^\vee,\k^*)$ corresponding to $0\in\Delta$ is called the \emph{big open torus} and is isomorphic to $T$.  We can view elements of $N^\vee$ as regular functions on $T$.  Specifically, we view $m\in N^\vee$ as the function $\chi^m$, for $t\in \Hom(N^\vee,\k^*)$,
\[\chi^m:t\mapsto t(m).\]
We refer to $\chi^m$ as a character of $T$.
The group $T$ acts on $\O_\sigma$ as follows: for $\gamma:\sigma^\perp\rightarrow\k^*$, we have $t\cdot \gamma$ given by for $m\in\sigma^\perp$,
\[t\cdot\gamma(m)=\chi^m(t)\gamma(m).\] 
The toric variety $X(\Delta)$ can be decomposed (as a set with action of $T$ as)
\[X(\Delta)=\bigsqcup_{\sigma\in\Delta} \O_\sigma.\]
The closure of the orbit $\O_\sigma$ is denoted by $V(\sigma)$.  The cones of $\Delta$ are in inclusion-reversing bijection with orbit closures of $X(\Delta)$: if $\tau$ is a face of $\sigma$ in $\Delta$, then $V(\sigma)\subseteq V(\tau)$.  In other words, $\O_\sigma\subseteq \overline{\O_\tau}$ if and only if $\tau\subseteq \sigma$. 
The closure of $T=\O_0$ is $X(\Delta)$.  Characters of $T$ extend to $X(\Delta)$ as rational functions.

A toric variety $X(\Delta)$ is compact if and only if the {\em support of $\Delta$}, $\bigcup \sigma$ is equal to $N_\R$.  A toric variety is smooth if and only if every cone is unimodular, that is, it is the non-negative span of a set of vectors that can be extended to a basis of the lattice $N$.

\begin{example} \label{e:smooth} The building blocks of smooth toric varieties are the toric varieties associated to unimodular simplices.  Let $N=\Z^n$.  Write the basis vectors of $N$ as $e_1,\dots,e_n$.  Let $\sigma$ be the nonnegative span of $e_1,\dots,e_k$.  Let $\Delta$ consist of  $\sigma$ and its faces.  Then $X(\Delta)=\k^k\times (\k^*)^{n-k}$.  Write the coordinates on $X(\Delta)$ as $(x_1,\dots,x_k,x_{k+1},\dots,x_n)$.
The faces of $\sigma$ are
\[\sigma_S=\Span_{\geq 0}(e_i\mid i\in S)\]
for $S\subseteq \{1,\dots,k\}$.  The torus orbits and their closures are given by  
\begin{eqnarray*}
\O_{\sigma_S}&=&\{x \mid x_i=0\ \text{if and only if}\ i\in S\},\\
V(\sigma_S)&=&\{x \mid x_i=0\ \text{if}\ i\in S\}.
\end{eqnarray*}
The group $(\k^*)^n$ acts by multiplication on coordinates.  The characters are the usual monomials in the $x_i$'s.
\end{example}

\begin{example} The most basic compact toric variety is $\P^n$.  Let $N=\Z^{n+1}/\Z$ where $\Z^{n+1}$ is spanned by unit vectors $e_0,e_1,\dots,e_n$ and the quotient is by the span of the diagonal element $e_0+e_1+\dots+e_n$.  The cones of $\Delta_{\P^n}$ are given by $\sigma_S=\Span_{\geq 0}(e_i\mid i\in S)$ for each proper subset $S\subsetneq \{0,1,\dots,n\}$.  Then $X(\Delta_{\P^n})=\P^n$.  The torus orbits and their closures are given in homogeneous coordinates by
\begin{eqnarray*}
\O_{\sigma_S}&=&\{[X_0:X_1:\dots:X_n]\in\P^n\mid X_i=0\ \text{if and only if}\ i\in S\},\\
V(\sigma_S)&=&\{[X_0:X_1:\dots:X_n]\in\P^n\mid X_i=0\ \text{if}\ i\in S\}.
\end{eqnarray*}
The group $T=(\k^*)^n$ acts on $\P^n$ by
\[(t_1,t_2,\dots,t_n)\cdot [X_0:X_1:\dots:X_n]=[X_0:t_1X_1:\dots:t_nX_n].\]
The characters are rational functions on $\P^n$: $N^\vee$ is canonically isomorphic to the sublattice of vectors $(m_0,m_1,\dots,m_n)\in\Z^{n+1}$ satisfying $m_0+m_1+\dots+m_n=0$, and 
\[\chi^m([X_0:X_1:\dots:X_n])=\prod_i X_i^{m_i}.\]
\end{example}

There is an important notion of morphisms of toric varieties.  Let $N$ and $N'$ be lattices with rational fans $\Delta,\Delta'$ in $N_\R,N_\R$, respectively.
Let $\phi:N'\rightarrow N$ be a homomorphism of lattices such that for any cone $\sigma'\in\Delta'$ there is a cone $\sigma\in\Delta$ with $\phi_\R(\sigma')\subset \sigma$.  In this case, we say $\phi:\Delta'\rightarrow\Delta$ is a morphism of fans.  Then, there is an induced map of toric varieties $\phi_*:X(\Delta')\rightarrow X(\Delta)$ that intertwines the torus actions.  Of particular interest for us is the case where $N=N'$ and the homomorphism $\phi$ is the identity.  In that case, $\Delta'$ is a {\em refinement} of $\Delta$, that is, every cone $\sigma'\in\Delta'$ is contained in a cone of $\Delta$.  The induced morphism $\phi_*:X(\Delta')\rightarrow X(\Delta)$ is a birational morphism.  

\begin{example} Blow-ups of $X(\Delta)$ at the orbit closures $V(\sigma)$ can be phrased in terms of  refinements of fans \cite[Section 2.4]{FultonInt}, \cite[Definition~3.3.17]{CLS}.  We will explain the case of a unimodular cone as in Example \ref{e:smooth}.  The general case is similar.  Let $\sigma$ be the cone spanned by $e_1,\dots,e_n$.  Let $\Delta$ be the fan consisting of all faces of $\sigma$.  Let $e'$ be the barycenter of $\sigma$ given by $e'=e_1+\dots+e_n$.  Let $\Delta'$ be the fan consisting of all cones spanned by subsets of $\{e',e_1,\dots,e_n\}$ not containing $\{e_1,\dots,e_n\}$.  By an explicit computation, one can show that $X(\Delta')$ is the blow-up of $\k^n$ at the origin.  
This has the effect of subdividing the cone $\sigma$ while not chainging its boundary.

We can apply this operation to a top-dimensional cone of a toric variety.  Let $\Delta$ be a fan, and let $\sigma$ be a top-dimensional, unimodular cone of $\Delta$.  We can form $\tilde{\Delta}$ by replacing $\sigma$ with the cones considered above and not changing any of the other cones.  Then $X(\tilde{\Delta})$ is the blow-up of $X(\Delta)$ at the smooth point $V(\sigma)$.

This operation can also be applied to smaller-dimensional cones as well.  Here, we follow the exposition of \cite{CLS}.  Let $X(\Delta)$ be a smooth toric variety.  For a cone $\sigma$, let $\sigma(1)$ be the set of primitive lattice vectors through the $1$-dimensional faces of $\sigma$.  Recall that a primitive lattice vector is a vector $v\in N$ such that if $v=nw$ for $n\in\Z$ and $w\in N$ then $n=\pm 1$.
 Let $\tau$ be a cone of $\Delta$.  Let the barycenter of $\tau$ be
\[u_\tau=\sum_{u_\rho\in \tau(1)} u_\rho.\]
For each cone $\sigma$ containing $\tau$, define a fan
\[\Delta_\sigma'(\tau)=\{\Span_{\geq 0}(S) \mid S\subseteq \{u_\tau\}\cup \sigma(1),\tau(1)\not\subseteq S\}.\] 
Then the subdivision of $\Delta$ relative to $\tau$ is the fan
\[\Delta'(\tau)=\{\sigma\in \Delta \mid \sigma\not\supseteq\tau\}\cup \bigcup_{\sigma\supseteq \tau} \Delta_\sigma'(\tau).\]
It turns out that $X(\Delta'(\tau))$ is the blow-up of $X(\Delta)$ along the subvariety $V(\tau)$.  The orbit closure $V(\R_{\geq 0}u_\tau)$ is the exceptional divisor of the blow-up.  Observe that only the cones containing $\tau$ are affected by the subdivision.

We can form the barycentric subdivision of $\Delta$ by first subdividing the cones $\tau$ with $\dim(\tau)=n$ and then subdividing the cones $\tau$ of $\Delta$ with $\dim(\tau)=n-1$ and so on down to the $2$-dimensional cones.  This produces an iterated blow-up of $X(\Delta)$.  This construction will be very important in the sequel. 
\end{example}

A natural way that rational fans arise is as normal fans to lattice polytopes.  Let  $P$ be a lattice polytope in $\R^n$, that is, a polytope whose vertices are in $\Z^n$.  Set $N=\operatorname{Hom}(\Z^n,\Z)$.  

\begin{definition}
For a face $Q$ of $P$, the \emph{(inward) normal cone} to $Q$ is the set
  \[\sigma_Q=\left\{w\in N\ \mid P_w\supseteq Q\right\}.\]
\end{definition}
  
 For $w$ in the relative interior of $\sigma_Q$, we have $P_w=Q$.  
 
 \begin{definition} The \emph{inward normal fan}, $N(P)$ of the polytope $P$ is 
 the union of the cones $\sigma_Q$ as $Q$ ranges over the faces of $P$.
 \end{definition}
 
 The correspondence between $Q$ and $\sigma_Q$ is inclusion-reversing.  If $P$ is full-dimensional, then $N(P)$ is strongly rational.  Because $P$ is an integral polytope, $N(P)$ is a rational fan.    

 We can use use normal fans to relate not-necessarily-normal toric varieties to the more usual toric varieties. Suppose we have $i:\P_\cA\rightarrow \P^N$.   If $i$ is an inclusion and $\P_\cA$ is normal, then $\P_\cA$ is the toric variety associated to the normal fan of the weight polytope of $\P_\cA$.
  
  The fan $\Delta_{\P^n}$ occurs as the normal fan to the  simplex whose vertices are $-e_0,-e_1,\dots,-e_n$ in  the hyperplane defined by $x_0+x_1+\dots+x_n=-1$ in $\R^{n+1}$.

\subsection{Intersection Theory and Minkowski Weights}
We review the central notions of intersection theory \cite{FultonIntersectiontheory}, eventually specializing to toric varieties.  We recommend \cite{FultonInt} as a reference for toric varieties, \cite{FS} for results on intersection theory on toric varieties, and \cite{KTT} for results most directly suited to our purposes.

Let $X$ be an $n$-dimensional algebraic variety over $\k$.  The \emph{cycle group} $Z_p(X)$ is the group of finite formal integer combinations $\sum n_i [Z_i]$ where each $Z_i$ is an irreducible $p$-dimensional subvariety of $X$.  These sums are called cycles.  If all the coefficients are non-negative, the cycle is said to be effective.  A cycle is declared to be rationally equivalent to $0$ if there exists irreducible $(k+1)$-dimensional varieties $W_1,\dots,W_l$ together with rational functions $f_1,\dots,f_l$ on $W_1,\dots,W_l$, respectively, such that 
\[\sum n_i Z_i=\sum_j (f_j)\]
where $(f_j)$ denotes the principal divisor on $W_j$ associated to the rational function $f_j$.  The Chow group $A_p(X)$ is the quotient of $Z_p(X)$ by the subgroup of cycles rationally equivalent to $0$.  Elements of $A_p(X)$ are cycle classes, but we may refer to them as cycles when we have picked a member of their class.  A cycle class is said to be \emph{effective} if it is rationally equivalent to an effective cycle.  The Chow groups should be thought of as an algebraic analogue of the homology groups $H_{2p}(X)$ with the caveat that not every homology class is realizable by a cycle and that rational equivalence is a much finer relation that homological equivalence.

If $X$ is smooth, then there is an intersection product
\[\cdot:A_p(X)\otimes A_{p'}(X)\rightarrow A_{p+p'-n}(X)\]
which  should be thought of as taking two irreducible subvarieties $Z,Z'$ to the rational equivalence class of their intersection $[Z\cap Z']$ if the subvarieties meet transversely.  A lot of work has to be done to make this intersection product well-defined for pairs of subvarieties that are not rational equivalent to pairs that meet transversely.  The intersection product is well-defined on rational equivalence classes.
We think of elements of $A_0(X)$ as a formal sum of points, and if $X$ is compact, there is a degree map:
\[\deg:A_0(X)\rightarrow\Z\]
taking $\sum n_i P_i$ to $\sum n_i$.
When $X$ is smooth and compact, we may define the Chow cohomology groups simply as $A^k(X)=A_{n-k}(X)$.  These groups are, in fact, graded rings under the cup product
\[\cup:A^k(X)\otimes A^{k'}(X)\rightarrow A^{k+k'}(X)\]
which is simply the intersection product.
We will make use of the cap-product
\[\cap:A^k(X)\otimes A_p(X)\rightarrow A_{p-k}(X)\]
where $c\cap z$ is given by taking the intersection product of $c$ (considered as a $(n-k)$-dimensional cycle) with $z$ considered as a $p$-dimensional cycle.
By definition, for $c,d\in A^*(X)$, $z\in A_*(X)$,
\[c\cap (d\cap z)=(c\cup d)\cap z.\]
For non-smooth varieties, the definition of Chow cohomology is quite different.
Here, we used Poincar\'{e} duality to simplify our exposition.

\begin{definition} A Chow cohomology class $c\in A^1(X(\Delta))$ is said to be numerically effective or {\em nef} if for every curve $C$ on $X$, $\deg(c\cap [C])\geq 0$.
\end{definition}

For $X(\Delta)$, a compact toric variety, the Chow cohomology groups $A^k(X(\Delta))$ are canonically isomorphic to a combinatorially-defined object, the group of Minkowski weights.  Let $\Delta^{(k)}$ denote the set of all cones in $\Delta$ of dimension $n-k$.  If $\tau\in\Delta^{(k+1)}$ is contained in a cone $\sigma\in\Delta^{(k)}$, let $v_{\sigma/\tau}\in N/N_\tau$ be the primitive generator of the ray $(\sigma+N_\tau)/N_\tau$.

\begin{definition} A function $c:\Delta^{(k)}\rightarrow\Z$ is said to be a {\em Minkowski weight} of codimension $k$ if it satisfies the {\em balancing condition}, that is, for every $\tau\in\Delta^{(k+1)}$,
\[\sum_{\substack{\sigma\in\Delta^{(k)}\\\sigma\supset\tau}} c(\sigma)v_{\sigma/\tau}=0\]
in $N/N_\tau$.
The \emph{support} of $c$ is the set of cones in $\Delta^{(k)}$ on which $c$ is non-zero.
\end{definition}

\begin{theorem} \cite{FS} The Chow group  $A^k(X(\Delta))$ is canonically isomorphic to the group of codimension $k$ Minkowski weights.   
\end{theorem}

The correspondence between Chow cohomology classes and Minkowski weights is as follows: given $d\in A^k(X)$, define $c(\sigma)=\deg\big(d\cap [V(\sigma)]\big)$.  The content of the Fulton-Stumfels result is that Chow cohomology classes are determined by their intersections with orbit closures.  The balancing condition is a combinatorial translation of the fact that cohomology classes are constant on rational equivalence classes generated by the rational functions given by characters of the torus $\O_\tau$ on the orbit closure $V(\tau)$.  The degree $\deg(c)$ of a class $c\in A^n(X)$ is defined to be $c(0)$, the value of $c$ on the unique zero-dimensional cone $0$.  
There is a combinatorial description of the cup product of Chow cohomology classes by the fan-displacement rule which we will not describe.  The identity in the Chow ring is $c\in A^0(X(\Delta))$ given by 
\[c(\sigma)=1\ \text{for}\ \sigma\in \Delta^{(0)}.\]
For compact smooth toric varieties, Chow cohomology is isomorphic to singular cohomology.

\begin{example}
Let us consider $\Delta_{\P^n}$.  We know that $A^k(\P^n)=\Z H^k$ where $H$ is a hyperplane class, and, consequently, $H^k$ is the class of a codimension $k$ projective subspace.
Now, we will see this fact in terms of Minkowski weights.
Let $c$ be a Minkowski weight of codimension $k$.  It is straightforward to verify that the Minkowski weight condition is equivalent to $c$ being constant on $\Delta^{(k)}$.  If $\tau\in\Delta^{(k)}$, $V(\tau)$ is a $k$-dimensional coordinate subspace.  Now, $c$, considered as a cycle, must intersect $V(\tau)$ in $c(\tau)$ points counted with multiplicity.  Therefore, $c$ is in the class of $c(\tau)H^k$.
\end{example}

There is an equivariant version of Chow cohomology which has a combinatorial description on toric varieties.  Here, equivariant Chow cohomology means equivariant with respect to the $T$-action and is analogous to equivariant cohomology.  We will only make use of the codimension $1$ case.  Let $N^\vee=\Hom(N,\Z)$ be the dual lattice to $N$, interpreted as linear functions on $N_\R$ with integer slopes.  A  linear function on a cone $\sigma$ can be interpreted as an element of $N^\vee(\sigma)=N^\vee/\sigma^\perp$.  A piecewise linear function $\alpha$ on $\Delta$ is a continuous function on the support of $\Delta$ whose restriction to each cone is a linear function with integer slopes, or more formally as the following:

\begin{definition} A piecewise linear function on the fan $\Delta$ is a collection of elements $\alpha_\sigma\in N^\vee/\sigma^\perp$ for each $\sigma\in \Delta$ such that for $\tau\subset \sigma$, the image of $\alpha_\sigma$ under the quotient $N^\vee/\sigma^\perp\rightarrow N^\vee/\tau^\perp$ is $\alpha_\tau$.
\end{definition}

Observe that if $\phi:\Delta'\rightarrow\Delta$ is a moprhism of fans, and $\alpha$ is a piecewise linear function on $\Delta$, then the pullback $\phi^*\alpha$ is a piecewise linear function of $\Delta'$.
The piecewise linear function $\alpha$ can be thought of as a $T$-Cartier divisor on $X(\Delta)$: 
if the restriction of $\alpha$ to $\sigma$, $\alpha_{\sigma}$ is given by $m\in M/M(\sigma)$, and the Cartier divisor is locally defined by the rational function $\chi^m$, the character associated to $m$ on the toric open affine associated to $\sigma$.   This $T$-Cartier divisor can be thought of as an element of $A_T^1(X(\Delta))$, the equivariant Chow cohomology group \cite{Edidin-Graham}.    There is a natural non-equivariant restriction map $\iota^*:A_T^1(X(\Delta))\rightarrow A^1(X(\Delta))$ that takes the $T$-Cartier divisor to an ordinary Chow cohomology class.  The piecewise linear function $\alpha$ induces a $T$-equivariant line bundle $L$ on $X(\Delta)$ and $\iota^*\alpha=c_1(L)$, the first Chern class of $L$.

If $c\in A^k(X(\Delta))$ is viewed as a Minkowski weight, we may compute the cup product $\iota^*\alpha\cup c$ as an element of $A^{k+1}(X)$ by using a formula that first appeared in \cite{AR}:
for $\sigma\in\Delta^{(k)}, \tau\in\Delta^{(k+1)}$, let $u_{\sigma/\tau}$ be a vector in $N_\sigma$ descending to $v_{\sigma/\tau}$ in $N/N_\tau$; then the value of $\iota^*\alpha\cup c$ on a cone $\tau\in\Delta^{(k+1)}$ is
\[
(\iota^*\alpha\cup c)(\tau)=-\sum_{\sigma\in\Delta^{(k)}\mid\sigma\supset\tau} \alpha_\sigma(u_{\sigma/\tau})c(\sigma)+\alpha_\tau\left(\sum_{\sigma\in\Delta^{(k)}\mid\sigma\supset\tau}c(\sigma)u_{\sigma/\tau}\right)
\]
where $\alpha_\sigma$ (respectively $\alpha_\tau$) is the linear function on $N_\sigma$ (on $N_\tau$) which equals $\alpha$ on $\sigma$ (on $\tau$).   Taking $\iota^*\alpha\cup c$ yields a Minkowski weight.  The following lemma can be proved using intersection theory \cite{KTIT} or by elementary means \cite[Prop~3.7]{AR}:

\begin{lemma} Let $\alpha$ be a piecewise linear function on $\Delta$, and let $c$ be a Minkowski weight of codimension $k$.  Then $\iota^*\alpha\cup c$ is a Minkowski weight of codimension $k+1$.
\end{lemma} 

A $T$-Cartier divisor $\alpha$ is said to be {\em nef} if for every codimension $1$ cone $\tau\in\Delta^{(1)}$, we have $\iota^*\alpha(\tau)\geq 0$. This says that the cohomology class $\iota^*\alpha$ is non-negative on any $1$-dimensional orbit closure.  This happens if and only if $\iota^*\alpha$ is non-negative on every curve class and is therefore nef in the classical sense.  Consequently if $\alpha$ is nef, it  induces a $T$-equivariant line bundle on $X(\Delta)$ whose first Chern class is nef. 
Nefness can be interpreted in terms of convexity.  If $\tau$ is a codimension $1$ cone of $\Delta$, it is contained in two top-dimensional cones, $\sigma_1,\sigma_2$.  We can pick $u_{\sigma_1/\tau},u_{\sigma_2/\tau}$ such that $u_{\sigma_1/\tau}+u_{\sigma_2/\tau}=0$.
Then, we have
\[
\iota^*\alpha(\tau)=-\alpha_{\sigma_1}(u_{\sigma_1/\tau})-\alpha_{\sigma_2}(u_{\sigma_2/\tau})
\]
which expresses how the slope of $\alpha$ changes as we pass through the wall $\tau$ from $\sigma_1$ to $\sigma_2$.  Therefore, the condition $\iota^*\alpha(\tau)\geq 0$ can be interpreted as a convexity condition on the piecewise linear function $\alpha$.

\begin{example} \label{e:hyperplaneclass} Consider the fan $\Delta_{\P^n}$.  Let $w_1,\dots,w_n$ be coordinates on $N_{\R}$ given by the basis vectors $e_1,\dots,e_n$.  Set $e_0=-e_1-\dots-e_n$.  Let $\alpha$ be the piecewise linear function given by \[\alpha=\min(0,w_1,\dots w_n).\]
Let us compute $\iota^*\alpha\in A^{n-1}(\P^n)$.  The $(n-1)$-dimensional cones of $\Delta_{\P^n}$ are of the following form: for $\{j,j'\}\subset\{0,1,\dots,n\}$
\[\tau_{jj'}=\Span(e_i\mid i\in\{0,1,\dots,n\}, i\neq j, i\neq j'),\]
By explicit computation,
\[\iota^*\alpha(\tau_{jj'})=1.\]
The orbit closures $V(\tau_{jj'})$ are the coordinate lines in $\P^n$ cut out by setting all but two homogeneous coordinates to $0$. 
This tells us that $\iota^*\alpha$ intersects each $1$-dimensional coordinate line in a point,  and therefore, it is a hyperplane class.  \end{example}

If $\phi:\Delta'\rightarrow\Delta$ is a refinement of fans, then we can treat a piecewise linear function $\alpha$ on $\Delta$ as a piecewise linear function on $\Delta'$, the pullback $\phi^*\alpha$.  Moreover, if $\alpha$ is nef on $\Delta$, then $\phi^*\alpha$ is nef on $\Delta'$.

There is a way of associating a Minkowski weight of codimension $r$ to a codimension $r$ subvariety $Y$ of a smooth, complete toric variety $X(\Delta)$.  This Minkowski weight should be thought of as a Poincar\'{e} dual to the cycle $[Y]$.  Let $Y$ be a subvariety of dimension $r$. Define a function
\[
c : \Delta^{(n-r)} \to \mathbb{Z}, \qquad \sigma \mapsto \deg\big([Y] \cdot [V(\sigma)]\big).
\]
Then $c$ is a Minkowski weight, called the {\em associated cocycle} of $Y$. See \cite{KTT} or \cite{ST} for details.

\begin{definition} A $d$-dimensional irreducible subvariety $Y$ of $X(\Delta)$ is said to {\em intersect the torus orbits of $X(\Delta)$ properly}, if for any cone $\sigma$ of $\Delta$,
\[\dim(Y\cap V(\sigma))=\dim(Y)+\dim V(\sigma)-n.\]
\end{definition}

\begin{lemma}\label{l:assoccyc} \cite[Lemma 9.2]{KTT} \excise{Suppose $Y$ intersects the torus orbits of $X$ properly.} If $c$ is the associated cocycle of a subvariety $Y$, then
\[c\cap [X]=[Y]\in A_{r}(X).\]
\end{lemma}

If $Y$ intersects orbits properly,  the associated cocycle is essentially the same thing as the tropicalization of $Y^\circ\subset(\k^*)^n$ where $Y^\circ$ is defined below.

\begin{definition}
For a subvariety $Y\subset X(\Delta)$, the \emph{interior of $Y$} is
\[Y^\circ=Y\cap (\k^*)^n\]
where $(\k^*)^n$ is the big open torus of $X(\Delta)$.
\end{definition}

It is not difficult to show that the Chow  ring of projective space, $A^*(\P^n)$ is $\Z[H]/H^{n+1}$ where $H$ is the class of a hyperplane.  Because the intersection of $n$ hyperplanes in generic position is a point, $\deg(H^n)=1$.  Moreover, we have 
\[A^*(\P^n\times \P^n)\cong \Z[H_1,H_2]/(H_1^{n+1},H_2^{n+1})\]
where the codimension of a cohomology class is the total degree of the corresponding polynomial in $H_1,H_2$.  An element of $A^k(\P^n\times \P^n)$ can be written as
\[c=\sum_{i=0}^k a_i H_1^i H_2^{k-i}\]
for $a_i\in \Z$.  Because $\deg(H_1^nH_2^n)=1$, $a_i=\deg(H_1^{n-i}H_2^{n-k+i}\cup c).$

\section{Bergman fans} \label{s:bergman}
\subsection{Definition of a Bergman fan} \label{ss:defbergman}
The Bergman fan is a combinatorial object associated to a matroid.  We will see that it leads to a cryptomorphic definition of a matroid.  It was first introduced as the logarithmic limit set by Bergman \cite{Bergman} and then shown to be a finite polyhedral complex by Bieri-Groves \cite{BieriGroves}.  Sturmfels gave a combinatorial definition \cite{StuSolving} which was elaborated by Ardila-Klivans \cite{AKBergman}.  

The logarithmic limit set is an invariant of a $(d+1)$-dimensional linear subspace $V\subset \C^{n+1}$.  The \emph{amoeba} $A(V)$ of $V$ is the set of all vectors of the form
\[
\big(\log|x_0|,\log|x_2|,\ldots,\log|x_n|\big) \in \mathbb{R}^{n+1}
\]
for $x\in V$.
One considers the Hausdorff limit of dilates $t A(V)$ as $t$ goes to $0$. This gives a logarithmic limit set, which is also called the tropicalization and captures the asymptotic behavior of the amoeba.  Because $V$ is invariant under dilation by elements of $\C^*$, the amoeba and hence the logarithmic limit set is invariant under translation by the diagonal vector $(1,\dots,1)$.  Therefore, we may view the logarithmic limit set as a subset of $\R^{n+1}/\R$.  We will give a definition of the logarithmic limit set with reference only to the matroid of $V$ and which makes sense for non-representable matroids.

\excise{As an aside, we note that the matroid $M_w$ is, in a certain sense, a Gr\"{o}bner degeneration of the matroid $M$ if it is realizable.  Suppose $M$ is realized by a $d$-dimensional projective subspace $V$ in $\P^n$.  Let $R=\k[[t]]$, the ring of formal power series in $t$..  Suppose $w\in \Z^{n+1}$, then we may consider $t^{-w} V_R$ where $t^{w}=(t^{w_0},t^{w_1},\dots,t^{w_n})\in (\G_m)_R$ and $V_R=V\times_{\k} R$.  Then the special fiber of $t^{-w} V_R$ is 
\[\operatorname{in}_w V=(t^{-w} V_R)\times_R \k.\]
The irreducible component of $\operatorname{in}_w V$ that intersects the open torus $(\k^*)^n$ is a projective subspace with matroid $M_w$.}

Let $M$ be a matroid on $E=\{0,1,\dots,n\}$.  We first define the underlying set of the Bergman fan.  

\begin{definition} An element $w\in\R^E$ is said to be {\em valid} if $M_w$ has no loops.  The \emph{underlying set of the Bergman fan} of $M$ is the subset of $\R^E$ consisting of valid $w$.
\end{definition}

Observe that the underlying set of the Bergman fan is closed.   If we have a sequence $\{w_i\}$ in the underlying set of the Bergman fan with $\lim  w_j=w$, then $P(M)_w$ contains $P(M)_{w_j}$ for sufficiently large $j$.  If $M_{w_j}$ has no loops, every element of the ground set occurs as an element of some basis of $M_{w_j}$.  To show that this is true for $M_w$, let $i\in E$.  Then, for all $j$, $P(M)_{w_j}$ is not contained in the hyperplane $x_i=0$.  It follows that the same is true for $M_w$.  Consequently $i$ is an element of some basis of $P(M)_w$ and hence is not a loop.

\begin{proposition} \cite{StuSolving} The underlying set of the Bergman fan is the logarithmic limit set of any realization of $M$ as a linear subspace in $\k^{n+1}$.
\end{proposition}

Now, because the underlying set of the Bergman fan is invariant under translation by the vector $(1,1,\dots,1)$, we quotient by that vector.  Consider the lattice $N=\Z^{n+1}/\Z(1,1,\dots,1)$.  We write $N_\R$ for $\R^{n+1}/\R(1,1,\dots,1)$.  We will treat $N$ as dual to the lattice in the hyperplane $x_0+x_1+\dots+x_n=d+1$ that contains the matroid polytope $P(M)$.  The underlying set of the Bergman fan is the union of cones of the normal fan to the matroid polytope because the matroid $M_w$ only depends on which cone of the normal fan to which $w$ belongs.   Therefore, the underlying set of the Bergman fan can be given the structure of a subfan of the normal fan of the matroid polytope.  However, following \cite{AKBergman}, we will put a finer structure on the Bergman fan.  This structure will be very important in the sequel

We first define cones $\sigma_{S_\bullet}$ in $N_\R$ that refine the normal fan to the matroid polytope.  That means that every cone $\sigma_{S_\bullet}$ will be contained in a cone $\sigma_{P(M)_w}$ of the normal fan of the matroid polytope.
These cones will therefore have the property that if $w_1,w_2\in \sigma_{S_\bullet}^\circ$, then $M_{w_1}=M_{w_2}$.

For a subset $S \subset E$, let $e_S$ be the vector 
\[
e_S= \sum_{i \in S} e_i
\]
in $N_\R$.
Note that $e_E=0$.  
For a flag of subsets,
\[S_\bullet=\{\emptyset\subsetneq S_1\subsetneq \ldots \subsetneq S_k\subsetneq E\},\]
let $\sigma_{S_\bullet}=\Span_{\geq 0}(e_{S_1},\ldots,e_{S_l}).$
If $w$ is in the relative interior of $\sigma_{S_\bullet}$ then $w_i<w_j$ if and only if $i$ occurs in an
$S_k$ with a larger index than $j$ does.  Similarly, given $w\in \R_{\geq 0}^E$, we may pick subsets 
$\emptyset\subsetneq S_1\subsetneq\ldots\subsetneq S_l\subsetneq E$ such that the elements of 
$E$ with maximum $w$-weight are exactly the elements of $S_1$, the ones with the next smallest $w$-weight are the elements of $S_2$, and so on.  Consequently, $w$ is an element of the relative interior of $\sigma_{S_\bullet}$.

\begin{lemma} Let $M$ be a matroid.  For any $S_\bullet$, $\sigma_{S_\bullet}$ is contained in a cone of the normal fan to the matroid polytope, $N(P(M))$.
\end{lemma}

\begin{proof}
It suffices to show that the relative interior $\sigma_{S_\bullet}^\circ$ is contained in the relative interior of a cone of $N(P(M))$.   By our description of the normal fan of the matroid polytope, we need to show that for $w_1,w_2\in\sigma_{S_\bullet}^\circ$, $M_{w_1}=M_{w_2}$.  However, the condition $w_1,w_2\in\sigma_{S_\bullet}^\circ$ implies the relative ordering of weights of bases  with respect to $w_1$ and $w_2$ are the same.
\end{proof}

\begin{lemma} \cite{AKBergman} \label{l:validflats} A cone $\sigma_{S_\bullet}$ consists of valid vectors $w$ if and only if $S_\bullet$ is a flag of flats.  \end{lemma}

\begin{proof}
Because the underlying set of the Bergman fan is closed, it suffices to prove the same statement with $\sigma_{S_\bullet}$ replaced by $\sigma_{S_\bullet}^\circ$.  

Suppose some $S_i$ is not a flat of $M$.  Then there exists $j\in \operatorname{cl} S_i\setminus S_i$ and an $i'>i$ such that $j\in S_{i'}\setminus S_{i'-1}$.  Because $j$ is in the closure of a flat in $S_{i'-1}$, $j$ must be a loop in $M|_{S_{i'}}/S_{i'-1}$, hence in $M_w$ by the behaviour of loops under direct sum.

Now, suppose that each $S_i$ is a  flat of $M$.  Let $j\in E$.  We must find a basis of $M_w$ containing $j$.  There is an $i$ such that $j\in S_i\setminus S_{i-1}$.  Because $M$ has no loops, $j$ is not a loop of $M|_{S_i}$.   Because $j$ is not in the closure of $S_{i-1}$, it is not a loop of $M|_{S_i}/S_{i-1}$, and therefore, it is not a loop in $M_w$.
\end{proof}

Lemma \ref{l:validflats} shows us that the underlying set of the Bergman fan is a union of cones of the form $\sigma_{S_\bullet}$.  We use this fact to define the {\emph{Bergman fan}} $\Delta_M$ as the simplicial fan in $N_\R$ given by $\sigma_{F_\bullet}$ for flags of flats.  A $k$-step {\em flag of proper flats} is a sequence of proper flats ordered by containment:
\[F_\bullet=\{\emptyset \subsetneq F_1 \subsetneq \cdots \subsetneq F_k\subsetneq E\}\]
The cone associated to $F_\bullet$ is the non-negative span
\[\sigma_{F_\bullet}=\Span_{\geq 0} \{ e_{F_1}, \ldots, e_{F_k} \}.\]

\begin{definition} The {\emph Bergman fan} of $M$ is the fan $\Delta_M$ consisting of the cones $\sigma_{F_\bullet}$ for all flags of flats $F_\bullet$.
\end{definition}

Note that this is a fan structure, not just an underlying set.   Ardila and Klivans introduced this fan in \cite{AKBergman} and called it the fine subdivision of the Bergman fan of the matroid.  Because every flag of flats in a matroid can be extended to a maximal flag of proper flats of length $d$, the fan $\Delta_M$ is pure of dimension $d$.

Recall that the order complex of a finite poset \cite{StanleyBook} is the simplicial complex whose vertices are elements of the poset and whose simplices are the chains of the poset.  The Bergman fan is a realization of the cone over the order complex of the lattice of proper, non-trivial flats, $L(M)\setminus\{\emptyset,E\}$.

\subsection{The Uniform matroid and the permutohedral variety}

Of particular interest is the Bergman fan associated to the uniform matroid.  We will consider $\U=U_{n+1,n+1}$, the rank $n+1$ uniform matroid on $E=\{0, \ldots, n\}$.  Its Bergman fan  consists of cones in $N_\R=\R^{n+1}/\R$.  Every subset of $E$ is a flat.  Consequently, the top-dimensional cones of $\Delta_{\U}$ are of the form 
\[\Span_{\geq}\{e_{i_0},e_{i_0}+e_{i_1},\dots,e_{i_0}+\dots+e_{i_{n-1}}\}\]
for every permutation $i_0,\dots,i_n$ of $0,\dots,n$.  Because these cones are generated by a  basis of $N$, $X(\Delta_\U)$ is smooth.
The fan is the barycentric subdivision of the fan corresponding to $\P^n$, and $X(\Delta_{\U})$ is the toric variety obtained from $\P^n$ by a sequence of blowups,
\[
\pi_1:X(\Delta_{\U}) = X_{n-1} \rightarrow \cdots \rightarrow X_1 \rightarrow X_0 = \P^n,
\]
where $X_{i+1} \rightarrow X_i$ is the blowup along the proper transforms of the $i$-dimensional torus-invariant subvarieties of $\P^n$. The fan, $\Delta_{\U}$ is the normal fan to the permutohedron in its affine span, where the permutohedron is the convex hull of all coordinate permutations of the point $(0,1,\dots,n)\in\R^{n+1}$ \cite{GKZ}.  

\begin{example} The Bergman fan of $U_{d+1,n+1}$ also has a simple description.  The top-dimensional cones 
are of the form 
\[\Span_{\geq 0}\{e_{i_0},e_{i_0}+e_{i_1},\dots,e_{i_0}+\dots+e_{i_{d-1}}\}\]
for every $d$-tuple $(i_0,\dots,i_{d-1})$ of distinct elements of $\{0,\dots,n\}$. 

In particular, if $d+1=1$, then the Bergman fan consists simply of the origin.  If $d+1=2, n+1=3$, then the Bergman has top-dimensional cones $\R_{\geq 0}e_0,\R_{\geq 0}e_1,\R_{\geq 0}e_2$, and its underlying set is the much-pictured tropical line in the plane with vertex at the origin.
\end{example}

\begin{definition} The $n$-dimensional {\em permutohedral variety} is $X(\Delta_{\U})$, the toric variety associated to the Bergman fan of the uniform matroid $\U=U_{n+1,n+1}$.
\end{definition}

Now, the permutohedral variety $X(\Delta_{\U})$ possesses two maps, $\pi_1,\pi_2$, to $\P^n$ that will be of particular importance.  As noted above, every cone of $\Delta_{\U}$ is contained in a cone of $\Delta_{\P^n}$.  The induced map $\pi_1:X(\Delta_{\U})\rightarrow \P^n$ is the blow-up described above.  Now, let $-\Delta_{\P^n}$ be the fan whose cones are of the form $-\sigma$ for each $\sigma\in\Delta_{\P^n}$.  The toric variety $X(-\Delta_{\P^n})$ is $\P^n$ but with its torus action precomposed by taking inverse.  Now, $\Delta_\U$ is a refinement of $-\Delta_{\P^n}$ according to the following lemma:

\begin{lemma} Each cone of $\Delta_\U$ is contained in a cone of $-\Delta_{\P^n}$.
\end{lemma}

\begin{proof}
For a subset $S\subseteq E=\{0,1,\ldots,n\}$, $e_S=-e_{E\setminus S}$.  Consequently, if 
\[F_\bullet=\{\emptyset \subsetneq F_1 \subsetneq \cdots \subsetneq F_k\subsetneq E\}\]
is a flag of flats, so is
\[F^c_\bullet=\{\emptyset \subsetneq E\setminus F_k \subsetneq \cdots \subsetneq E\setminus F_1\subsetneq E\}.\]
Therefore, $-\sigma_{F_\bullet}=\sigma_{F^c_\bullet}$ is a cone of $\Delta_\U$, and $-\Delta_\U=\Delta_\U$.  Since $\Delta_\U$ refines $\Delta_{\P^n}$, the conclusion follows.
\end{proof}

We set $\pi_2$ to be the induced birational morphism $\pi_2:X(\Delta_{\U})\rightarrow X(-\Delta_{\P^n})$.  The existence of $\pi_1$ and $\pi_2$ shows that the permutohedral variety resolves the  indeterminacy of the standard  generalized Cremona transformation
\[
\operatorname{Crem} : \mathbb{P}^n \dashrightarrow \mathbb{P}^n, \qquad (z_0:\cdots:z_n) \mapsto  (z_0^{-1}:\cdots:z_n^{-1}).
\]
This map is induced by multiplication by $-1$ on $N$.  This map is only a rational map on $\P^n$ because multiplication by $-1$ does not take cones of $\Delta_{\P^n}$ to cones of $\Delta_{\P^n}$. However, it does induce a birational automoprhism of $X(\Delta_{\U})$ fitting into a commutative diagram
\[\xymatrix{
X(\Delta_\U)\ar[r]^{\operatorname{Crem}}\ar[d]^{\pi_1}&X(\Delta_\U)\ar[d]^{\pi_1}\\
\P^n\ar@{-->}[r]^{\operatorname{Crem}}&\P^n.
}\]
Note that $\pi_2=\pi_1\circ{\operatorname{Crem}}.$
There is a natural morphism $\pi_1\times\pi_2:X(\Delta_\U)\rightarrow\P^n\times\P^n$.  The presence of the Cremona transformation is related to the use of reciprocal hyperplanes as in \cite{BrokenCircuitRing} and \cite{EntropicDiscriminant}. 

There are natural subvarieties of $X(\Delta_\U)$, called proper transforms, that are associated to subspaces of $\P^n$:

\begin{definition}  Let $V$ be a $(d+1)$-dimensional subspace of $\k^{n+1}$ that is not contained in any hyperplane, and let $\P(V)\subset \P^n$ be its projectivization.  The \emph{proper transform} of $\P(V)$ in $X(\Delta_\U)$ is 
\[\widetilde{\P(V)}=\overline{\pi_1^{-1}(\P(V)^\circ)}.\]
\end{definition}

By the blow-up interpretation of the permutohedral variety, the proper transform is an iterative blow-up of $\P(V)$ at its intersections with the coordinate subspaces of $\P^n$.  First one blows up the $0$-dimensional intersections, then the proper transforms of $1$-dimensional intersections, and then so on.  It is easily seen that $\widetilde{\P(V)}$ intersects the torus orbits of $X(\Delta_{\U})$ properly.  

There are two piecewise linear functions of $\Delta_{\U}$ that will be of great importance in the sequel.  We have the piecewise linear function on $\Delta_{\P^n}$ from Example \ref{e:hyperplaneclass}, 
\[\alpha=\min(0,w_1,\dots,w_n)\]
 We will abuse notation and denote the pullback $\pi_1^*\alpha$ on $X(\Delta_{\U})$ by $\alpha$.  We can also pull back the analogous piecewise linear function from  $X(-\Delta_{\P^n})$ to obtain
\[\beta=\min(0,-w_1,\dots,-w_n).\]
The non-equivariant cohomology classes $\iota^*\alpha$ corresponds to the proper transform of a generic hyperplane in $X(\Delta_\U)$.  On the other hand, $\iota^*\beta$ corresponds to the closure in $X(\Delta_{\U})$ of reciprocal hyperplanes in $(\k^*)^n$ given by
\[a_0+\frac{a_1}{x_1}+\dots+\frac{a_n}{x_n}=0\]
for generic choices of $a_i\in \k$ where $(\k^*)^n$ is the big open torus.
We may suppress $\iota^*$ and write $\alpha$ and $\beta$ for the induced non-equivariant cohomology classes.

\subsection{The Bergman fan as a Minkowski weight}
 
In this section, we show how the Bergman fan $\Delta_M$ can be thought of a Minkowski weight on $\Delta_{\U}$.   We will let $M$ be a rank $d+1$ matroid so that for $F_\bullet$, a flag of flats in $M$ of maximum length, $\sigma_{F_\bullet}$ is a $d$-dimensional cone.

\begin{definition} The Minkowski weight on $\Delta_\U$ corresponding to the rank $d+1$ matroid $M$ is given by $\Delta_M:\Delta_{\U}^{(n-d)}\rightarrow\Z$ where
\[\Delta_M(\sigma_{F_\bullet})=\begin{cases}
1&\text{if $F_\bullet$ is a  flag of flats in $M$}\\
0&\text{otherwise}.
\end{cases}\]
\end{definition}

\begin{lemma} The function $\Delta_M$ is a Minkowski weight on $\Delta_\U$
\end{lemma}

\begin{proof}
Let $\sigma_{G_\bullet}$ be a codimension $1$ cone in $\Delta_M$.  Then $\sigma_{G_\bullet}$ corresponds to a flag of flats of length $d-1$ of the form
\[G_{\bullet}=\{\emptyset\subsetneq F_1 \subsetneq \dots\subsetneq F_{j-1}\subsetneq F_{j+1}\subsetneq \dots\subsetneq F_d\}\]
where $r(F_i)=i$.
We restrict $M$ to $F_{j+1}$ since any cone containing $\sigma_{G_\bullet}$ corresponds to a flag of flats refining $G_\bullet$.  The flats $F'_1,\dots,F'_l$ of $M|_{F_{j+1}}$ properly containing $F_{j-1}$ are exactly the flats that can be inserted in $G_\bullet$  to obtain a flag of flats of length $d$.  The sets $F'_k\setminus F_{j-1}$ partition $F_{j+1}\setminus F_{j-1}$ by Definition \ref{d:flats}.   Let $(F_{\bullet})_k$ be the flag of flats given by inserting $F'_k$ into $G_\bullet$.  Therefore, 
$u_{\sigma_{(F_\bullet)_k}/\sigma_{G_\bullet}}=e_{F'_k}$.  From 
\[\sum_{k=1}^l \Delta_M(\sigma_{(F_\bullet)_k}) u_{\sigma_{(F_\bullet)_k}/\sigma_{G_\bullet}} =
\sum_{k=1}^l e_{F'_k}\in\Span(e_{F_{j-1}},e_{F_{j+1}})\subseteq N_{\sigma_{G_\bullet}},\]
we see that $\Delta_M$ is a Minkowski weight.
\end{proof}

The following is straightforward:
\begin{lemma} \label{l:bergmanassociated} Let $V\subset\k^{n+1}$ be a representation of a simple rank $d+1$ matroid on $\{0,1,\dots,n\}$.  Then the Minkowski weight $\Delta_M$ is the associated cocycle of the proper transform $\widetilde{\P(V)}\subset X(\Delta_\U)$.
\end{lemma}

It follows that the Bergman fan is the same thing as the tropicalization of $\P(V)^\circ$.  
In less technical terms, the Minkowski weight records the toric strata that $\widetilde{\P(V)}$ intersects.  These corresponds to flags of flats of $M$.  

It turns out that among tropicalizations, the Bergman fans exactly correspond to the tropicalizations of linear subspaces.  This corresponds to a theorem that emerged in a discussion of Mikhalkin and Ziegler and which was written down in \cite{KPReal}.  See \cite{Huhthesis} for another proof perhaps more suitable for the statement here.  We call this theorem, ``the duck theorem'' in the sense that {\em if it looks like a duck, swims like a duck, and quacks like a duck, then it is probably a duck}:

\begin{theorem} Let $Y^\circ\subset (\k^*)^n\subset X(\Delta_{U_{\U}})$ be a variety whose associated cocycle is $\Delta_M$, the Bergman fan of a matroid $M$.  Then $Y^\circ=\P(V)^\circ$ where $V$ is a subspace realizing the matroid $M$.
\end{theorem}

This theorem can be used to come up with counterexamples to questions about tropical lifting, that is, to determine whether a Minkowski weight is the associated cocycle of an algebraic subvariety of $(\k^*)^n$.  By starting with a non-representable matroid,  one can produce a Minkowski weight $\Delta_M$ on $X(\Delta_\U)$ which is is not the associated cocycle of any subvariety of $X(\Delta_\U)$.  Indeed if such a subvariety existed, it would be the closure of $\P(V)^\circ$ where $V$ is a representation of $M$.  

In addition, this theorem can be used to study which homology classes in $X(\Delta_{\U})$ can be realized by irreducible subvarieties.  In \cite{Huhthesis}, it is shown that the Bergman class of a loopless matroid of rank $d+1$, considered as cohomology class is nef in the sense that it has non-negative intersection with all effective cycles of complementary dimension.  Moreover, it generates an extremal ray of the nef cone in dimension $d$.  It is effective in that it can be written as the sum of effective cycles.  However, it can  be represented by an irreducible subvariety over $\k$ if and only if the matroid is representable over $\k$.  This shows that the study of the realizability of homology classes in toric varieties is heavily dependent on the ground field and is as complicated as the theory of representability of matroids.

\subsection{$(r_1,r_2)$-Truncation of Bergman fans}

One advantage of working with Minkowski weights in $\Delta_{\U}$ rather than matroids is that they allow a new operation, $(r_1,r_2)$-truncation as developed by Huh and Katz. More details can be found in \cite{Huhthesis}.

\begin{definition}
Let $M$ be a matroid of rank $d+1$.  Let $r_1,r_2$ be integers with $1\leq r_1\leq r_2\leq d+1$.  The {\em $(r_1,r_2)$-truncation of $M$} is the Minkowski weight of dimension $r_2-r_1+1$, $\Delta_{M[r_1,r_2]}$ on $\Delta_\U$ defined as follows: if $\sigma_{F_\bullet}$ is the cone determined by 
\[F_{\bullet}=\{F_{r_1}\subsetneq F_{r_1+1} \subsetneq \ldots \subsetneq F_{r_2}\}\]
then $\Delta_{M[r_1,r_2]}(\sigma_{F_\bullet})$ is
\begin{enumerate}
\item $|\mu(\emptyset, F_{r_1})|$ if each $F_i$ is a flat of $M$ of rank $i$ or
\item $0$ otherwise.
\end{enumerate}
\end{definition}

Note that the $(1,r_2)$-truncation of $M$ is the Bergman fan of the matroid $\Trunc^{r_2+1}(M)$.  However, if $r_1\geq 2$, the $(r_1,r_2)$-truncation of $M$ is not the Bergman fan of any matroid.  It is not obvious that $\Delta_{M[r_1,r_2]}$ is indeed a Minkowski weight.  This will follow from the following lemma whose proof we will defer until subsection \ref{ss:alphabeta}.

\begin{proposition} \label{p:alphabeta} We have the following relation among Minkowski weights:
\[\alpha^{d-r_2}\beta^{r_1-1}\cup\Delta_M=\Delta_{M[r_1,r_2]}.\]
Moreover, we have the following degree computations:
\begin{enumerate}
\item $\deg(\alpha\cup \Delta_{M[r,r]})=\mu^{r-1}$
\item $\deg(\beta\cup \Delta_{M[r,r]})=\mu^r$
\end{enumerate}
where $\mu^r$ is a coefficient of the reduced characteristic polynomial.
\end{proposition}

Note that applying powers of $\alpha$ in the above makes geometric sense.  It states that $\alpha^{k}\cup\Delta_M=\Delta_{\Trunc^{d+1-k}(M)}$.  Because $\alpha$ is the class of a hyperplane in $\P^n$, we are intersecting a projective subspace $\P(V)$ with generic hyperplanes to obtain a projective subspace whose matroid is a truncation of $M$ by Lemma \ref{l:trunc}.

We have the following easy corollary:
\begin{corollary} \label{c:reddeg} There is the following intersection-theoretic description of coefficients of the reduced characteristic polynomial:
\label{l:reddeg} 
\[\deg(\alpha^{d-r}\beta^r\cup\Delta_M)=\mu^r.\]
\end{corollary}

\begin{proof}
By applying the above proposition, we have
\[\deg(\alpha^{d-r}\beta^r\cup\Delta_M)=\deg(\beta\cup(\alpha^{d-r}\beta^{r-1}\cup\Delta_M))=\deg(\beta \cup \Delta_{M[r,r]})=\mu^r.\]
\end{proof}

\subsection{The Bergman fan as a cryptomorphic definition of a matroid}

Minkowski weights on $\Delta_{\U}$ are interesting combinatorial objects in their own right.  They contain matroids on $\{0,1,\dots,n\}$ as a specific subclass but are closed under the truncation operation defined above.  It would be an interesting exercise to generalize the matroid operations that we discussed in Section \ref{s:operations} to them.   Here, extensions are closely related to the notion of tropical modifications as introduced by Mikhalkin.  See \cite{Shaw} for more on tropical modifications.

It is a consequence of Proposition \ref{p:alphabeta} that if $c$ is a Minkowski weight associated to a rank $d+1$ matroid on $\{0,1,\dots,n\}$ then $\alpha^{d}\cup c$ is the Minkowski weight corresponding to the rank $0$ matroid.  Consequently, $\deg(\alpha^{d}\cup c)=1$.  The converse is true by the following theorem which was  noted by Mikhalkin, Sturmfels, and Ziegler as described by \cite{Oberwolfach} and proved by Fink \cite{FinkCycles}.

\begin{theorem} Let $c\in A^{n-d}(X(\Delta_{\U}))$ be a Minkowski weight of codimension $n-d$.  Then $c=\Delta_M$ for some matroid $M$ of rank  $d+1$ if and only $\deg(\alpha^{d}\cup c)=1$.
\end{theorem}

\excise{We give the proof of the easy implication.
\begin{proof} Suppose that $c$ is the Bergman fan of a matroid $M$ of rank $d+1$.  Then by Proposition \ref{p:alphabeta}, 
\[\deg(\alpha^d\cup\Delta_M)=\deg(\alpha\cup \Delta_{M[1,1]})=\mu^0=1.\] 
\end{proof}}

The above theorem can be thought of as cryptomorphic definition of a matroid.  In the course of the proof, Fink gives an explicit recipe for finding the matroid polytope associated to $c$.

\section{Log-concavity of the characteristic polynomial}

\subsection{Proof of log-concavity}
We will prove that $\chi_M(q)$ is log-concave when $M$ is a representable matroid.  By an easy algebra computation, it suffices to show that the reduced characteristic polynomial $\overline{\chi}_M(q)$ is log-concave.

We will identify $\mu^k$ with intersection numbers on a particular algebraic variety.  Let $V\subset \k^{n+1}$ be a linear subspace representing $M$.  We will give the proof assuming Proposition \ref{p:alphabeta} for now.

\begin{proposition} \label{p:intersection} There is a complete irreducible algebraic variety $\widetilde{\P(V)}$ with nef divisors $\alpha,\beta$ such that 
\[\mu^k=\deg(\alpha^{d-k}\beta^k\cap [\widetilde{\P(V)}]).\]
\end{proposition}

Log-concavity then follows by applying the  Khovanskii-Teissier inequality \cite[Example 1.6.4]{Lazarsfeld}:

\begin{theorem} \label{t:kt}
Let $X$ be a complete irreducible $d$-dimensional variety, and let $\alpha,\beta$ be nef divisors on $X$.
Then  $\deg((\alpha^{d-k}\beta^{k})\cap [X])$
is a log-concave sequence.
\end{theorem}

\begin{proof} We give an outline of the proof of this inequality.  
By resolution of singularities, Chow's lemma and the projection formula, we may suppose that $X$ is smooth and projective.
It suffices to show
\[\deg((\alpha^{d-k}\beta^{k})\cap [X])^2\geq \deg((\alpha^{d-k+1}\beta^{k-1})\cap [X])\deg((\alpha^{d-k-1}\beta^{k+1})\cap [X]).\]
By Kleiman's criterion \cite{Kleiman} and continuity of intersection numbers, we can perturb $\alpha$ and $\beta$ in $A^1(X)$ so that they are ample classes.  By homogeneity of the desired inequality, we can replace $\alpha$ and $\beta$ by $m\alpha$ and $m\beta$ for $m\in\Z_{\geq 1}$ so that they are very ample.  By the Kleiman-Bertini theorem \cite{KleimanBertini}, the intersection $\alpha^{d-k-1}\beta^{k-1}\cap [X]$  is represented by an irreducible smooth surface $Y$.  Now, we need only prove
\[\deg(\alpha\beta\cap [Y])^2\geq \deg(\alpha^2\cap [Y])\deg(\beta^2\cap [Y]).\]
By the Hodge index theorem, the intersection product on $A^1(Y)$ has signature $(1,n-1)$.  The restriction of the intersection product to the subspace spanned by $\alpha$ and $\beta$ is indefinite.  The desired inequality is exactly the non-positivity of the determinant of the matrix of the intersection product.
\end{proof}

We now give the proof of Proposition \ref{p:intersection}:
\begin{proof}
Let $M$ be represented by the linear subspace $V\subset\k^{n+1}$, and
let $\widetilde{\P(V)}$ be the proper transform of $\P(V)$ in $X(\Delta_\U)$.  Let $\alpha,\beta$ be as in Subsection \ref{ss:defbergman}.  Now, we know
\[\Delta_M\cap [X(\Delta_\U)]=\widetilde{\P(V)}\]
by  Lemma \ref{l:assoccyc} and Lemma \ref{l:bergmanassociated}.
Consequently, 
\[\deg(\alpha^{d-k}\beta^k\cap [\widetilde{\P(V)}])=\deg(\alpha^{d-k}\beta^k \cup \Delta_M)=\mu^k\]
where the last equality follows from Corollary \ref{c:reddeg}.
\end{proof}

By viewing $\alpha$ and $\beta$ as hyperplane classes on $\P^n$, we immediately have the following corollary:

\begin{corollary}
The cycle class of $[(\pi_1\times\pi_2)(\widetilde{\P(V)})]$ in $\P^n\times\P^n$ is
\[\mu^0[\P^d\times\P^0]+\mu^1[\P^{d-1}\times\P^1]+\dots+\mu^r[\P^0\times\P^d].\]
\end{corollary}

Now, we outline Lenz's proof of Mason's conjecture in the representable case which relates the $f$-vector of a matroid to the characteristic polynomial of a well-chosen matroid.   We form the $f$-polynomial of $M$ by setting
\[f_M(q)=\sum_{i=0}^{d+1} f_iq^{d+1-i}.\]
The free co-extension of $M$ is the matroid
\[M\times e =(M^*+_E e)^*\]
for a new element $e$.  
Because $M$ is representable, after a possible extension of the field $\k$, so is $M\times e$.
Lenz proves the following formula for the characteristic polynomial of $M\times e$
\[(-1)^d\chi_{M\times e}(-q)=(1+q)f_M(q)\]
by considering various specializations of the rank-generating polynomial.  The log-concavity of $f_M$ then follows from the log-concavity of the the reduced characteristic polynomial of $M\times e$.

\subsection{Intersection theory computations} \label{ss:alphabeta}

Now, we will prove Proposition \ref{p:alphabeta}.  It will following from the following three lemmas:

\begin{lemma} \label{l:alpha} Let $M$ be a rank $d+1$ matroid on $E=\{0,\dots,n\}$.  Then,
\[\alpha\cup\Delta_M=\Delta_{M[1,d-1]}\]
in $A^*\big(X(\Delta_{\U})\big)$.
\end{lemma}

Note here that $\Delta_{M[1,d-1]}=\Delta_{\Trunc^{d}(M)}$.

\begin{lemma} \label{l:beta} 
Let $M$ be a rank $d+1$ matroid on $E=\{0,\dots,n\}$.  Then,
\[\beta\cup\Delta_{M[r_1,d]}=\Delta_{M[r_1+1,d]}\]
in $A^*\big(X(\Delta_{\U})\big)$.
\end{lemma}

To simplify the proofs of these lemmas, we homogenize our piecewise linear functions.  We begin with the lattice $\Z^{n+1}$ spanned by $e_0,e_1,\dots,e_n$.  Let $w_0,w_1,\dots,w_n$ be the induced coordinates on $\R^{n+1}=\Z^{n+1}\otimes\R$.  We will let our lattice $N$ by defined by $w_0=0$ in $\Z^{n+1}$.  Consider the quotient by the diagonal line 
\[p:\R^{n+1}\rightarrow \R^{n+1}/\R\cong N_\R.\]
Let $\Delta'$ be the fan in $\R^{n+1}$ whose cones are the inverse images of the cones of $\Delta_\U$ under $p$.  Write $\sigma'_{F_\bullet}=p^{-1}(\sigma_{F_\bullet})$.  By assigning it the same values on relevant cones, we may treat $\Delta_M$ as a Minkowski weight on $\Delta'$, denoted by $\Delta'_M$.
 We will consider the piecewise linear functions on $\Delta'$,
\begin{eqnarray*}
\alpha'&=&\min(w_0,w_1,\dots,w_n)\\
\beta'&=&\min(-w_0,-w_1,\dots,-w_n)
\end{eqnarray*}
Then $\alpha$ and $\beta$ are the restrictions of $\alpha'$ and $\beta'$ to $N_\R$, respectively.  Because the entire situation is invariant under translation by the diagonal vector $e_0+e_1+\dots+e_n$, we can do the intersection theory computation on $\R^{n+1}$ and then intersect with $N_\R$.

We first give the proof of Lemma \ref{l:alpha}.

\begin{proof}
The Minkowski weight $\alpha\cup\Delta'_M$ is supported on $d$-dimensional cones in $\Delta'_\U$.  They correspond to $(d-1)$-step flags of proper flats
\[
F_{\bullet}=\{\emptyset \subsetneq F_1\subsetneq F_2\subsetneq \dots\subsetneq F_{d-1}\subsetneq F_d=E\}.
\]
The cone $\sigma'_{F_\bullet}$ is contained in $\sigma'_{G_\bullet}$ if and only if the flag $G_{\bullet}$ is obtained from $F_{\bullet}$ by inserting a single flat.  Write this relation as $G_\bullet\gtrdot F_\bullet$.  This flat must be inserted between two flats $F_j\subset F_{j+1}$ where $r(F_{j+1})=r(F_j)+2$. There is a unique choice of $j$ where this happens.  Suppose $G_\bullet$ is obtained from inserting a proper flat $F$ between $F_j\subset F_{j+1}$.
 Let $u_{G_\bullet/F_\bullet}$ be an integer vector in $\sigma_{G_\bullet}$ that generates the image of $\sigma_{G_\bullet}$ in $N/N_{\sigma_{F_\bullet}}$. We may choose $u_{G_{\bullet}/F_{\bullet}}$ to be $e_{F}$. The value of $\alpha' \cup \Delta'_M$ on $\sigma'_{F_\bullet}$ is given by
\[
(\alpha' \cup \Delta'_M)(\sigma'_{F_\bullet})=-\sum_{G_\bullet\gtrdot F_\bullet} \alpha'_{G_\bullet}(u_{G_\bullet/F_\bullet})+\alpha'_{F_\bullet}\Big(\sum_{G_\bullet\gtrdot F_\bullet} u_{G_\bullet/F_\bullet}\Big)
\]
where $\alpha'_{G_{\bullet}}$ (respectively $\alpha'_{F_\bullet}$) is the linear function on $N_{\sigma_{G_\bullet}}$ (on $N_{\sigma_{F_\bullet}}$) which equals $\alpha'$ on $\sigma_{G_\bullet}$ (on $\sigma_{F_{\bullet}}$).

We now compute the right side.  Because $F\neq E$, 
\[\alpha'_{G_{\bullet}}(u_{G_\bullet/F_\bullet})=\alpha'(e_F)=0.\]
Let $f$ be the number of flats that can be inserted between $F_j$ and $F_{j+1}$.  Because every element of $F_{j+1}\setminus F_j$ is contained in exactly one such flat $F$ by Definition~\ref{d:flats} \eqref{iflat:3}, we have
\begin{eqnarray*}
\sum_{G_{\bullet} \gtrdot F_{\bullet}} u_{G_{\bullet}/F_{\bullet}}&=&e_{F_{j+1}}+(f-1)e_{F_j}.\\
\end{eqnarray*}
and so
\begin{eqnarray*}
\alpha'_{F_{\bullet}}\Big(\sum_{G_{\bullet} \gtrdot F_{\bullet}} u_{G_{\bullet}/F_{\bullet}}\Big)&=&
\begin{cases}1&\ \text{if}\ F_{j+1}=E\\ 
0&\ \text{otherwise}\end{cases}
\end{eqnarray*}
Consequently, we have
\[
(\alpha' \cup \Delta'_M)(\sigma'_{F_\bullet})=\begin{cases}1&\ \text{if}\ j=d-1\\
0&\ \text{if otherwise}.\end{cases}
\]
Therefore $\alpha \cup \Delta_M$ is non-zero on exactly the cones in the support of  $\Delta_{\Trunc^{d}(M)}$ where it takes the value $1$.
\end{proof}

Now, we prove Lemma \ref{l:beta}.

\begin{proof}
The set-up is as in the proof of the above lemma.  
For a flag of flats 
\[F_{\bullet}=\{\emptyset=F_0 \subsetneq F_1\subsetneq F_2\subsetneq \dots\subsetneq F_{d-1}\subsetneq F_d=E\},
\]
write $\nu_{F_\bullet}=|\mu(\emptyset,F_1)|.$  Let $F$ be a flat inserted between $F_j\subset F_{j+1}$ to obtain a flag of flats $G_\bullet$. 
Here, we have
\[
(\beta' \cup \Delta'_{M[r_1,d]})(\sigma'_{F_\bullet})=-\sum_{G_\bullet\gtrdot F_\bullet} \nu_{G_{\bullet}}\beta'_{G_{\bullet}}(u_{G_{\bullet}/F_{\bullet}})+\beta'_{F_{\bullet}}\Big(\sum_{G_{\bullet}\gtrdot F_{\bullet}} \nu_{G_\bullet}u_{G_{\bullet}/F_{\bullet}}\Big)
\]
Because $F\neq \emptyset$, 
\[\beta'_{G_{\bullet}}(u_{G_{\bullet}/F_{\bullet}})=\beta'(e_F)=-1.\]

Now we consider two cases: $F_j\neq\emptyset$ and $F_j=\emptyset$.
If $F_j\neq \emptyset$,
\[\sum_{G_\bullet\gtrdot F_\bullet} \nu_{G_{\bullet}}\beta'_{G_{\bullet}}(u_{G_{\bullet}/F_{\bullet}})=f|\mu(\emptyset,F_1)|.\]
and
\begin{eqnarray*}
\sum_{G_\bullet \gtrdot F_\bullet} \nu_{G_\bullet}u_{G_{\bullet}/F_{\bullet}}&=&|\mu(\emptyset,F_1)|(e_{F_{j+1}}+(f-1)e_{F_j}).
\end{eqnarray*}
Therefore, $\beta'(\sum_{G_\bullet \gtrdot F_\bullet} \nu_{G_\bullet}u_{G_{\bullet}/F_{\bullet}})=f
|\mu(\emptyset,F_1)|$ and $(\beta' \cup \Delta'_{M[r_1,d]})(\sigma'_{F_\bullet})=0$.

If $F_j=\emptyset$,
\[\sum_{G_\bullet \gtrdot F_\bullet} \nu_{G_\bullet}\beta'(u_{G_{\bullet}/F_{\bullet}})=-\sum_{F\lessdot F_1} |\mu(\emptyset,F)|\]
and
\[
\beta'_{F_{\bullet}}\Big(\sum_{G_{\bullet} \gtrdot F_{\bullet}} \nu_{G_\bullet} u_{G_{\bullet}/F_{\bullet}}\Big)=\beta'(\sum_{F\lessdot F_1} |\mu(\emptyset,F)|e_F)=-\sum_{a\in F\lessdot F_1}|\mu(\emptyset,F)|
\]
where some $a\in F_1$ chosen so to maximize the quantity on the right.
Then,
\[(\beta' \cup \Delta_{M[r_1,d]})(\sigma'_{F_\bullet})=\sum_{a\notin F \lessdot F_1}  |\mu(\emptyset,F)|=|\mu(\emptyset,F_1)|\]
where the last equality follows from Lemma~\ref{l:weisner}.
Therefore $\beta \cup \Delta_M$ is non-zero on exactly the cones in the support of $\Delta_{M[r_1+1,d]} $ where it takes the expected value.
\end{proof}

\begin{lemma}
We have the following equality of degrees:
\begin{enumerate}
\item $\deg(\alpha\cup \Delta_{M[r,r]})=\mu^{r-1}$
\item $\deg(\beta\cup \Delta_{M[r,r]}=\mu^r$
\end{enumerate}
where $\mu^k$ is the coefficient of the reduced characteristic polynomial.
\end{lemma}

\begin{proof}
Let $\gamma'$ be $\alpha'$ or $\beta'$ as above.  Because the only codimension $1$ cone in $\Delta_{M[r,r]}$ is the origin, we have the following formula for the degree:
\[
\deg(\gamma' \cup \Delta_{M[r,r]})=-\sum_{F\in L(M)_r} |\mu(\emptyset,F)|\gamma'(e_F)+\gamma'\Big(\sum_{F\in L(M)_r} |\mu(\emptyset,F)| e_F \Big)
\]
For $\gamma'=\alpha'$, this becomes
\[\alpha'\Big(\sum_{F\in L(M)_r} |\mu(\emptyset,F)| e_F \Big)=\sum_{a\in F\in L(M)_r} |\mu(\emptyset,F)|=\mu^{r-1}\]
for some $a\in E$ where the last equality follows from Lemma \ref{l:reducedcoeff}.
For $\gamma'=\beta'$, we have
\[\deg(\beta' \cup \Delta_{M[r,r]})=\sum_{F\in L(M)_r} |\mu(\emptyset,F)|-\sum_{a\in F\in L(M)_r} |\mu(\emptyset,F)|=\sum_{a\not\in F\in L(M)_r} |\mu(\emptyset,F)|=\mu^r\]

\end{proof}

\section{Future Directions}

One would like to prove the log-concavity of the characteristic polynomial in the non-representable case.  There are a couple of lines of attack that are being considered in future work by Huh individually and with the author.

The proof presented above requires representability to invoke the Khovanskii-Teissier inequality.  
The  proof of the Khovanskii-Teissier inequality reduces to the Hodge index theorem on a particular algebraic surface.  Recall that the log-concavity statement concerns only three consecutive coefficients of the characteristic polynomial at a time.  These three coefficients are intersection numbers on this surface.  One might try to prove a combinatorial analogue of the Hodge index theorem on the combinatorial analogue of this surface which is the truncated Bergman fan $\Delta_{M[r,r+1]}$ by showing that a {\em combinatorial intersection matrix} has a single positive eigenvalue.  This combinatorial intersection matrix is called the Tropical Laplacian and will be investigated in future work with Huh.  Unfortunately, the algebraic geometric arguments used in proofs of the Hodge index theorem do not translate into combinatorics, and the hypotheses for a combinatorial Hodge index theorem are unclear.  Still, the conclusion of the combinatorial Hodge index theorem for $\Delta_{M[r,r+1]}$ has been verified experimentally for all matroids on up to nine elements by Theo Belaire \cite{BelaireKatz} using the matroid database of Mayhew and Royle \cite{MayhewRoyle}.

Another approach to the Rota-Heron-Welsh conjecture is to relax the definition of representability.  In our proof above, we needed the Chow cohomology class $\Delta_M$ to be Poincar\'{e}-dual to an irreducible subvariety of $X(\Delta_\U)$ in order to apply the Khovanskii-Teissier inequality.  However, it is sufficient that some positive integer multiple of $\Delta_M$ be Poincar\'{e}-dual to an irreducible subvariety.  It is a part of a general philosophy of Huh \cite{Huhthesis} that it is very difficult to understand which homology classes are representable while it is significantly easier to understand their cone of positive multiples.  

Another question of interest is to understand the relation between the work of Huh-Katz and Fink-Speyer.  What does positivity (in the sense of \cite{Lazarsfeld}) say about $K$-theory.  How does positivity restrict the Tutte polynomial?  Are there other specializations of the Tutte polynomial that obey log-concavity?

Finally, the author would like to promote the importance of a combinatorial study of Minkowski weights on the permutohedral variety.  They are very slight enlargements of the notion of matroids.  One can certainly introduce notions of deletion and contraction and therefore minors.  Are there interesting structure theorems?  To this author, $(r_1,r_2)$-truncation is an attractive and useful operation and should be situated in a general combinatorial theory.

\bibliography{math}
\bibliographystyle{amsalpha}

\end{document}